\documentclass[reqno, 12pt]{amsart}

\usepackage{tikz-cd}

\usepackage{amsfonts, amsthm, amsmath, amssymb}
\usepackage{hyperref}
\hypersetup{colorlinks=false}
\usepackage[all]{xy}

\usepackage[margin=1in]{geometry}

\usepackage{helvet}

\RequirePackage{mathrsfs} \let\mathcal\mathscr

\usepackage{hyperref}
\usepackage{dsfont}
\usepackage{upgreek}
\usepackage{mathabx,yfonts}
\usepackage{enumerate}

\numberwithin{equation}{section}

\newtheorem{theorem}{Theorem}[section] 
\newtheorem{lemma}[theorem]{Lemma}
\newtheorem{proposition}[theorem]{Proposition}
\newtheorem{corollary}[theorem]{Corollary}

\newtheorem*{goal}{Goal}

\theoremstyle{definition}
\newtheorem*{acknowledgements}{Acknowledgements}
\newtheorem{remark}[theorem]{Remark}
\newtheorem{definition}[theorem]{Definition}
\newtheorem{example}[theorem]{Example}

\renewcommand{\phi}{\varphi}

\renewcommand{\leq}{\leqslant}

\renewcommand{\geq}{\geqslant}

\renewcommand{\bar}{\overline}

\renewcommand{\c}{\mathbf{c}}

\DeclareMathOperator{\disc}{D}

\DeclareMathOperator{\Hom}{Hom} 
 
\DeclareMathOperator{\epi}{Epi} 

\let\emptyset\varnothing

\DeclareSymbolFont{bbold}{U}{bbold}{m}{n}
\DeclareSymbolFontAlphabet{\mathbbold}{bbold}

\renewcommand{\c}{\mathcal}
\renewcommand{\epsilon}{\varepsilon}

\renewcommand{\leq}{\leqslant}
\renewcommand{\geq}{\geqslant}

\title
[Jump sets in local fields]
{
Jump sets in local fields
}

\author{C. Pagano}
\address{Max Planck Institute for mathematics\\ 
Vivatgasse 7 \\ 
Bonn \\ 
5311\\ 
Germany} 
\email{carlein90@gmail.com}

\date{\today}

\begin{document}

\begin{abstract}
We show how to use the combinatorial notion of jump sets to parametrize the possible structures of the group of principal units of local fields, viewed as filtered modules. We establish a natural bijection between the set of jump sets and the orbit space of a $p$-adic group of filtered automorphisms acting on a free filtered module. This, together with a Markov process on Eisenstein polynomials, culminates into a \emph{mass-formula} for unit filtrations. As a bonus the proof leads in many cases to explicit invariants of Eisenstein polynomials, yielding a link between the filtered structure of the unit group and ramification theory. Finally, with the basic theory of filtered modules developed here, we recover, with a more conceptual proof, a classification, due to Miki, of the possible sets of upper jumps of a wild character: these are all jump sets, with a set of exceptions explicitly prescribed by the jump set of the local field and the size of its residue field.
\end{abstract}

\maketitle

\setcounter{tocdepth}{1}
\setcounter{page}{1}
\tableofcontents

\section{Introduction}
In this paper we introduce \emph{jump sets}, elementary combinatorial objects, and use them to establish several fundamental results concerning two natural filtrations in the theory of local fields. These are the unit filtration and the ramification filtration. We subdivide our main results into three themes and introduce each of the themes with a basic question. We use the answer to each question as a starting point to explain our main results.    

\subsection{Three questions} \label{Three questions}
 
\subsubsection{Principal units} \label{principal units}
Let $p$ be a prime number. A non-archimedean local field is a field $K$, equipped with a non-archimedean absolute value $|\cdot|$, such that $K$ is a non-discrete locally compact space with respect to the topology induced by $|\cdot|$. Write $O:=\{x \in K:|x| \leq 1 \}$ for the ring of integers and $m:=\{x \in K:|x|<1 \}$ for its unique maximal ideal. We assume that $p$ is the residue characteristic of $K$, i.e. the characteristic of the finite field $O/m$. Denote by $f_K$ the positive integer satisfying $p^{f_K}=\#O/m$. Recall that $O$ is a discrete valuation ring, and denote by $\text{v}_K:K^{*} \to \mathbb{Z}$ the valuation that maps any generator of the ideal $m$ to $1$. 

The inclusions $K^{*} \supseteq O^{*} \supseteq U_1(K)=1+m=\{\text{principal units}\}$ \ split in the category of topological groups. So, as topological groups, we have $K^{*} \simeq_{\text{top.gr.}} \mathbb{Z} \times O^{*}$, $O^{*}=(O/m)^{*} \times U_1(K)$, where $\mathbb{Z}$ is taken with the discrete topology. This paper focuses on $U_1(K)$. The profinite group $U_1(K)$ is a pro-$p$ group, thus, being abelian, it has a natural structure of $\mathbb{Z}_p$-module. As a topological $\mathbb{Z}_p$-module $U_1(K)$ is very well understood. If $\text{char}(K)=0$ then $U_1(K) \simeq \mathbb{Z}_p^{[K:\mathbb{Q}_p]} \times \mu_{p^{\infty}}(K)$, while if $\text{char}(K)=p$ then $U_1(K) \simeq \mathbb{Z}_p^{\omega}$. Here $\omega$ denotes the first infinite ordinal number and $\mu_{p^{\infty}}(K)$ denotes the $p$-part of the group of roots of unity of $K$. In both cases the isomorphism is meant in the category of topological $\mathbb{Z}_p$-modules. For a reference see \cite[Chapter 1, Section 6]{Fesenko--Vostokov}

The $\mathbb{Z}_p$-module $U_1(K)$ comes naturally with some additional structure, namely the filtration $U_1(K) \supseteq U_2(K) \supseteq \ldots \supseteq U_i(K) \supseteq \ldots $, where $ U_i(K)=1+m^i$. In order to take into account this additional structure we make the following definition. A \emph{filtered $\mathbb{Z}_p$-module} is a sequence of $\mathbb{Z}_p$-modules, $M_1 \supseteq M_2 \supseteq \ldots \supseteq M_i\supseteq \ldots $, with $\bigcap_{i \in \mathbb{Z}_{\geq 1}}M_i=\{0\}$. We will use the symbol $M_{\bullet}$ to denote a filtered $\mathbb{Z}_p$-module. A morphism of filtered $\mathbb{Z}_p$-modules is a morphism of $\mathbb{Z}_p$-modules $\phi:M_1 \to N_1$ such that $\phi(M_i) \subseteq N_i$ for each positive integer $i$. A filtered module can be also described in terms of its \emph{weight} map $w:M_1 \to \mathbb{Z}_{\geq 1} \cup \{\infty\}$ attaching to each $x$ the sup of the set of integers $i$ such that $x \in M_i$. 
\\

$\mathbf{Question \ (1)}$ What does $U_1(K)$ look like as a filtered $\mathbb{Z}_p$-module?\\
\\
In other words, we ask what is, as a function of $K$, the isomorphism class of $U_1(K)$ in the category of filtered $\mathbb{Z}_p$-modules. We will sometimes use the symbol $U_{\bullet}(K)$ to stress the presence of the additional structure present in $U_1(K)$, coming from the filtration. Denote by $G_K$ the absolute Galois group of $K$. Thanks to local class field theory, the above question is essentially asking to describe $G_K^{\text{ab}}$ as a filtered group, where the filtration is given by the upper numbering on $G_K^{\text{ab}}$. Equipping any quotient of $G_K$ with the upper numbering filtration and studying it in the category of filtered groups is a natural thing to do. Indeed it is a fact that the local field $K$ can be uniquely determined from the filtered group $G_K$, see \cite{Mochizuki}. 
\subsubsection{Galois sets} \label{Question 2}
Fix $K^{\text{sep}}$ a separable closure of $K$. Denote by $G_K:=\text{Gal}(K^{\text{sep}}/K)$ the absolute Galois group. Denote by $|\cdot|$ the unique extension of $|\cdot|$ to $K^{\text{sep}}$. Take $L/K$ finite separable. Thus $L$ naturally comes with a \emph{Galois set}: $\Gamma_L=\{K\text{-embeddings} \ L \to K^{\text{sep}}\}$. 
Recall by Galois theory that this is a transitive $G_K$-set with $|\Gamma_L|=[L:K]$. This holds for any field $K$. But, if $K$ is a local field, there is an additional piece of structure, namely a $G_K$-invariant metric on $\Gamma_L$, defined as follows:
$d(\sigma,\tau)=\text{max}_{x \in O_L}|\sigma(x)-\tau(x)|$ \  $(\sigma,\tau \in \Gamma_L)$. Here $O_L$ denotes the ring of integers of $L$. Observe that the maximum is attained since $O_L$ is compact and the function in consideration is continuous. If $L/K$ is \text{unramified} then the metric space $\Gamma_L$ is a simple one: $d(\sigma,\tau)=1$ whenever $\sigma \neq \tau$. Since every finite separable extension of local fields splits canonically as an unramified one and a totally ramified one, we go to the other extreme of the spectrum and consider $L/K$ totally ramified: in other words we put $L=K(\pi)$, with $g(\pi)=0$, where $g \in K[x]$ is \emph{Eisenstein}. We can now phrase the second question. 
\\

$\mathbf{Question \ (2)}$
Which invariants does the metric space impose on the coefficients of $g$?\\
\\
As we shall see, the answer to our second question comes often with a surprising link to the answer to our first question. 

\subsubsection{Jumps of characters}
 A \emph{character} of $U_1(K)$ is a continuous group homomorphism \\
$\chi:U_1(K) \to \mathbb{Q}_p/\mathbb{Z}_p \simeq \mu_{p^{\infty}}(\mathbb{C})$. Define $J_{\chi}=\{i \in \mathbb{Z}_{\geq 1}:\chi(U_i(K)) \neq \chi(U_{i+1}(K))\}=\{\text{jumps for $\chi$}\}$. Since $U_1(K)$ is a profinite group, a character $\chi$ has always finite image. Moreover it is easy to check that at each jump the size of the image gets divided exactly by $p$. So one has that $\text{order}(\chi)=p^{|J_{\chi}|}<\infty$.  In particular $J_{\chi}$ is always a finite subset of $\mathbb{Z}_{\geq 1}$. We can now phrase our third question.
\\

$\mathbf{Question \  (3)}$
Given a local field $K$, which subsets of $\mathbb{Z}_{\geq 1}$ occur as $J_{\chi}$ for a character of $U_1(K)$? \\
\\
Thanks to local class field theory this question is essentially asking to determine which sets $A \subseteq \mathbb{Z}_{\geq 1}$ occur as the set of jumps in the upper filtration of $\text{Gal}(L/K)$, for some $L$, a finite cyclic totally ramified extension of $K$, with $[L:K]$ a power of $p$. This connection is articulated in Section \ref{wild extension}. 
\subsection{Shifts and jump sets} \label{main results}
The goal of this subsection is to explain the notion of a \emph{jump set}. Jump sets are defined using  \emph{shifts}. A \emph{shift} is a strictly increasing function $\rho:\mathbb{Z}_{\geq 1} \to \mathbb{Z}_{\geq 1}$, with $\rho(1)>1$. If $T_{\rho}=\mathbb{Z}_{\geq 1}-\rho(\mathbb{Z}_{\geq 1})$ is finite, put $e^{*}=\text{max}(T_{\rho})+1$. The example of shift relevant for local fields is the following:
\[
\rho_{e,p}(i)=\text{min}\{i+e,pi\} \ \text{for} \ p \ \text{prime}, e \in \mathbb{Z}_{>0} \cup \{\infty\}.
\]
In this example one has that if $e \neq \infty$, then $e^{*}=\lceil \frac{pe}{p-1} \rceil$. The case $e \neq \infty$ will be used for local fields of characteristic $0$, and the case $e=\infty$ will be used for local fields of characteristic $p$.

The following property explains how this shift can be used to express how $p$-powering in $U_1$ changes the weights in the filtration.

\emph{Crucial property}: If $K$ is local field, $e=\text{v}_K(p)$, then
$$
U_i^p \subset U_{\rho(i)} \ \text{for} \ \rho=\rho_{e,p}. 
$$
This follows at once inspecting valuations in the binomial expansion $(1+x)^p=1+px+ \ldots +x^p$. 
For a local field $K$ we denote by $\rho_K$ the shift $\rho_{e,p}$. 

We can now provide the notion of a jump set for a shift $\rho$ and respectively, in case $T_{\rho}$ is finite, of an extended jump set for $\rho$.
A \emph{jump set} for $\rho$ (resp.\ an \emph{extended jump set} for $\rho$) is a finite subset $A \subseteq \mathbb{Z}_{\geq 1}$, satisfying the following two conditions: \\
\\
$(C.1)$ if $a,b \in A$, and $a<b$ then $\rho(a) \leq b$,  \\
$(C.2)$ one has that $A - \rho(A) \subseteq T_{\rho}$ \ (resp.\ $A - \rho(A) \subseteq T_{\rho}^{*}=T_{\rho} \cup \{e^{*}\}$). \\

Write $\text{Jump}_{\rho}=\{\text{jump sets for $\rho$}\}$ \ (resp.\ $\text{Jump}_{\rho}^{*}=\{\text{extended jump sets for $\rho$}\}$). The jump set $A$ can be reconstructed from the following data. \\
\\
$(a)$ $I_{A}=A-\rho(A)$. \\
$(b)$ The function $\beta_A:A-\rho(A) \to \mathbb{Z}_{\geq 1}$, $i \to |[i,\infty) \cap A|$. \\
\\
The pair $(I_A,\beta_A)$ satisfies the following three conditions.\\
\\
$(C.1)'$ One has that $I_A \subseteq T_{\rho}$ (resp.\ $I_A \subseteq T_{\rho}^{*}$), \newline
$(C.2)'$ the map $\beta_A$ is a strictly decreasing map $\beta:I_A \to \mathbb{Z}_{\geq 1}$, \newline
$(C.3)'$ the map $i \mapsto \rho^{\beta(i)}(i)$  from $I_A$ to $\mathbb{Z}_{\geq 1}$ is strictly increasing. 
\\

Conversely, given any pair $(I,\beta)$ satisfying properties $(C.1)', (C.2)'$ and $(C.3)'$, we can attach to it a jump set for $\rho$ denoted by $A_{(I,\beta)}$ (resp.\ an extended jump set for $\rho$). The assignments  $A \mapsto (I_A,\beta_A)$ and $(I,\beta) \mapsto A_{(I,\beta)}$ are inverses to each other. Namely we have
$$A_{(I_A,\beta_A)}=A,$$
and
$$(I_{A_{(I,\beta)}},\beta_{A_{(I,\beta)}})=(I,\beta).
$$
We will refer also to the pair $(I,\beta)$ as a jump set. 
\subsubsection{Answer to question $(1)$} \label{Answer (1)}
We will answer question (1) exploiting the following analogy with usual $\mathbb{Z}_p$-modules. We denote by $\mu_p(K):=\{\alpha \in K: \alpha^p=1\}$. It is not difficult to show that $\mu_p(K)=\{1\}$ if and only if
$$U_1(K) \simeq_{\mathbb{Z}_p\text{-mod}}\prod_{i \in T_{\rho_K}} \mathbb{Z}_p^{f_K}.
$$
Suppose that $\mu_p(K) \neq \{1\}$. Then $U_1(K)$ has a presentation:
$$ 0\to \mathbb{Z}_p \to \mathbb{Z}_p^{[K:\mathbb{Q}_p]+1} \to U_1(K) \to 0.
$$
Denote by $v_0$ the image of $1$ in the inclusion of $\mathbb{Z}_p$ into $\mathbb{Z}_p^{[K:\mathbb{Q}_p]+1}$. 
One can obtain a different presentation using the natural action of $\text{Aut}_{\mathbb{Z}_p}(\mathbb{Z}_p^{[K:\mathbb{Q}_p]+1})$ on $\text{Epi}_{\mathbb{Z}_p}(\mathbb{Z}_p^{[K:\mathbb{Q}_p]+1},U_1(K))$, which denotes the set of surjective morphisms of $\mathbb{Z}_p$-modules from $\mathbb{Z}_p^{[K:\mathbb{Q}_p]+1}$ to $U_1(K)$. In this way all presentations are obtained. 
That is, $\text{Aut}_{\mathbb{Z}_p}(\mathbb{Z}_p^{[K:\mathbb{Q}_p]+1})$  acts transitively on $\text{Epi}_{\mathbb{Z}_p}(\mathbb{Z}_p^{[K:\mathbb{Q}_p]+1},U_1(K))$. Thus knowing $U_1(K)$ as a $\mathbb{Z}_p$-module is tantamount to knowing the orbit of the vector $v_0$ under the action of $\text{Aut}_{\mathbb{Z}_p}(\mathbb{Z}_p^{[K:\mathbb{Q}_p]+1})$. But recall that for all $v_1, v_2 \in \mathbb{Z}_p^{[K:\mathbb{Q}_p]+1}$ one has that
$$
 v_1 \sim_{\text{Aut}_{\mathbb{Z}_p}} v_2 \leftrightarrow \text{ord}(v_1)=\text{ord}(v_2).
  $$
Here $\text{ord}$ of a vector $v \in \mathbb{Z}_p^{[K:\mathbb{Q}_p]+1}$ denotes the minimum of $\text{v}_{\mathbb{Q}_p}(a)$ as $a$ varies among the coordinates of $v$ with respect to the standard basis of $\mathbb{Z}_p^{[K:\mathbb{Q}_p]+1}$. Therefore we have that $$
\{v: \mathbb{Z}_p^{[K:\mathbb{Q}_p]+1}/\mathbb{Z}_pv \simeq U_1(K) \}=\{v:|\mu_{p^{\infty}}(K)|=p^{\text{ord}(v)}
\}.
$$
We will see that in the finer category of filtered $\mathbb{Z}_p$-modules the story is very similar. To reach an analogous picture we need to introduce the analogues of the actors appearing above. Namely we need a notion of a ``\emph{free-filtered-module}".

As we shall explain in section \ref{def direct sums}, with filtered modules one can do the usual operations of direct sums, direct product, and when the modules are finitely generated of taking quotients. Having this in mind, one defines what may be thought of as the building blocks for ``free-filtered-modules", namely the analogue of rank $1$ modules over $\mathbb{Z}_p$ (but now there will be many different rank $1$ filtered modules), as follows. Let $\rho$ be a shift, and let $i$ be a positive integer.
\begin{definition}
The $i$-th standard filtered module, $S_i$, for $\rho$, is given by setting $S_i=\mathbb{Z}_p$, with weight map
$$w(x)=\rho^{\text{ord}_p(x)}(i).$$
\end{definition}
The analogues of a ``free-filtered-module" used to describe $U_1(K)$ will be
$$M_{\rho}=\prod_{i \in T_{\rho}}S_i,$$ \ \ $$M_{\rho}^{*}=\prod_{i \in T_{\rho}^{*}}S_i.$$
We have the following theorem.
\begin{theorem}\label{the free guys}
Let $K$ be a local field, with $|O/m|=p^{f_K}$. Then
$U_1 \simeq M_{\rho_K}^{f_K}$ as filtered $\mathbb{Z}_p$-modules if and only if $\mu_p(K)=\{1\}$.
\end{theorem}
So we are left with the case $\mu_p(K) \neq \{1\}$. In particular we have that $\text{char}(K)=0$. We proceed in analogy with the case of $\mathbb{Z}_p$-modules described above. 

To describe $U_{\bullet}$ as a \emph{filtered} $\mathbb{Z}_p$-module one \emph{constructs} a filtered presentation: 
$$M_{\rho_K}^{f_K-1} \oplus M_{\rho_K}^{*} \twoheadrightarrow U_{\bullet}(K).
$$
Just as with $\mathbb{Z}_p$-modules, one can obtain a different presentation using the natural action of $\text{Aut}_{\text{filt}}(M_{\rho_K}^{f_K-1} \oplus M_{\rho_K}^{*})$ on $\text{Epi}_{\text{filt}}(M_{\rho_K}^{f_K-1} \oplus M_{\rho_K}^{*},U_{\bullet}(K))$. As established in Proposition \ref{transitive action} we obtain a statement in perfect analogy with the case of $\mathbb{Z}_p$-modules explained above.  Namely we have the following crucial proposition. 
\begin{proposition} \label{orbit for intro}
Let $K$ be a local field with $\mu_p(K) \neq \{1\}$. Then the action of $\emph{Aut}_{\emph{filt}}(M_{\rho_K}^{f_K-1} \oplus M_{\rho_K}^{*})$ upon the set $\emph{Epi}_{\emph{filt}}(M_{\rho_K}^{f_K-1} \oplus M_{\rho_K}^{*},U_{\bullet}(K))$ is transitive. 
\end{proposition}

For a local field $K$ as in Proposition \ref{orbit for intro} knowing the filtered module $U_{\bullet}(K)$ is tantamount to knowing the set of vectors $v \in M_{\rho_K}^{f_K-1} \oplus M_{\rho_K}^{*}$ such that 
$$(M_{\rho_K}^{f_K-1} \oplus M_{\rho_K}^{*})/\mathbb{Z}_pv \simeq_{\text{filt}} U_{\bullet}(K).
$$ 
Thanks to Proposition \ref{orbit for intro} the set of such vectors $v$ consists of a single orbit under the action of the group $\text{Aut}_{\text{filt}}(M_{\rho_K}^{f_K-1} \oplus M_{\rho_K}^{*})$. Thus we are led to study the orbits of $\text{Aut}_{\text{filt}}(M_{\rho_K}^{f_K-1} \oplus M_{\rho_K}^{*})$ acting on $M_{\rho_K}^{f_K-1} \oplus M_{\rho_K}^{*}$, just as we did above in the case of $\mathbb{Z}_p$-modules. In particular we are led to find the filtered analogue of the function $\text{ord}$. It is in this context that jump sets come into play. For two vectors $v_1,v_2 \in M_{\rho}^{f-1} \oplus M_{\rho}^{*}$ we will use the notation
$$v_1 \sim_{\text{Aut}_{\text{filt}}} v_2
$$
to say that $v_1$ and $v_2$ are in the same orbit under the action of $\text{Aut}_{\text{filt}}(M_{\rho}^{f-1} \oplus M_{\rho}^{*})$. Observe that if $\phi \in \text{Epi}_{\text{filt}}(M_{\rho_K}^{f_K-1} \oplus M_{\rho_K}^{*},U_{\bullet}(K))$, then in particular $\text{ker}(\phi) \subseteq p \cdot (M_{\rho_K}^{f_K-1} \oplus M_{\rho_K}^{*})$. Therefore we proceed to describe only orbits of $\text{Aut}_{\text{filt}}(M_{\rho}^{f-1} \oplus M_{\rho}^{*})$ acting upon $p \cdot (M_{\rho}^{f-1} \oplus M_{\rho}^{*})$. However there is no loss of generality in doing so. Indeed it is clear that given $v_1,v_2$ in $M_{\rho}^{f-1} \oplus M_{\rho}^{*}$ one has that $v_1 \sim_{\text{Aut}_{\text{filt}}} v_2$ if and only if $p \cdot v_1 \sim_{\text{Aut}_{\text{filt}}}p \cdot v_2$.  We attach to each extended jump set $(I,\beta)$ a vector in $p\cdot(M_{\rho}^{f-1} \oplus M_{\rho}^{*})$ defined as follows:
$$
v_{(I,\beta)}=(x_j)_{j \in T_{\rho}^{*}} \in p \cdot M_{\rho}^{*}=\prod_{j \in T_{\rho}^{*}}p \cdot S_j 
$$
$$\text{by} \ x_j=0  \ \text{if} \ j \not\in I, \ x_j=p^{\beta(j)} \ \text{if} \ j \in I.
$$

\begin{theorem}\label{j.s.param.orbit}(\emph{Jump sets parametrize orbits}) Let $\rho$ be any shift with $\#T_{\rho}<\infty$ and  $f$ be a positive integer. Then there exists a unique map 
\[\emph{\text{filt-ord}}:p\cdot(M_{\rho}^{f-1} \oplus M_{\rho}^*) \to \emph{Jump}_{\rho}^{*}
\]
having the following two properties. \\
$(1)$ For all $v_1,v_2 \in p\cdot(M_{\rho}^{f-1} \oplus M_{\rho}^*) $ one has
$$
 v_1 \sim_{{\emph{\text{Aut}}}_{\emph{\text{filt}}}} v_2 \leftrightarrow \emph{\text{filt-ord}}(v_1)=\emph{\text{filt-ord}}(v_2).
  $$
$(2)$ For each $(I,\beta) \in \emph{Jump}_{\rho}^{*}$, we have that $$ \emph{\text{filt-ord}}(v_{(I,\beta)})=(I,\beta).
$$
\end{theorem}
In fact the proof of Theorem \ref{j.s.param.orbit}, as given in Section \ref{filtered modules}, provides us with an effective way to \emph{compute} the map $\text{filt-ord}$.
This goes as follows. Let $v$ be in $p\cdot(M_{\rho}^{f-1} \oplus M_{\rho}^*)$. Firstly define the following subset of $\mathbb{Z}_{\geq 1}^2$
$$S_v:=\{(i,\text{ord}(v_i)) \}_{i \in T_{\rho}^{*}:v_i \neq 0},
$$
where $v_i$ is the projection of $v$ on the factor $S_{i}^f$ if $i<e_{\rho}^{*}$ and on $S_{e_{\rho}^{*}}$ in case $i=e_{\rho}^{*}$.
Next, for any shift $\rho$ consider the following partial order $\leq_{\rho}$ defined on $\mathbb{Z}_{\geq 1}^2$.  We let $(a_1,b_1) \leq_{\rho} (a_2,b_2)$ if and only if 
$$b_2 \geq b_1 \ \text{and} \ \rho^{b_2}(a_2) \geq \rho^{b_1}(a_1).
$$
Finally define $S_v^{-}$ to be the set of minimal points of $S_v$ with respect to $\leq_{\rho}$. One can easily show that there is a unique extended jump set $(I_v,\beta_v) \in \text{Jump}_{\rho}^{*}$ such that
$$ S_{v}^{-}=\text{Graph}(\beta_v).
$$
It is shown in Section \ref{filtered modules} that $\text{filt-ord}(v)=(I_v,\beta_v)$. This phenomenon of a jump set arising as the set of minimal or maximal elements of some finite subset of $\mathbb{Z}_{\geq 1}^{2}$ is a leitmotif of this paper. Another instance of this phenomenon will emerge at the end of this sub-section in Theorem \ref{a surprising relation}, in the context of Eisenstein polynomials. We mention that this way of computing $\text{filt-ord}$ is used in \cite{de Boer--Pagano} where, among other things, algorithmic problems of this subject are explored.

From Theorem \ref{j.s.param.orbit} one concludes the following. 
\begin{theorem}\label{the quasi-free guys}
Let $K$ be a local field, with $\mu_p(K) \neq \{1\}$ and $|O/m|=p^{f_K}$. Then there is a \emph{unique} $(I_K,\beta_K) \in \emph{Jump}_{\rho_K}^{*}$ such that
 $$U_1(K) \simeq M_{\rho_K}^{f_K-1} \oplus (M_{\rho_K}^{*}/\mathbb{Z}_pv_{(I_K,\beta_K)})$$ 
 as filtered $\mathbb{Z}_p$-modules.
 \end{theorem}
So when $\mu_p(K) \neq \{1\}$, knowing $U_1(K)$ as a filtered module is tantamount to knowing the extended $\rho_K$-jump set $(I_K,\beta_K)$.

The next theorem tells us, for given $e,f$, which orbits of the action of $\text{Aut}_{\text{filt}}(M_{\rho_{e,p}}^{f-1} \oplus M_{\rho_{e,p}}^{*})$ on $M_{\rho_{e,p}}^{f-1} \oplus M_{\rho_{e,p}}^{*}$  are realized by a local field $K$ with $\mu_p(K) \neq \{1\}$, $e_K=e$ and $f_K=f$. In other words, together with Theorem \ref{the free guys} this provides a complete classification of the filtered $\mathbb{Z}_p$-modules $M_{\bullet}$ such that
$$U_{\bullet}(K) \simeq_{\text{filt}} M_{\bullet},
$$
for some local field $K$, therefore answering Question $(1)$. 
 
\begin{theorem} \label{realizable j.s. are realized}
Let $p$ be a prime number, let $e,f \in \mathbb{Z}_{>0}$, and let $(I,\beta)$ be an extended $\rho_{e,p}$-jump set. Then the following are equivalent. \\ 
(1) There exists a local field $K$ with residue characteristic $p$ and
$$\mu_p(K) \neq \{1\}, \ f_K=f, \ e=\emph{v}_K(p), \ (I_K,\beta_K)=(I,\beta).$$
(2) We have that $p-1|e, I \neq \emptyset$ and
$$\rho_{e,p}^{\beta(\emph{min}(I))}(\emph{min}(I))=\frac{pe}{p-1} \ (=e^{*}).$$
\end{theorem}
For a shift $\rho$ such that $T_{\rho}$ is finite, the extended jump sets $(I,\beta) \in \text{Jump}_{\rho}^{*}$ such that $I \neq \emptyset$ and $\rho^{\beta(\text{min}(I))}(\text{min}(I))=e^{*}$ are said to be \emph{admissible}. The implication $(2) \to (1)$, in the above theorem, is proved in Section \ref{U1 as filtered module} in Theorem \ref{admissible j.s. occur}. The implication $(1) \to (2)$ follows from Proposition \ref{rho for loc fields} and Theorem \ref{classification of quasi free} combined.

Our next main result provides a quantitative strengthening of Theorem \ref{realizable j.s. are realized}. Once we fix $e \in (p-1)\mathbb{Z}_{\geq 1}$ and a positive integer $f$, then, thanks to Theorem \ref{realizable j.s. are realized}, we know precisely which $(I,\beta) \in \text{Jump}_{\rho_{e,p}}^{*}$ occur as $(I_K,\beta_K)$ for some local field $K$ with $\mu_p(K) \neq \{1\}, e_K=e, f_K=f$. But Theorem \ref{realizable j.s. are realized} doesn't tell us ``\emph{how often}" each $(I,\beta)$ occurs. To make this point precise we should firstly agree in which manner we \emph{weight} local fields. A very natural way to do this is provided by Serre's Mass formula \cite{Serre m.f.}. We briefly recall how this works.

Let $E$ be a local field. Write $q=|O_E/m_E|$. Let $e$ be a positive integer. Let $S(e,E)$ be the set of isomorphism classes of separable totally ramified degree $e$ extensions $K/E$.
To $K\in S(e,E)$ one gives mass $\mu_{e,E}(K):=\frac{1}{q^{c(K/E)}|\text{Aut}_{E}(K)|}$, where $c(K/E)=\text{v}_K(\delta_{K/E})-e+1$, and $\delta_{K/E}$ denotes the different of the extension $K/E$. 
Serre's Mass formula \cite{Serre m.f.} states that $\mu_{e,E}$ is a probability measure on $S(e,E)$, i.e. 
$$ \sum_{K \in S(e,E)}\mu_{e,E}(K)=1.$$

Now we can make the ``\emph{how often}" written above precise. Namely given $e \in (p-1)\mathbb{Z}_{\geq 1}, f \in \mathbb{Z}_{\geq 1}$ and $(I,\beta) \in \text{Jump}_{\rho}^{*}$, write $E_f:=\mathbb{Q}_{p^f}(\zeta_p)$. Here $\mathbb{Q}_{p^f}$ denotes the degree $f$ unramified extension of $\mathbb{Q}_p$. We can ask to evaluate
$$\sum_{K\in S(\frac{e}{p-1},E_f): (I_K,\beta_K)=(I,\beta)}\mu_{\frac{e}{p-1},E_f}(K),
$$
in words we are asking to evaluate the probability that a random $K$, totally ramified degree $\frac{e}{p-1}$ extension of $E_f$, has $(I_K,\beta_K)=(I,\beta)$. 

Observe that, thanks to Proposition \ref{orbit for intro} and Theorems \ref{j.s.param.orbit} and \ref{the quasi-free guys} combined, we know that for $K \in S(\frac{e}{p-1},E_f)$ the set of vectors $O:=\{v \in M_{\rho_K}^{f_K-1} \oplus M_{\rho_K}^{*}: U_{\bullet}(K) \simeq_{\text{filt}} (M_{\rho_K}^{f_K-1} \oplus M_{\rho_K}^{*})/\mathbb{Z}_pv \}$ is precisely equal to the orbit of the vector $v_{(I_K,\beta_K)}$ under $\text{Aut}_\text{filt}(M_{\rho_K}^{f_K-1} \oplus M_{\rho_K}^{*})$. Moreover $M_{\rho_K}^{f_K-1} \oplus M_{\rho_K}^{*}$ viewed as a topological group is compact, and hence has a Haar measure. It is then natural to think that, for a given admissible extended $\rho_{e,p}$-jump set $(I,\beta)$, a randomly chosen totally ramified degree $\frac{e}{p-1}$ extension $K$ of $E_f$, satisfies
$$(I_K,\beta_K)=(I,\beta)
$$ with probability \emph{proportional} to the Haar measure of the orbit of $v_{(I,\beta)}$. Our next theorem shows that this turns out to be exactly right. 

 For $(I,\beta) \in \text{Jump}_{\rho_{e,p}}^{*}$, with $I \neq \emptyset$, it is easy to see that the set
$\text{filt-ord}^{-1}((I,\beta))$ is an open subset of $M_{\rho_{e,p}}^{f-1} \oplus M_{\rho_{e,p}}^{*}$. Normalize $\mu_{\text{Haar}}$, imposing that
$$\mu_{\text{Haar}}(\bigcup\limits_{(I,\beta) \ \text{admissible}}\text{filt-ord}^{-1}(I,\beta))=1.$$
In other words, choose the unique normalization of the Haar measure that induces a probability measure on the union of the orbits of the vectors $v_{(I,\beta)}$ as $(I,\beta)$ runs among admissible extended jump sets for $\rho_{e,p}$. We call admissible those orbits of $M_{\rho_{e,p}}^{f-1} \oplus M_{\rho_{e,p}}^{*}$, under the action of $\text{Aut}_{\text{filt}}(M_{\rho_{e,p}}^{f-1} \oplus M_{\rho_{e,p}}^{*})$, that contain a vector $v_{(I,\beta)}$ with $(I,\beta)$ admissible. 
 Let $E_f$ be $\mathbb{Q}_{p^f}(\zeta_p)$, the unramified extension of $\mathbb{Q}_p(\zeta_p)$ of degree $f$.
\begin{theorem} \label{counting}
 Let $e \in (p-1)\mathbb{Z}_{\geq 1}, f \in \mathbb{Z}_{\geq 1}$ and $(I,\beta) \in \emph{Jump}_{\rho_{e,p}}^{*}$ be an admissible jump set. Then the probability that a random totally ramified degree $\frac{e}{p-1}$ extension $K$ of $E_f$ satisfies $(I_K,\beta_K)=(I,\beta)$, is equal to the probability that a vector $v \in M_{\rho_{e,p}}^{f-1} \oplus M_{\rho_{e,p}}^{*} $, randomly chosen among admissible orbits, is in the orbit of $v_{(I,\beta)}$. In other words
$$\sum_{K\in S(\frac{e}{p-1},E_f): (I_K,\beta_K)=(I,\beta)}\mu_{\frac{e}{p-1},E_f}(K)=\mu_{\emph{\text{Haar}}}({\emph{\text{filt-ord}}}^{-1}(I,\beta)).$$
\end{theorem}
From the first proof given by Serre \cite{Serre m.f.}, Theorem \ref{counting} can be equivalently expressed as a volume computation in a space of Eisenstein polynomials. Namely for $e \in (p-1)\mathbb{Z}_{\geq 1}$ and $f \in \mathbb{Z}_{\geq 1}$, denote by $\text{Eis}(\frac{e}{p-1}, \mathbb{Q}_{p^f}(\zeta_p))$ the set of degree $\frac{e}{p-1}$-Eisenstein polynomials over $\mathbb{Q}_{p^f}(\zeta_p)$. This can be viewed as a topological space equipped with a natural probability measure, simply by using the Haar measure on the coefficients. For a $g(x) \in \text{Eis}(\frac{e}{p-1}, \mathbb{Q}_{p^f}(\zeta_p))$, denote by $F_{g(x)}:=\mathbb{Q}_{p^f}(\zeta_p)[x]/(g(x))$. We can reformulate Theorem \ref{counting} in the following manner.

\begin{theorem} \label{counting2}
 Let $e \in (p-1)\mathbb{Z}_{\geq 1}, f \in \mathbb{Z}_{\geq 1}$ and $(I,\beta) \in \emph{Jump}_{\rho_{e,p}}^{*}$ be an admissible jump set. Then the volume of the set of $g(x) \in \emph{Eis}(\frac{e}{p-1}, \mathbb{Q}_{p^f}(\zeta_p))$ satisfying $(I_{F_{g(x)}},\beta_{F_{g(x)}})=(I,\beta),$
equals
$$\mu_{\emph{\text{Haar}}}({\emph{\text{filt-ord}}}^{-1}(I,\beta)).
$$
\end{theorem} 
The above two Theorems are implied by Theorem \ref{counting rephrased}. As a bonus, the method of the proof of Theorem \ref{counting rephrased} allows us to \emph{explicitly compute} the jump set $(I_{F_{g(x)}},\beta_{F_{g(x)}})$ out of the valuation of the coefficients of $g(x)$, for a large class of Eisenstein polynomials $g(x)$. This will be the class of \emph{strongly separable Eisenstein polynomials}, which are defined right after Proposition \ref{strongly sep. poly}. To state our next Theorem, we begin attaching to any $g(x) \in \text{Eis}(\frac{e}{p-1}, \mathbb{Q}_{p^f}(\zeta_p))$, an element $(I_{g(x)},\beta_{g(x)}) \in \text{Jump}_{\rho_{\infty,p}}$. Under certain conditions, given below, we have that actually $(I_{g(x)},\beta_{g(x)}) \in \text{Jump}_{\rho_{e,p}}^{*}$ and $(I_{F_{g(x)}},\beta_{F_{g(x)}})=(I_{g(x)},\beta_{g(x)})$. We shall begin by explaining the construction of $(I_{g(x)},\beta_{g(x)})$. Write
$$g(x):=\sum_{i=0}^{\frac{e}{p-1}}a_ix^i.
$$ 
Firstly consider the following subset of $\mathbb{Z}^2$
$$S_{g(x)}:=\{\big(\frac{\text{v}_{E_f}(a_i)\frac{e}{p-1}+i}{p^{\text{v}_{\mathbb{Q}_p}(i)}},\text{v}_{\mathbb{Q}_p}(i)+1 \big) \}_{1 \leq i \leq \frac{e}{p-1}: \text{v}_{\mathbb{Q}_p}(i) \leq \text{v}_{\mathbb{Q}_p}(e) \ \text{and} \ a_i \neq 0}.
$$
Recall the definition of the partial order $\leq_{\rho}$ attached to a shift $\rho$ given right after Theorem \ref{j.s.param.orbit}. We denote by $S_{g(x)}^{-}$ the set of \emph{minimal} elements of $S_{g(x)}$ with respect to the order $\leq_{\rho_{\infty,p}}$. One can prove that there is a unique pair
$$(I_{g(x)},\beta_{g(x)}) \in \text{Jump}_{\rho_{\infty,p}},
$$
such that $S_{g(x)}^{-}=\text{Graph}(\beta_{g(x)})$. It turns out that if $g(x)$ is strongly separable, a notion that we are going to provide right after Proposition \ref{strongly sep. poly}, then the pair $(I_{g(x)},\beta_{g(x)})$ is also in $\text{Jump}_{\rho_{e,p}}$. 
 
We next make a definition that will have the effect of sub-dividing the characteristic $0$ local field extensions into two sub-categories. Loosely speaking, when the ramification of $E/F$ will not be ``too big" compared to $\text{v}_{E}(p)$, then the arithmetic of this extension will be, for our purposes, indistinguishable from the arithmetic of a characteristic $p$ extension. We make this notion precise in the following definition, while the relation with characteristic $p$ fields will only become visible in Theorem \ref{a surprising relation 2}. For an extension of local fields $F/E$ we denote by $\delta_{F/E}$ the different of the extension. 
\begin{definition} \label{strongly Eis}
Let $F/E$ be any extension of local fields of residue characteristic $p$. We say that $F/E$ is \emph{strongly separable} if 
$$\text{v}_{F}(\delta_{F/E})<\text{v}_{F}(p).
$$
\end{definition}
Observe that in characteristic $p$ the notions of strongly separable and separable coincide. One can easily show the following general fact.
\begin{proposition} \label{strongly sep. poly}
Let $n$ be a positive integer. Consider $F/E$ a monogenic degree $n$ extension given by an Eisenstein polynomial $g(x):=\sum_{i=0}^{n}a_ix^i$. Then $F/E$ is strongly separable if and only if there exists $i \in \{1,\ldots ,n\}$ such that $(i,p)=1$ and $\emph{v}_{E}(a_i)<\emph{v}_E(p)$. 
\end{proposition}
An Eisenstein polynomial $g(x) \in \text{Eis}(n,E)$ giving rise to a strongly separable extension is itself called strongly separable. So Proposition \ref{strongly sep. poly} says that $g(x)$ is strongly separable if and only if it has a coefficient $a_i$ with $(i,p)=1$ and $\text{v}_E(a_i)<\text{v}_E(p)$. We can now state our next result. For a positive integer $f$, recall that $E_f$ denotes $\mathbb{Q}_{p^f}(\zeta_p)$, the unramified extension of $\mathbb{Q}_p(\zeta_p)$ of degree $f$. 
\begin{theorem} \label{valuation coefficients}
Let $e \in (p-1)\mathbb{Z}_{\geq 1}, f \in \mathbb{Z}_{\geq 1}$ and $g(x) \in \emph{Eis}(\frac{e}{p-1}, E_f)$ be strongly separable. Then
$$(I_{g(x)},\beta_{g(x)})=(I_{E_f[x]/g(x)},\beta_{E_f[x]/g(x)}).
$$

\end{theorem} 
As explained at the end of Section \ref{finding jump sets inside}, the assumption of being strongly separable cannot be omitted. 
Theorem \ref{valuation coefficients} is deduced in Section \ref{finding jump sets inside} from a slightly finer result. Moreover in that Section we provide a \emph{procedure} that allows one to compute $(I_{g(x)},\beta_{g(x)})$ very quickly, even by hand. See \cite{de Boer--Pagano} for an actual implementation of this as well. 

The moral of Theorem \ref{valuation coefficients} is that in a portion of the space of Eisenstein polynomials, the assignment $K \mapsto (I_K,\beta_K)$ can be read off very explicitly from the valuations of the coefficients of an Eisenstein polynomial giving the field $K$. In general this is not the case, but nevertheless one is able to establish the exact counting formula as in Theorem \ref{counting} by means of a genuinely probabilistic argument.

\subsubsection{Answer to question $(2)$}
Let $n$ be a positive integer and let $L/K$ be a degree $n$ totally ramified separable extension of local fields with residue characteristic $p$. Suppose $L/K$ is given by $g(x) \in \text{Eis}(n,K)$, i.e. $L=K[x]/g(x)$. Denote by $\Gamma_L$ the metric space introduced in \ref{Question 2}. One can find invariants of $g(x)$ from the structure of the metric space $\Gamma_L$ as follows. Fix $\pi \in K^{\text{sep}}$ a root of $g(x)$. Denote by $\sigma_{\pi} \in \Gamma_L$ the corresponding embedding 
$$ \sigma_{\pi}(x)=\pi.
$$ 
Consider the polynomial 
$$g_{\text{twist}}(t)=g(\pi \cdot t+\pi) \in K[\pi][t].
$$ 
The knowledge of the Newton polygon of $g_{\text{twist}}(t)$ tells us precisely how the distances are disposed around $\sigma_{\pi}$ in $\Gamma_L$. But recall that $\Gamma_L$ is a transitive $G_K$-set, and every element of $G_K$  acts as an isometry on $\Gamma_L$. Hence the Newton polygon of $g_{\text{twist}}(\pi \cdot x+\pi)$ is an invariant of the metric space $\Gamma_{L}$ independent of the choice of $\pi$ and of $g$. Denote this polygon by $$\text{Newt}(L/K).
$$ 
Observe that in case $L/K$ is Galois, then the knowledge of $\text{Newt}(L/K)$ amounts to the knowledge of the map $\mathbb{Z}_{>0} \to \mathbb{Z}_{>0}$
$$ u \mapsto |\text{Gal}(L/K)_{u}|, \ (u \in \mathbb{Z}_{>0})
$$
where $\text{Gal}(L/K)_u$ denotes the lower $u$-th ramification group as defined in \cite{Local fields}. But $\text{Newt}(L/K)$ makes sense also for non-Galois extensions. 

This Newton polygon is called the \emph{ramification polygon} in the literature, and, among other things, a complete survey on this subject can be found in \cite{Pauli--Sinclair}. In that paper the polynomial in consideration is instead $\frac{g(\pi t+\pi)}{\pi^n}$. Of course this has simply the effect of shifting the polygon vertically by $-n$. As it will become clear to the reader in a moment, we have chosen our normalization since the form of our results is slightly more pleasant with our convention. 

The following fact, certainly folklore, can be shown by direct inspection. We refer the reader to Section \ref{Procedure Eis} for how to calculate in practice $(I_{g(x)},\beta_{g(x)})$: this together with the basic properties of $\text{Newt}(L/K)$, which can be found in \cite{Pauli--Sinclair}, gives the following fact quite rapidly. 
\begin{theorem} \label{strongly separable easy newton}
Let $n$ be a positive integer and let $K$ be a local field with residue characteristic $p$. Let $g(x) \in \emph{Eis}(n, K)$ be a strongly separable polynomial. Then
$$\emph{\text{Lower-Convex-Hull}}(\{(p^{\beta_{g(x)}(i)-1},p^{\beta_{g(x)}(i)-1}i): i \in I_{g(x)} \} \cup \{(n,n)\})=\emph{\text{Newt}}(K[x]/g(x)/K).
$$
\end{theorem} 
In other words Theorem \ref{strongly separable easy newton} gives us a way to read off $\text{Newt}(K[x]/g(x)/K)$  from $(I_{g(x)},\beta_{g(x)})$, in case $g(x)$ is strongly separable. Hence combined with Theorem \ref{valuation coefficients} we obtain the following surprising result. 
\begin{theorem} \label{a surprising relation}
Let $L/\mathbb{Q}_{p^f}(\zeta_p)$ be a strongly separable totally ramified extension. Then
$$ \emph{\text{Lower-Convex-Hull}}(\{(p^{\beta_{L}(i)-1},p^{\beta_{L}(i)-1}i): i \in I_L \} \cup \{(n,n)\})=\emph{\text{Newt}}(L/\mathbb{Q}_{p^f}(\zeta_p)).$$
\end{theorem}
Hence for a strongly separable extension $L/\mathbb{Q}_{p^f}(\zeta_p)$ the knowledge of the filtered $\mathbb{Z}_p$-module $U_{\bullet}(L)$ implies the knowledge of the ramification polygon $\text{Newt}(L/\mathbb{Q}_{p^f}(\zeta_p))$.  Moreover we see something else going on: for such an extension the full object $(I_{g(x)},\beta_{g(x)})$ is an invariant of the extension. This indeed follows from Theorem \ref{valuation coefficients}: that Theorem is telling us that the object $(I_{g(x)},\beta_{g(x)})$ encodes the structure of $U_{\bullet}(\mathbb{Q}_{p^f}(\zeta_p)[x]/g(x))$ as a filtered $\mathbb{Z}_p$-module. But in the more general case of Theorem \ref{strongly separable easy newton} we see a priori only a way to \emph{deduce} an invariant from $(I_{g(x)},\beta_{g(x)})$, without any structural information provided for $(I_{g(x)},\beta_{g(x)})$ itself. In particular it gives us no a priori guarantees that $(I_{g(x)},\beta_{g(x)})$ is the same as $g(x)$ varies among polynomials representing the same field. In Section \ref{generalization of I beta} we pinpoint this additional structural information. Namely to \emph{any} strongly separable extension $L/K$ of local fields, we will attach $(I_{L/K},\beta_{L/K})$, a $\rho_{\infty,p}$-jump set that encodes structural information about the filtered inclusion
$$U_{\bullet}(K) \subseteq U_{\bullet}(L). 
$$
In particular, if $\mu_p(L)=\{1\}$ then $(I_{L/K},\beta_{L/K})$ has the following simple interpretation. In this case one can attach, essentially by means of Theorem \ref{j.s.param.orbit}, to any  element $u$ of $U_1(K)-U_2(K)$ a $\rho_{e_L,p}$-jump set $(I_{L/K}(u),\beta_{L/K}(u))$. The jump set $(I_{L/K}(u),\beta_{L/K}(u))$ tells us the orbit of $u$ under the action of $\text{Aut}_{\text{filt}}(U_{\bullet}(L))$. Let $u$ be any element of $U_1(K)-U_2(K)$ and let $g(x)$ be any Eisenstein polynomial giving $L/K^{\text{nr}}$, where $K^{\text{nr}}$ is the maximal unramified extension of $K$ in $L$. It turns out that $(I_{L/K}(u),\beta_{L/K}(u))=(I_{g(x)},\beta_{g(x)})$. In particular all the elements of $U_1(K)-U_2(K)$ are in the same orbit for the action of $\text{Aut}_{\text{filt}}(U_{\bullet}(L))$. This orbits correspond to a single jump set $(I_{L/K},\beta_{L/K})$. 

For general strongly separable extensions of local fields we have the following joint generalization of Theorem \ref{valuation coefficients} and Theorem \ref{a surprising relation}.

\begin{theorem} \label{a surprising relation 2}
Let $L/K$ be a strongly separable totally ramified extension of local fields of residue characteristic $p$. Then 
$$ \emph{\text{Lower-Convex-Hull}}(\{(p^{\beta_{L/K}(i)-1},p^{\beta_{L/K}(i)-1}i): i \in I_{L/K}\} \cup \{(n,n)\})=\emph{\text{Newt}}(L/K).$$
Moreover if $L/K$ is given by an Eisenstein polynomial $g(x)$, then
$$(I_{L/K},\beta_{L/K})=(I_{g(x)},\beta_{g(x)}).
$$
\end{theorem}
Therefore Theorem \ref{a surprising relation 2} provides an intrinsic description of $(I_{g(x)},\beta_{g(x)})$ as a filtered invariant of the corresponding inclusion of groups of principal units. In particular this says that $(I_{g(x)},\beta_{g(x)})$ is an invariant of the Eisenstein polynomial $g(x)$ as long as $g(x)$ is strongly separable. 

\subsubsection{Answer to question (3)}
Denote by $\mathcal{J}_{K}$ the set of possible sets of jump for a character of $U_1(K)$. Clearly $\mathcal{J}_K$ is determined by the structure of $U_1(K)$ as a filtered $\mathbb{Z}_p$-module. So one can use the answer to question (1) in order to answer question (3). The first step is answering the same problem for free filtered modules. The main idea for doing this is again to exploit the action of the group of filtered automorphisms. Denote by $\widehat{M_{\rho}^f}$ the group of characters of $M_{\rho}^{f}$. There is a natural action of $\text{Aut}_{\text{filt}}(M_{\rho}^{f})$ on $\widehat{M_{\rho}^f}$. The action clearly preserves the set of jumps of each character. It turns out that conversely one can reconstruct the orbit of the character from the set of jumps: two characters in $\widehat{M_{\rho}^{f}}$ are in the same orbit under the action of $\text{Aut}_{\text{filt}}(M_{\rho}^{f})$ if and only if they have the same set of jumps. Moreover the possible sets of jumps are exactly the $\rho$-jump sets. This fact is expressed in the following theorem.
\begin{theorem}(\emph{Jump sets parametrize orbits of characters}) \label{Jump sets parametrize orbits of characters}
Let $\rho$ be a shift, and $f$ be a positive integer. Then the set of possible sets of jumps of characters of the free-filtered $\mathbb{Z}_p$-module $M_{\rho}^{f}$ is exactly $\emph{Jump}_{\rho}$. Moreover two characters have the same set of jumps if and only if they are in the same orbit under the group $\emph{Aut}_{\emph{filt}}(M_{\rho}^{f})$.
\end{theorem}
So in particular we have the following result.
\begin{theorem}
Let $K$ be a local field with $\mu_p(K)=\{1\}$, then $\mathcal{J}_K=\emph{Jump}_{\rho_K}$.
\end{theorem}
We now consider the case $\mu_p(K) \neq \{1\}$. By Theorem \ref{the quasi-free guys}, we first look at the possible sets of jumps of characters of $M_{\rho}^{f-1} \oplus M_{\rho}^{*}$. These are precisely the extended jump sets, as we next explain. 
\begin{theorem}(\emph{Jump sets parametrize orbits of characters---part 2})
Let $\rho$ be a shift with $\#T_{\rho}<\infty$. Let $f$ be a positive integer. Then the set of possible sets of jumps of characters of the free-filtered $\mathbb{Z}_p$-module $M_{\rho}^{f-1} \oplus M_{\rho}^{*}$ is exactly $\emph{\text{Jump}}_{\rho}^{*}$. Moreover two characters have the same set of jumps if and only if they are in the same orbit under the group $\emph{\text{Aut}}_{\emph{\text{filt}}}(M_{\rho}^{f-1} \oplus M_{\rho}^{*})$.
\end{theorem}
We then show that this, essentially thanks to Proposition \ref{a lot surjective}, implies that $\mathcal{J}_{K} \subseteq \text{Jump}_{\rho_K}^{*}$ always, i.e. a set of jumps for a character is always an extended $\rho_K$-jump set. The remaining task is to classify which orbits of characters of $M_{\rho}^{f-1} \oplus M_{\rho}^{*}$ admit a representative killing a given element of $M_{\rho}^{f-1} \oplus M_{\rho}^{*}$. In this way we obtain the final classification, which is Theorem \ref{classification wild characters}. This Theorem says that $\mathcal{J}_K$ consists of the elements of $\text{Jump}_{\rho_K}^{*}$ that are $(I_K,\beta_K,f_K,p)$-\emph{compatible}. Compatibility is an explicit combinatorial criterion that consists in a comparison between a jump set $(I,\beta)$ and the jump set of the field $(I_K,\beta_K)$: in the comparison an important role is played by the case distinction of whether $f_K \geq 2$ or not and whether $p=2$ or not. For a precise definition see Definition \ref{incompatibility}. In the rest of Section \ref{wild extension} we establish several explicit applications of this criterion, stressing especially the first dichotomy. As an example we give here the following result. 
\begin{theorem}
Let $K_1,K_2$ be two totally ramified extensions of $\mathbb{Q}_p(\zeta_p)$. Then $\mathcal{J}_{K_1}=\mathcal{J}_{K_2}$ if and only if $U_{\bullet}(K_1) \simeq _{\mathbb{Z}_p\emph{-filt}} U_{\bullet}(K_2)$. 
\end{theorem}
In other words, for totally ramified extensions $K/\mathbb{Q}_p(\zeta_p)$, not only do we have an explicit criterion to compute $\mathcal{J}_{K}$ from the filtered $\mathbb{Z}_p$-module $U_{\bullet}(K)$, but we can conversely reconstruct the filtered $\mathbb{Z}_p$-module $U_{\bullet}(K)$ from $\mathcal{J}_{K}$.

Finally we remark that, by the reciprocity map, this criterion gives an explicit classification of the possible sets of jumps in the upper numbering of a cyclic wild extension of a local field. We explain this in further detail in Section \ref{wild extension}. 

\subsection{Further results and questions}
We hope to have shed some light on the role that the jump set $(I_K,\beta_K)$ plays in the arithmetic of the local field $K$. This makes some basic questions about this invariant worth investigating. A very basic one is the following. Let $K$ be a local field with $\mu_p(K) \neq \{1\}$. Let $e$ be in  $e_K \mathbb{Z}_{\geq 1}$ and let $f$ be in $f_K\mathbb{Z}_{\geq 1}$. \\

\emph{Question}: For which $(I,\beta) \in \text{Jump}_{\rho_{e,p}}^{*}$ does there exist an extension $L$ of $K$ such that $e_L=e, f_L=f$ and $(I_L,\beta_L)=(I,\beta)$?   \\

We have made some progress on this question, see Section \ref{Jump sets under field extensions}. In that Section we establish some peculiarly specific rules that constraint the possible changes of a jump set under a totally ramified extension. As the reader will learn in that section, the interesting case, among totally ramified extensions, is only that of wild extensions. In the present paper we leave open a complete characterization of which jump sets occur under such extensions, providing only necessary conditions. From further calculations, not included in the present paper, we believe that a full classification might be within reach, but the final result might look quite intricate. 

In a different direction, we would like to mention that most of the results of the present paper can be viewed as an investigation of the filtered $\mathbb{Z}_p$-modules arising from taking points of one of the simplest formal groups, namely $\mathbb{G}_m$. The theory in Section \ref{filtered modules} should be general enough to cover the case of other Lubin-Tate formal groups giving rise to filtered $O_K$-modules with cyclic torsion sub-module, where $K$ is any other local field, and $O_K$ its ring of integers. It would be an interesting investigation to see which of the results of the present paper extend to this context. For instance, it should be possible to provide a theorem on the lines of Theorem \ref{admissible j.s. occur}.  

Finally we would like to conclude with yet another potentially worthwhile direction of investigation. Our mass formula, contained in Theorem \ref{counting}, follows the first interpretation of Serre's weight for local fields, namely using volumes of Eisenstein polynomials. But Serre \cite{Serre m.f.} established also a different interpretation of these weights, by means of division algebras. This suggests the possibility of studying the filtered pro-$p$ group $U_{\bullet}(D)$ of principal units of a central division algebra over a local field, and to study the action of the group $\text{Aut}_{\text{filt}}(U_{\bullet}(D))$ on the set of maximal abelian filtered $\mathbb{Z}_p$-sub-modules. It would be very elegant to reach in this manner a different proof of Theorem \ref{counting}. 

\subsection{Comparison with the literature}
An explicit classification of the possible upper jumps of wild characters of a local field $K$, i.e.\ of the set $\mathcal{J}_K$, was given in a series of papers, by respectively Maus, Miki and Sueyoshi \cite{Maus}, \cite{Miki}, \cite{Sueyoshi}. The first author has given the criterion for characteristic $p$ local fields. The full classification was given by Miki, and some of Miki's arguments in \cite{Miki} were simplified by Suyeoshi in \cite{Sueyoshi}, where the reader can find also a neat statement for Miki's criterion. Two points come here in order. The first point is that in \cite{Miki} and \cite{Sueyoshi} the invariant $(I_K,\beta_K)$ was already introduced. This is buried in \cite[Lemma 17]{Miki}. In the language of this paper, we can say that $(I_K,\beta_K)$ was understood as the unique element of $\text{Jump}_{\rho_K}^{*}$ such that there is an equation of the form $\zeta_p=\prod_{i \in I_K} u_i^{p^{\beta_K(i)-1}}$, where $\text{v}_K(u_i-1)=i$, and in case $\frac{pe_K}{p-1} \in I_K$, then $u_{\frac{pe_K}{p-1}} \not \in K^{*p}$. The uniqueness was proved in an ad hoc manner in the above mentioned \cite[Lemma 17]{Miki}. The present work is the first place in the literature where the \emph{structural meaning} of the invariant $(I_K,\beta_K)$ is established: it gives, together with $f_K$ and $p:=\text{char}(O_K/m_K)$, the structure of $U_{\bullet}(K)$ as a filtered module. Apart from being conceptually more satisfying, this slightly more abstract approach has two practical advantages. Firstly it leads naturally to all the above mentioned additional results: the interpretation of jump sets in terms of \emph{filtered orbits} of vectors, see Theorem \ref{j.s.param.orbit}, leads to the mass formula for unit filtrations, Theorem \ref{counting}, which in turns leads naturally to Theorem \ref{a surprising relation}, which links the filtered structure of $U_{\bullet}(K)$ with ramification theory. To the best of our knowledge all these results are new. Secondly the interpretation of jump sets as parametrizing filtered orbits of characters, see \ref{Jump sets parametrize orbits of characters}, makes it an easy job to deduce, from first principles, our classification of the possible sets of jumps for a character, contained in Theorem \ref{classification wild characters}. This brings us to the second point. Namely the  combinatorial criterion of \cite{Miki} is not tautologically equal to the one contained in Theorem \ref{classification wild characters}. We check, by direct combinatorial inspection, that they coincide in Proposition \ref{inadequacy equivalent to incompatibility}, showing in this way that the tools of this paper give, among other things, a simple unified approach to deduce all the results in \cite{Maus}, \cite{Miki} and \cite{Sueyoshi}, by means of a general theory of filtered modules. 

Coming to more recent literature, in 2014, I. del Corso and L. Capuano \cite{Capuano--Del Corso} have obtained a classification of all possible upper jumps in an exponent $p$ extensions of a local field $K$. It would be interesting to push this further obtaining a classification, for \emph{any} finite abelian $p$-group $A$, of the possible structures $A_{\bullet}$ as \emph{filtered group} on $A$ such that $\text{Epi}_{\text{filt}}(U_{\bullet}(K),A_{\bullet}) \neq \emptyset$. For instance, this might be useful in counting the average number of extensions with prescribed ramification data at $p$, in families of number fields containing $\zeta_p$. For a first work in the direction of such counting with ``prescribed ramification", see \cite{Pagano--Sofos}.

Finally we would like to mention that the ramification polygon of an Eisenstein polynomial has been the object of study of several papers \cite{Pauli}, \cite{Romano}, \cite{Pauli--Sinclair}, especially in relation to the problem of calculating Galois groups of Eisenstein polynomials. In his Ph.D. thesis, D. Romano \cite{Romano} provided a characterization of \emph{strongly Eisenstein} polynomials in terms of their Galois group. In a sense these are the polynomials with the simplest possible ramification polygon. It is then interesting that strongly Eisenstein polynomials $g(x)$ over $\mathbb{Q}_{p^f}(\zeta_{p^j})$ with $(p,j) \neq (2,1)$ and $\text{v}_{\mathbb{Q}_p}(\text{deg}(g(x)))>j$, can be also characterized in terms of filtered modules, see Theorem \ref{strongly Eisenstein characterized}. Under the assumption $(p,j) \neq (2,1)$ and $\text{v}_{\mathbb{Q}_p}(\text{deg}(g(x)))>j$, these polynomials are the ones giving the simplest possible filtered module, which is also the most frequent one, in the sense of Theorem \ref{counting}: it occurs $\frac{p^f-1}{p^f}$ of the times, just as the probability for an Eisenstein polynomial over $\mathbb{Q}_{p^f}(\zeta_{p^j})$ to be strongly Eisenstein. The work of Romano has been substantially refined by S. Pauli and C. Greve \cite{Pauli}. 

\begin{acknowledgements}
This project is part of my Ph.D. work. I would like to warmly thank my Ph.D. advisor, Hendrik Lenstra, for suggesting that I could think about this subject. In particular I would like to thank him for suggesting several of the starting ideas of the project and for frequent insightful feedback on my progress. I would also like to thank him for relevant pointers to the literature on local fields: this considerably enriched the scope of the results of this paper. 
 
Many thanks go to Ted Chinburg and Sebastian Moore, for showing interest in this research, for following its development and for providing several useful suggestions. In particular I would like to thank them for suggesting to seek for an \emph{explicit} relation between the jump set $(I_K,\beta_K)$ and the set of sets of jumps $\mathcal{J}_K$. 

I am thankful to Ilya Nekrasov and Sergey Vostokov for showing interest in these results,  asking me for a talk on this subject, where they provided useful feedback. 

I would like to thank Tim Dokchitser for suggesting to contact Maurizio Monge. 

Many thanks to Maurizio Monge for pointing out the papers \cite{Maus}, \cite{Miki} and \cite{Sueyoshi}, when I announced to him Theorem \ref{classification wild characters}. Also I would like to thank him for providing me with some of his notes on these papers. 
\end{acknowledgements}

\section{Jump sets} \label{jump sets}
The goal of this section is to define and explain the notion of a jump set, which is the key object of this paper. Jump sets are defined in terms of shifts. A \emph{shift} is a strictly increasing function $\rho:\mathbb{Z}_{\geq 1} \to \mathbb{Z}_{\geq 1}$, with $\rho(1)>1$. For a shift $\rho$, we denote by $T_{\rho}$ the set $\mathbb{Z}_{\geq 1}-\rho(\mathbb{Z}_{\geq 1})$. If $T_{\rho}$ is finite, we denote by $e^{*}$ the positive integer $\text{max}(T_{\rho})+1$. We denote by $e'_{\rho}$ the positive integer $\rho^{-1}(e^*)$. The shifts that will be relevant for local fields are the ones explained in the following.
\begin{example} \label{main example shift}
For $p$ a prime, and $e \in \mathbb{Z}_{>0} \cup \{\infty\}$ denote $\rho_{e,p}(i)=\text{min}\{i+e,pi\}$. It is a shift. 
Clearly $T_{\rho_{e,p}}$ is finite iff $e$ is finite. Indeed one has always $e=|T_{\rho_{e,p}}|$. If $e \neq \infty$, then $e^{*}=\lceil \frac{pe}{p-1} \rceil$. 
The reason why these shifts will play a role is due to the following property. 

\emph{Crucial property}: let $K$ local field, of residue characteristic $p$, let $e=\text{v}_K(p)$, then we have that
 $$U_i^p \subset U_{\rho(i)},$$ 
for $\rho=\rho_{e,p} \ (=\rho_K)$. One can see this by inspection of the valuations in the expansion $(1+x)^p=1+px+ \ldots +x^p$.
\end{example}
We now define $\rho$-jump sets (resp.\ extended $\rho$-jump sets).
\begin{definition}
 A \emph{jump set} for $\rho$ (resp.\ an \emph{extended jump} set for $\rho$) is a finite subset $A \subseteq \mathbb{Z}_{\geq 1}$ such that:\\ \\
$\bullet$ \ if $a,b \in A$, and $a<b$ then $\rho(a) \leq b$, \\
 $\bullet$ \ $A - \rho(A) \subseteq T_{\rho}$ (resp.\ $A - \rho(A) \subseteq T_{\rho}^{*}=T_{\rho} \cup \{e^{*}\}$). \\
\\
 Write $\text{Jump}_{\rho}=\{\text{jump sets for $\rho$}\}$ \ (resp.\ $\text{Jump}_{\rho}^{*}=\{\text{extended jump sets for $\rho$}\}$).
\end{definition}
A jump set for $\rho$ will also be called $\rho$-jump set (resp.\ an extended jump set for $\rho$ will also be called an extended $\rho$-jump set).

If $A$ is a $\rho$-jump set (resp.\ an extended jump set) then we denote by $I_A$ the set $A - \rho(A)$, and by $\beta_A$ the map $\beta_A:I_{A} \to \mathbb{Z}_{\geq 1}$, $i \to |[i,\infty) \cap A|$. This allows us to express the notion of jump sets in different, but equivalent, terms. Namely the pair $(I_A,\beta_A)$ evidently has the following three properties. \\
\\
$(1)$ \ $I_A \subseteq T_{\rho}=\mathbb{Z}_{>0}-\rho(\mathbb{Z}_{>0})$ (resp.\ $I_A \subseteq T_{\rho}^{*}=T_{\rho} \cup \{e^{*}\}$), \\
$(2)$ \ $\beta_A$ is a strictly decreasing map $\beta:I_A \to \mathbb{Z}_{\geq 1}$, \\
$(3)$ \ the map $i \mapsto \rho^{\beta(i)}(i)$  from $I_A$ to $\mathbb{Z}_{\geq 1}$ is strictly increasing. \\
\\
Suppose now we have a pair $(I,\beta)$ with the three above properties $(1),(2),(3)$. We can attach to such an $(I,\beta)$ an element $A_{(I,\beta)}$ of $\text{Jump}_{\rho}$ (resp.\ of $\text{Jump}_{\rho}^{*}$) defined as follows. If $I=\emptyset$ then $A_{(I,\beta)}=\emptyset$. Suppose now that $I$ is not empty. Then put
$$A_{(I,\beta)}:=\{\rho^{n}(i)\}_{i \in I-\{\text{max}(I)\}, 0 \leq n <\beta(i)-\beta(s(i))} \cup \{\rho^{n}(\text{max}(I))\}_{0 \leq n <\beta(\text{max}(I))},
$$
where, for $i \in I-\{\text{max}(I)\}$, the element $s(i)$ denotes the successor of $i$ in $I$. The following proposition follows in a straightforward manner from the definitions. 
\begin{proposition} \label{equivalent data}
The assignments $A \mapsto (I_A,\beta_A)$ and $(I,\beta) \mapsto A_{(I,\beta)}$ are inverse to each other yielding a bijection between $\emph{Jump}_{\rho}$ (resp.\ $\emph{Jump}_{\rho}^{*}$) and the set of pairs $(I,\beta)$ having the following properties: \\
\\
$\bullet$ \ $I \subseteq T_{\rho}=\mathbb{Z}_{>0}-\rho(\mathbb{Z}_{>0})$ (resp.\ $I \subseteq T_{\rho}^{*}=T_{\rho} \cup \{e^{*}\}$), \\
$\bullet$ \ $\beta$ is a strictly decreasing map $\beta:I \to \mathbb{Z}_{\geq 1}$, \\
$\bullet$ \ the map $i \mapsto \rho^{\beta(i)}(i)$  from $I$ to $\mathbb{Z}_{\geq 1}$ is strictly increasing.
\end{proposition}

From now on, we shall often write $(I,\beta)$ to denote a jump set (resp.\ an extended jump set), meaning implicitly that we are identifying it with an actual jump set via the above mentioned bijection.
\begin{example}
$\bullet$ There is a unique jump set having $|I|=0$, namely the empty set $A=\emptyset \in \text{Jump}_{\rho}$. \\
$\bullet$ \ A $\rho$-jump set (resp.\ extended $\rho$-jump set) $(I,\beta)$ with $|I|=1$ is given by the choice of an element, $a$, of $T_{\rho}$ (resp.\ $T_{\rho}^{*}$), and of a positive integer $m=\beta(a)$. The actual jump set will then be $\{a,\rho(a),\ldots,\rho^{m-1}(a)\}$. \\
$\bullet$ \ A $\rho$-jump set (resp.\ extended $\rho$-jump set) $(I,\beta)$ with $|I|=2$ is given by the choice of two elements, $a<b$, of $T_{\rho}$ (resp.\ $T_{\rho}^{*}$)), and of two positive integers $m_1=\beta(a)>\beta(b)=m_2$, such that $\rho^{m_1-m_2}(a)<b$ (or equivalently $\rho^{m_1}(a)<\rho^{m_2}(b)$). The actual jump set will then be $\{a,\rho(a),\ldots ,\rho^{m_1-m_2-1}(a) \} \cup \{b,\rho(b),\ldots,\rho^{m_2-1}(b) \}$. 
\end{example}
\begin{example}
We now explain a general procedure to inductively construct any jump set $A$ for $\rho$ (resp.\ extended jump set). As a first step one decides whether $A=\emptyset$ or not. In case $A=\emptyset$ one has obtained a jump set and stops. Suppose instead that one wants to construct a jump set $A \neq \emptyset$. Then pick an $i_1 \in T_{\rho}$ (resp.\ in $T_{\rho}^{*}$) and a positive integer $n_1$. Consider the set $$A_1:=\{\rho^{j}(i_1)\}_{0 \leq j<n_1}.
$$
Now you can stop and have obtained a jump set $A:=A_1$. In this case $I=\{i_1\}$ and $\beta(i_1)=n_1$. If you want instead a jump set with $|I|>1$, then you check whether there is a $y \in T_{\rho}$ (resp.\ in $T_{\rho}^{*}$) such that $\rho^{n_1}(i_1)<y$. If such a $y$ doesn't exist, then we set $A:=A_1$ and we stop having obtained a jump set (resp.\ an extended jump set). Otherwise you pick any such $y$ and put $y:=i_2$ and pick a positive integer $n_2$. Then write 
$$A_2:=A_1 \cup \{\rho^{j}(i_2)\}_{0 \leq j<n_2}. 
$$
Now you can stop and have obtained a jump set $A:=A_2$. In this case $I=\{i_1,i_2\}$ and $\beta(i_1)=n_1+n_2, \beta(i_2)=n_1$. If you want instead a jump set with $|I|>2$, then you check whether there is a $y \in T_{\rho}$ (resp.\ $T_{\rho}^{*}$) such that $\rho^{n_2}(i_2)<y$. If such a $y$ doesn't exist, then we set $A:=A_2$ and we stop having obtained a jump set (resp.\ an extended jump set). Otherwise you pick any such $y$ and put $y:=i_3$ and pick a positive integer $n_3$. Then write
$$A_3:=A_2 \cup \{\rho^{j}(i_3)\}_{0 \leq j<n_3}.
$$
In this case we have $I=\{i_1,i_2,i_3\}$ and $\beta(i_1)=n_1+n_2+n_3, \beta(i_2)=n_2+n_3, \beta(i_3)=n_3$.

One continues inductively as follows. Having arrived at $A_k$, together with $i_k, n_k$, for $k \in \mathbb{Z}_{\geq 3}$, either we set $A:=A_k$ and we have obtained a jump set, or we verify whether there exists a $y \in T_{\rho}$ (resp.\ in $T_{\rho}^{*}$) such that $\rho^{n_k}(i_k)<y$. If such a $y$ doesn't exist then we set $A:=A_k$ and we stop having obtained a jump set (resp.\ an extended jump set). Otherwise we pick any such $y$ and set $y:=i_{k+1}$, we choose a positive integer $n_{k+1}$ and write
$$A_{k+1}=A_k \cup \{\rho^{j}(i_{k+1})\}_{0 \leq j<n_k}.
$$
The set $A_{k+1}$ is a jump set  for $\rho$ (resp.\ an extended jump set). In this case we have $I=\{i_1,\ldots ,i_{k+1} \}$ with $\beta(i_1)=n_1+\ldots+n_{k+1}, \beta(i_2)=n_2+\ldots+n_{k+1},\ldots, \beta(i_{k})=n_k+n_{k+1}, \beta(i_{k+1})=n_{k+1}$.
\end{example}
Jump sets will often arise as the set of \emph{maximal} or \emph{minimal} of certain sets, with respect to the following partial order. This partial order will also play an important role in the classification of the possible sets of jumps of a character. 
\begin{definition}
Let $(a_1,b_1),(a_2,b_2)$ be in $(\mathbb{Z}_{\geq 1})^2$. We let $(a_1,b_1) \leq_{\rho} (a_2,b_2)$ if and only if 
$$b_2 \geq b_2 \ \text{and} \ \rho^{b_2}(a_2) \geq \rho^{b_1}(a_1).
$$
 \end{definition}
 Let now $A$ be a subset of $T_{\rho}$ (resp.\ of $T_{\rho}^*$), and let $b:A \to \mathbb{Z}_{\geq 1}$. Let $\text{Max}(A,b)$ and $\text{Min}(A,b)$ be the subsets of $\text{Graph}(b)$ consisting of, respectively, the maximal and the minimal elements with respect to $\leq_{\rho}$. Then the following fact follows from the definition of a jump set. 
\begin{proposition} \label{jump set attached to a function} There are unique jump sets $(I_{(A,b)}^{+},\beta_{(A,b)}^{+})$ and $(I_{(A,b)}^{-},\beta_{(A,b)}^{-})$ \ (resp.\ extended jump sets) such that $\emph{Graph}(\beta_{(A,b)}^{+})=\emph{Max}(A,b)$ and $\emph{Graph}(\beta_{(A,b)}^{-})=\emph{Min}(A,b).$
\end{proposition}

Proposition \ref{jump set attached to a function} is repeatedly used throughout this paper.  Moreover it occurs always in the same manner, namely to recover an intrinsic description of an object presented in a non-canonical fashion. This will firstly apply in the context of filtered modules in Proposition \ref{reduction process}, to reconstruct from a coordinate representation, with respect to a filtered basis (see \ref{definition of a basis}) the orbit of a vector of a free filtered module (see \ref{free modules def}) acted upon by the group of filtered automorphisms. Another example is given by Proposition \ref{reduction for characters}, where Proposition \ref{jump set attached to a function} is used to determine the set of jumps of a character. Finally it is used in the context of Eisenstein polynomials in Theorem \ref{a surprising relation} and Theorem \ref{a surprising relation 2}. 
\section{Filtered modules} \label{filtered modules}
\subsection{Overview}
The goal of this section is to use jump sets to parametrize quasi-free filtered modules (see definition \ref{def of quasi free}). As stated in Proposition \ref{rho for loc fields}, principal units give rise to a free or quasi-free filtered module. So the material of this section will provide exactly the amount of general (elementary) theory of filtered modules sufficient to classify, in terms of jump sets, the possible structures of $U_1$, as a filtered module. 

The rest of the section is organized as follows:

In \ref{generalities filt mod} we will collect very general facts about filtered modules that will be applied in the other sections.

In \ref{DVR} we will specialize to the case where the base ring, $R$, is a complete DVR.

In \ref{rho} we explain how one can attach to a filtered module $M$ a non-decreasing function $\rho_M$, by looking at the action of $\pi_R$, a uniformizer in $R$, on the filtration.

In \ref{shit} we introduce the notion of free filtered modules: in a precise sense they stand as universal modules among those having a fixed $\rho$-map (see \ref{Universal property of free filtered modules} for the precise universal property). Next we will introduce the notion of quasi-free filtered modules, which in a precise sense are just one step more complicated than the free ones. The goal of the rest of the section is classifying quasi-free modules.

In \ref{transitive} we will provide presentations of a quasi-free filtered module via a free filtered module and exploit the action of the filtered automorphism group of the free filtered module on the set of presentations of a given quasi-free filtered module.

In \ref{orbits} we will parametrize the set of orbits of lines in a free filtered module, under the filtered automorphism group, via jump sets.

In \ref{jump sets and quasi free} we will use \ref{transitive} and \ref{orbits} to explain how jump sets parametrize the set of quasi-free filtered modules.

In \ref{reading jump inside} we explain an internal procedure to reconstruct the jump set of a quasi-free filtered module. This will suggest a generalization which will be exploited in later sections. This will be used to detect a more general connection between phenomena in the filtration and ramification theory. See also Theorem \ref{a surprising relation 2}.  
\subsection{General facts about filtered modules} \label{generalities filt mod}
Let $R$ be a commutative ring with unity.
\begin{definition} \label{definition filtered module}
\emph{A filtered $R$-module} is a sequence of $R$-modules, $M_1 \supseteq M_2 \supseteq \ldots \supseteq M_i\supseteq \ldots$ \ with $\bigcap_{i \in \mathbb{Z}_{\geq 1}}M_i=\{0\}$. 
\end{definition}

We will usually denote by $M_{\bullet}$ a filtered $R$-module $M_1 \supseteq M_2 \supseteq \ldots \supseteq M_i\supseteq \ldots$. A filtered module comes with a weight map $w:M_1 \to \mathbb{Z}_{\geq 1} \cup \{\infty\}$, defined as $w(x):=\sup \{i \in \mathbb{Z}_{\geq 1}: x \in M_i\}$. The weight map $w$ enjoys the following conditions: $w^{-1}(\{\infty\})=\{0\}$ and if $x,y \in M_1, a \in R$, then $w(x+y) \geq \text{min}\{w(x),w(y)\}$ and $w(ax) \geq w(x)$. Clearly one can recover the filtration from the knowledge of $w$, and conversely given an $R$-module $M$, together with a map $w:M \to \mathbb{Z}_{\geq 1} \cup \{\infty\}$ enjoying the above conditions, one can define the filtration $M_i:=\{x \in M: w(x) \geq i \}$. It follows that one can equivalently speak of a filtered $R$-module as a pair $(M,w)$, where $M$ is an $R$-module and $w$ is a map with the above properties. We will interchangeably denote a filtered module as $M_{\bullet}$ and as a pair $(M,w)$.
\begin{definition}\label{morphism filtered module}
Given $M_{\bullet},N_{\bullet}$ two filtered $R$-modules, a morphism of filtered $R$-modules $\phi:M_{\bullet} \to N_{\bullet}$ is a morphism of $R$-modules $\phi:M_1 \to N_1$, such that, for each positive integer $i$, $\phi(M_i) \subseteq N_i$.
\end{definition}
With definitions \ref{definition filtered module} and \ref{morphism filtered module}, filtered $R$-modules form a category, which we will denote as $\text{Filt-}R\text{-mod}$. We next explain basic constructions in this category which we will use later in this section. 
\subsubsection{Direct products and direct sums} \label{def direct sums}
Let $\{{M_h}_{\bullet}\}_{h \in \mathcal{H}}$ be a collection of filtered $R$-modules. The filtration $\prod_{h \in \mathcal{H}}M_{h,1} \supseteq \prod_{h \in \mathcal{H}}M_{h,2} \supseteq \ldots \supseteq \prod_{h \in \mathcal{H}}M_{h,n} \supseteq \ldots $ gives to $\prod_{h \in \mathcal{H}}M_{h,1}$ the structure of a filtered $R$-module. This filtered module behaves as a categorical direct product. The filtration $\bigoplus_{h \in \mathcal{H}}M_{h,1} \supseteq \bigoplus_{h \in \mathcal{H}}M_{h,2} \supseteq \ldots \supseteq \bigoplus_{h \in \mathcal{H}}M_{h,n} \supseteq \ldots$ gives to $\bigoplus_{h \in \mathcal{H}}M_{h,1}$ the structure of a filtered $R$-module. This filtered module behaves as a categorical direct sum. 
\subsubsection{Metric structure} \label{metric structure}
Let $(M,w)$ be a filtered module. Fix a real number $c \in (0,1)$. Then we have a distance on $M$, defined as $d(x,y)=c^{w(x-y)}$, which gives to $M$ the structure of a metric space and of a Hausdorff topological group. In the notation $M_{\bullet}$, the topology can be alternatively described by saying that the $\{M_i\}_{i \in \mathbb{Z}_{\geq 1}}$ form a fundamental system of neighborhoods of $0_{M_1}$.\\

It is with respect to this metric that we will perform, in the rest of this paper, any metric or topological operation on a filtered $R$-module. For instance a filtered module $M_{\bullet}$ will be said to be complete, if $M_1$, with the above metric, is a complete metric space. There is a completion functor from $R\text{-filt-mod}$ to the full subcategory whose objects are complete filtered modules, $\text{Compl-$R$-filt-mod}$, which consists simply of completing the underlying metric space. We denote this functor by \ $\widehat{•}$\ . It is left adjoint to the inclusion functor $\text{Compl-$R$-filt-mod} \subseteq R\text{-filt-mod} $  which is the identity on both objects and morphisms. Thus one has a natural transformation of the identity, which we denote by $\text{compl}:\text{id}_{R\text{-filt-mod}} \to \widehat{•}$\ . This natural transformation consists of the natural inclusion of a filtered module $M_{\bullet}$ in its completion, which we denote by $\widehat{M_{\bullet}}$.
\subsubsection{Sub-modules}
If $(M,w)$ is a filtered $R$-module, and $N \subseteq M$ an $R$-sub-module of $M$, then $(N,w_{|N})$ is a filtered $R$-module. If the filtration for $M$ is $M_1 \supseteq M_2 \supseteq \ldots \supseteq M_i \supseteq \ldots$, the one for $N$ is $N \cap M_1 \supseteq N \cap M_2 \supseteq \ldots \supseteq N \cap M_i \supseteq \ldots$. It is in this sense that we will speak of a filtered $R$-sub-module.
\subsubsection{Quotients} \label{quotients}
Let $M_{\bullet}$ be a filtered $R$-module and $N \subseteq M_1$ an $R$-sub-module of $M$. Then the filtration $M_1/N=(M_1+N)/N \supseteq (M_2+N)/N \supseteq \ldots \supseteq (M_i+N)/N \supseteq \ldots$, gives to $M_1/N$ the structure of a filtered $R$-module if and only if $N$ is closed. Indeed this filtration defines a fundamental system of neighbours of $0_{M_1/N}$ corresponding to the quotient topology coming from $M_1$: the requirement of being a filtered module is equivalent to the requirement that this topology is Hausdorff, and the quotient of a topological group by a normal subgroup is Hausdorff iff the normal subgroup is closed, since a topological group is Hausdorff iff the origin is closed. \\

We now introduce the functors which will play an important role in the rest of the section.
\begin{definition} \label{functors}
(a) Let $M_{\bullet}, N_{\bullet}$ be two filtered $R$-modules, and $i,j$ two positive integers with $i \leq j$. Denote by $F_{i,j}(M_{\bullet}):=M_i/M_{j}$. Given a morphism of filtered $R$-modules $\phi:M_{\bullet} \to N_{\bullet}$, denote by $F_{i,j}(\phi)$, the induced morphism $F_{i,j}(\phi): M_i/M_{j} \to N_i/N_{j}$.
Denote by $F_{i,j}$ the functor, $F_{i,j}:\text{Filt-}R\text{-mod} \to R\text{-mod}$, obtained in this way. Denote by $F_i$ the functor $F_{i,i+1}$.
\end{definition}
The rest of this section describes the relations between a morphism $\phi:M_{\bullet} \to N_{\bullet}$ of filtered $R$-modules and the sequence of morphisms $\{F_j(\phi):F_j(M_{\bullet}) \to F_j(N_{\bullet})\}_{j \in \mathbb{Z}_{\geq 1}}$ of $R$-modules. 
We begin by describing the effect of $F_j$ on the completion morphism:
\begin{remark}{\label{F_i bar}}
For every positive integer $i$, the natural transformation $\text{compl}$ induces an isomorphism of functors $F_i \circ \widehat{•}  \simeq_{\text{functors}}F_i$. 
\end{remark}
Next we determine basic properties when applying $F_j$ to the inclusion of the direct sum in the direct product.
\subsubsection{More on direct sum and direct product}
\begin{remark}{\label{F_j preserve}}
For each positive integer $j$  and  $\{(M_i,w_i)\}_{i \in I}$ any collection of filtered $R$-modules,  we have that \\
\\
$\bullet$ $F_j(\prod_{i \in I}M_i)=\prod_{i \in I}F_j(M_i)$ \\
$\bullet$ $F_j(\bigoplus_{i \in I}M_i)=\bigoplus_{i \in I}F_j(M_i)$ \\
$\bullet$ $F_j(\bigoplus_{i \in I}M_i\subseteq\prod_{i \in I}M_i)=(\bigoplus_{i \in I}F_j(M_i)\subseteq\prod_{i \in I}F_j(M_i))$, where in both cases we mean the natural inclusion of the direct sum in the direct product.
\end{remark}
\
 
\begin{proposition}{\label{products}}
Given $\{M_{i,\bullet} \}_{i \in I}$ any collection of $R$-filtered modules, the following are equivalent: \\
\emph{(a)} The inclusion of filtered modules $\bigoplus_{i \in I}M_{i,\bullet} \subseteq \prod_{i \in I}M_{i,\bullet}$ induces a dense inclusion of metric spaces.  \\
\emph{(b)} For each $ m \in \mathbb{Z}_{\geq 1}$ there are only finitely many $i \in I$ such that $\emph{min}(w_{M_{i,\bullet}}(M_{i,1})) \leq m$.  \\
\emph{(c)} We have that $F_m(\bigoplus_{i \in I}M_{i,\bullet} \subseteq \prod_{i \in I}M_{i,\bullet})$  is an isomorphism for all $m \in \mathbb{Z}_{\geq 1}$.
\begin{proof}
$\text{(a)} \to \text{(b)}$ Fix $m \in \mathbb{Z}_{\geq 1}$. Pick a vector $v=(v_i)_{i \in I} \in \prod_{i \in I}M_{i,1}$ such that, for all $i \in I$, $v_i=0$ or $w_{M_{i,\bullet}}(v_i) \leq m$ holds. By assumption we can find a finite subset, $J$, of $I$, and a vector $(y_i)_{i \in I} \in \prod_{i \in I}M_{i,1}$, such that $y_i=0$ if $i \not \in J$ and $(w_{\prod_{i \in I}M_{i,\bullet}})(v_i-y_i)_{i \in I}>m$. It follows that for all $i \not \in J$, $w_{M_{i,\bullet}}(v_i)>m$. Thus for every $v=(v_i)_{i \in I} \in \prod_{i \in I}M_i$, $w_{M_{i,\bullet}}(v_i)\leq m$ holds for only finitely many $i \in I$, that is $\text{min}(w_{M_{i,\bullet}}(M_i))\leq m$ holds for only finitely many $i \in I$. 

$\text{(b)} \to \text{(a)}$ Observe that assumption $\text{(b)}$ implies that $M_i=0$ holds for all but countably many $i\in I$: indeed, by assumption, the function $M_{i,\bullet} \to \text{min}(w(M_{i,1}))$ has finite fiber over every positive integer, so, except for a countable set of indices, $w_{M_{i,\bullet}}(M_{i,1})=\{\infty\}$ holds, which is equivalent (by definition of filtered module) to $M_{i,1}=0$ for all but countably many indices. So we can assume that $I=\mathbb{Z}_{\geq 1}$. Thus fix $v:=(v_n)_{n \in \mathbb{Z}_{\geq 1}} \in \prod_{i \in \mathbb{Z}_{\geq 1}}M_{i,1}$. Consider the sequence $\{h_l\}_{l \in \mathbb{Z}_{\geq 1}}:=\{(w_{l,i})_{i \in \mathbb{Z}_{\geq 1}}\}_{l \in \mathbb{Z}_{\geq 1}}$,  
 where $w_{l,i}=v_i$ if $i \leq l $, $0$ otherwise. One has that for all $m \in \mathbb{Z}_{\geq 1}$, $(w_{\prod_{i \in \mathbb{Z}_{\geq 1}}M_i})(v-w_l)>m$, holds for all but finitely many values of $l$. This means exactly that $h_l \to v$ as $l \to \infty$. Thus the inclusion of filtered modules $(\bigoplus_{i \in I}M_i,d_{\bigoplus_{i \in I}M_{i,\bullet}}) \subseteq (\prod_{i \in I}M_i,d_{\prod_{i \in I}M_{i,\bullet}})$ induces a dense inclusion of metric spaces. For the equivalence between $\text{(b)}$ and $\text{(c)}$ see Remark \ref{equivalence between a and c}.
\end{proof}
\end{proposition}

Finally we look at the relation between injectivity/surjectivity of $\phi$ and the pointwise injectivity/surjectivity of the sequence $\{F_j(\phi)\}_{j \in \mathbb{Z}_{\geq 1}}$:
\subsubsection{Surjectivity and injectivity}
\begin{proposition}{\label{phi surj}}
Let $M_{\bullet},N_{\bullet}$ be two filtered modules, and $\phi \in \emph{Hom}_{\emph{filt}}(M_{\bullet},N_{\bullet})$. Then the following holds: \\
\emph{(a)} Assume $M_{\bullet}$ complete. If for all $i \in \mathbb{Z}_{\geq 1}$ we have that $\emph{coker}(F_i(\phi))=0$, then $\emph{coker}(\phi)=0$. \\
\emph{(b)} We have that for all $i \in \mathbb{Z}_{\geq 1}$ the module $\emph{ker}(F_i(\phi))$ is $0$ if and only if for all $x \in M_1$ the weights $w_{M_{\bullet}}(x)$ and $w_{N_{\bullet}}(\phi(x))$ coincide. \\
\emph{(c)} If for all $i \in \mathbb{Z}_{\geq 1}$ we have that $\emph{ker}(F_i(\phi))=0$, then $\emph{ker}(\phi)=0$. \\
\emph{(d)} If $\phi$ is an isomorphism then for all $i \in \mathbb{Z}_{\geq 1}$ the map $F_i(\phi)$ is an isomorphism. If $M_{\bullet}$ is complete, the converse holds as well. \\
\begin{proof}
(a) Let $x \in N_1$. We construct inductively sequences $\{x_n\}_{n \in \mathbb{Z}_{\geq 0}},\{y_n\}_{n \in \mathbb{Z}_{\geq 0}}$ respectively $N_1,M_1$-valued, which will do for us the following: $\{\sum_{i=0}^{n}y_i\}_{n \in \mathbb{Z}_{\geq 1}}$ will be a convergent sequence, with $\phi(\sum_{i=0}^{n}y_i)-x=x_{n+1}$, with $\text{lim}_{n \to \infty} x_n=0$. Since $\phi$ is a filtered morphism and in particular continuous, and $M_{\bullet}$ is complete, we can conclude then that $\phi(\sum_{i=0}^{\infty}y_i)=x$. The construction of $\{x_n\}_{n \in \mathbb{Z}_{\geq 0}},\{y_n\}_{n \in \mathbb{Z}_{\geq 0}}$ goes as follows. Put $x_0=x,y_0=0$; construct $x_{n+1},y_{n+1}$ from $x_{n}$ in the following way. If $x_n=0$ put $x_{n+1}=y_{n+1}=0$. Otherwise $w_{N_{\bullet}}(x_n) \in \mathbb{Z}_{\geq 1}$ holds. Since the map $F_{w_{N_{\bullet}}(x_n)}(\phi)$ is surjective, pick $y \in M_{w_{N_{\bullet}}(x_n)}$ such that $(\phi)(y) \equiv x_n \ \text{mod} \ N_{w_{N_{\bullet}}(x_n)+1}$, and denote $y_{n+1}=y$ and $x_{n+1}=-\phi(y)+x_{n}$. By construction, the sequences $\{x_n\}_{n \in \mathbb{Z}_{\geq 0}},\{y_n\}_{n \in \mathbb{Z}_{\geq 0}}$ both converge to $0$. So by the ultrametric inequality and completeness of $M_{\bullet}$ the series $\sum_{n \in \mathbb{Z}_{\geq 0}}y_n$ converges to an element of $M_1$, which we denote by $\bar{y}$. By construction $\phi(\sum_{1\leq j \leq n}y_j)-x=x_{n+1} \to 0$, so, since $\phi$ is continuous, $\phi(\bar{y})=x$. So $\text{coker}(\phi)=0$. 

(b) By definition $M_i-M_{i+1}=\{x \in M_i, w_{M_{\bullet}}(x)=i \}$, on the other hand $\text{ker}(\phi)_{i}=0$ iff $\phi(M_i-M_{i+1})\subseteq N_i-N_{i+1}=\{y \in N_i, w_{M_{\bullet}}(y)=i\}$, thus $\text{ker}(\phi)_i=0$ for all $i \in \mathbb{Z}_{\geq 1}$ iff $w_{M_{\bullet}}(x)=w_{N_{\bullet}}(\phi(x))$ for all $x \in M_1$. 

(c) Thanks to (b) the hypothesis in (c) is equivalent to $\phi(M_i-M_{i+1})\subseteq N_i-N_{i+1}$, which implies that $\text{ker}(\phi)\subseteq \bigcap_{i \in \mathbb{Z}_{\geq 1}}M_i=\{0\}$. 

(d) The first implication follows from the general fact that a functor preserves isomorphisms, applied to the functors $F_i$. For the second implication: assume $M_{\bullet}$ complete, then (a) implies that  $\phi$ is surjective. On the other hand (c) implies that $\phi$ is also injective. Thus $\phi$ is a filtered isomorphism.
\end{proof}
\end{proposition}
\begin{remark}{\label{phi_1}}
Suppose $\phi:M_{\bullet} \to N_{\bullet}$ is a filtered epimorphism. Then $F_1(\phi)$ is surjective. Indeed by definition of filtered epimorphism, and the fact that $1$ is minimal in $\mathbb{Z}_{\geq 1}$ we have ${\phi}^{-1}(N_1-N_2)\subseteq M_1-M_2$: since $\phi$ is surjective, applying $\phi$ to both sizes of this relation one gets $N_1-N_2 \subseteq \phi(M_1-M_2)$, which proves that $F_1(\phi)$ is surjective. 
\end{remark}
\begin{definition} \label{definition of shifted filtered modules}
Let $i$ be a positive integer and let $M_{\bullet}$  be a filtered $R$-module. We define $M_{\bullet+i}$ to be the filtered $R$-module
$$M_{i+1} \supseteq M_{i+2} \supseteq \ldots
$$
\end{definition}
\begin{proposition}{\label{first non iso is epi}}
Let $M_{\bullet},N_{\bullet}$ be two filtered modules. Let $\phi:M_{\bullet} \to N_{\bullet}$ be a filtered epimorphism. Let $i$ be a positive integer such that $F_j(\phi)$ is an isomorphism for every $j$ such that $1 \leq j \leq i$. Then $\phi_{|M_{\bullet+i}}:M_{\bullet+i} \to N_{\bullet+i}$ is a filtered epimorphism and $F_{i+1}(\phi)$ is surjective. 
\begin{proof}
Indeed, by Proposition \ref{phi surj}, the hypothesis is equivalent to $F_{1,i+1}(\phi)$ being a filtered isomorphism. Thus $\phi(M_1-M_{i+1}) \subseteq N_1-N_{i+1}$. Thus, since $\phi$ is an epimorphism, it follows that $\phi(M_{i+1})=N_{i+1}$, in particular by remark \ref{phi_1} we have that $F_{i+1}(\phi)=F_1(\phi_{|M_{\bullet+i}})$ is surjective, proving the statement. 
\end{proof}
\end{proposition}
\begin{proposition}{\label{a lot surjective}}
Let $M_{\bullet},N_{\bullet}$ be two filtered modules with $M_{\bullet}$ complete. Let $\phi$ be an element of $\emph{Hom}_{\emph{filt}}(M_{\bullet},N_{\bullet})$. The following are equivalent: \\
\emph{(a)} For every positive integer $i$, we have that $\emph{coker}(F_i(\phi))=0$. \\
\emph{(b)} For every positive integer $i$, we have that $\emph{coker}(\phi_{|M_i}:M_i \to N_i)=0$. 
\begin{proof}
$\text{(a)} \to \text{(b)}$ Let $i$ be a positive integer. For a positive integer $j>i$, the equality $F_j(\phi_{|M_{\bullet+i}})=F_{i+j-1}(\phi)$ trivially holds. Thus assumption $(a)$ is preserved by restriction of $\phi$ to the filtered submodule $M_{\bullet+i}$. So Proposition \ref{phi surj} implies that $\text{coker}(\phi_{|M_i}:M_i \to N_i)=0$. 

$\text{(b)} \to \text{(a)}$ The statement trivially follows applying remark \ref{phi_1} to every filtered morphism $\phi_{|M_i}:M_{\bullet +i} \to N_{\bullet+i}$ since they are all assumed to be epimorphisms. 
\end{proof}
\end{proposition}
\begin{proposition}{\label{last ker}}
Let $M_{\bullet},N_{\bullet}$ be two filtered modules, $M_{\bullet}$ complete, and $\phi \in \emph{Hom}_{\emph{filt}}(M_{\bullet},N_{\bullet})$. Assume $i \in \mathbb{Z}_{\geq 1}$ is such that $\emph{ker}(F_j(\phi))=\emph{coker}(F_j(\phi))=0$ for all $j>i$. Then $\emph{ker}(\phi) \cap w_{M_{\bullet}}^{-1}\{i,\infty\}$ is an $R$-submodule, and the inclusion in $M_i$ induces an isomorphism $\emph{ker}(\phi) \cap w_{M_{\cdot}}^{-1}\{i,\infty\} \simeq \emph{ker}(F_i(\phi))$.
\begin{proof}
Since $\text{ker}(F_j(\phi))=0$ for all $j>i$, it follows that $\text{ker}(\phi) \cap M_i=\text{ker}(\phi) \cap w_{M_{\cdot}}^{-1}\{i,\infty\}$ proving thus that is an $R$-submodule, and that the inclusion in $F_i(M_{\bullet})$ is injective. Suppose $x \in M_i-M_{i+1}$, $\phi(x) \in N_{i+1}$ holds. Thanks to the assumption $\text{ker}(F_j(\phi))=\text{coker}(F_j(\phi))=0$ for all $j>i$, and to Proposition \ref{phi surj}, we see that $\phi_{|M_{\bullet+i+1}}$ is an isomorphism and thus it follows that there is exactly one $y \in M_{i+1}$ such that $\phi(x)=\phi(y)$. Thus, since $x \equiv x-y \ \text{mod} \ M_{i+1}$, and $x-y \in \text{ker}(\phi)$ we obtain that the natural map from $\text{ker}(\phi) \cap w_{M_{\bullet}}^{-1}\{i,\infty\}$ to $F_i(M_{\bullet})$ is also surjective.
\end{proof}
\end{proposition}
\begin{corollary}{\label{last ker +}}
Let $M_{\bullet},N_{\bullet}$ be two filtered modules, $M_{\bullet}$ complete, and $\phi \in \emph{Hom}_{\emph{filt}}(M_{\bullet},N_{\bullet})$. Assume $i \in \mathbb{Z}_{\geq 1}$ is such that $\emph{coker}(F_j(\phi))=0$ for all $j>i$ and $\emph{ker}(F_j(\phi))=0$ for all $j \neq i$. Then $\emph{ker}(\phi) \subseteq w_{M_{\cdot}}^{-1}\{i,\infty\}$, and this inclusion induces an isomorphism $\emph{ker}(\phi) \simeq_{R\emph{-mod}}\emph{ker}(F_i(\phi))$.  
\begin{proof}
Clearly the assumption that $\text{ker}(F_j(\phi))=0$ for all $j \neq i$ implies that $\text{ker}(\phi) \subseteq w_{M_{\bullet}}^{-1}\{i,\infty\}$. Thus lemma \ref{last ker} implies that this inclusion induces an isomorphism $$\text{ker}(\phi)=\text{ker}(\phi) \cap w_{M_{\bullet}}^{-1}\{i,\infty\} \simeq_{R\text{-mod}}\text{ker}(F_i(\phi)).$$
\end{proof}
\end{corollary}
\begin{remark}
Part (a),(c) of Proposition \ref{phi surj} do not hold without the assumption of completeness. An example is given as follows: take a collection of filtered modules $\{(M_i,w_{M_i})\}_{i \in \mathbb{Z}_{\geq 1}}$ such that for all $ m \in \mathbb{Z}_{\geq 1}$ there are only finitely many $i \in I$ such that $\text{min}(w_i(M_i)) \leq m$. Now consider $\bigoplus_{i \in I}(M_i,w_i)\subseteq \prod_{i \in I}(M_i,w_i)$. Then $F_m(M_i)=0$ for all but finitely many $i$. Thus, by remark \ref{F_j preserve}, we have that the inclusion of the direct sum of the direct product is preserved by $F_m$, but since it is over a finite set of indices (the ones where $F_m$ does not vanish) it is also an isomorphism. But if $M_i \neq 0$ for infinitely many $i \in \mathbb{Z}_{\geq 1}$ the inclusion of the direct sum in the direct product is not an isomorphism. This suggests the following proposition.
\end{remark}
\begin{proposition}{\label{phi surj comple}}
Let $M_{\bullet},N_{\bullet}$ be two filtered modules, and $\phi \in \emph{Hom}_{\emph{filt}}(M_{\bullet},N_{\bullet})$, denote by $\hat{\phi}:\hat{M} \to \hat{N}$ the map induced on the completions. Then the following hold: \\
\emph{(a)} If $\emph{coker}(F_i(\phi))=0$ for all $i \in \mathbb{Z}_{\geq 1}$, then $\emph{coker}(\hat{\phi})=0$. \\
\emph{(b)} If $\emph{ker}(F_i(\phi))=0$ for all $i \in \mathbb{Z}_{\geq 1}$, then $\emph{ker}(\hat{\phi})=0$. \\
\emph{(c)} $F_i(\phi)$ is an isomorphism for every $i \in \mathbb{Z}_{\geq 1}$ iff $\hat{\phi}$ is an isomorphism.
\end{proposition}
\begin{proof}
From remark \ref{F_i bar}, we know that $F_i$ and $F_i(\text{compl})$ are isomorphic functors. Thus $\text{coker}F_i(\phi)=0$ for all $i \in \mathbb{Z}_{\geq 1}$ is equivalent to $\text{coker}F_i(\hat{\phi})=0$ for all $i \in \mathbb{Z}_{\geq 1}$, and $\text{ker}(F_i(\phi))=0$ for all $i \in \mathbb{Z}_{\geq 1}$, is equivalent to $\text{ker}(F_i(\hat{\phi}))=0$ for all $i \in \mathbb{Z}_{\geq 1}$. Thus the proposition follows from Proposition \ref{phi surj}.
\end{proof}
\begin{remark}{\label{equivalence between a and c}}
Proposition \ref{phi surj comple} implies the equivalence between (b) and (c) in Proposition \ref{products}. Indeed if we have (c) of Proposition \ref{products} then we conclude that the completion of $\prod_{i \in I} M_i$ is also the completion of $\bigoplus_{i \in I} M_i$. Hence in particular $ \bigoplus_{i \in I} M_i$ is dense in $ \prod_{i \in I}M_i$. This gives that (c) implies (a). But we have shown in Proposition \ref{products} that (a) is equivalent to (b), hence (c) implies (b). Conversely it is an immediate verification that (b) implies (c). 
\end{remark}
\subsection{Filtered modules over a complete DVR} \label{DVR}
Now we specialize to the case where $R$ is a complete DVR: we ask completeness because in what follows, we want to apply Propositions \ref{phi surj}, \ref{a lot surjective}, \ref{last ker +}, and moreover it will be handy when taking filtered quotients of finitely generated modules (see \ref{quotients}). We fix a uniformizer of $R$, and we denote it by $\pi_R$. 
\subsubsection{The $\rho$-map} \label{rho} 
Let $M_{\bullet}$ a filtered $R$-module, denote by $w$ its weight map. Define $\rho_{M_{\bullet}}:\mathbb{Z}_{\geq 1} \to \mathbb{Z}_{\geq 1} \cup \{\infty\}$ as follows: $\rho_{M_{\bullet}}(i):=\text{sup}\{j \in \mathbb{Z}_{\geq 1}, \pi_RM_{i} \subseteq M_j\}$. In terms of the weight map we have that $\rho_{M_{\bullet}}(i)=\min_{x \in M_i} \{w(\pi_Rx)\}$.
\begin{remark} \label{property of a linear module}
The condition that $\rho_{M_{\bullet}}$ is a shift map is equivalent to the conjunction of the following two conditions: \\
(a) For all positive integers $i$, one has that $M_i/M_{i+1}$ is an $R/(\pi_R)$-vector space. Moreover $\pi_RM_i \neq 0$. \\
(b) For all positive integers $i$ the $R$-linear map ${\pi_R}_{|M_{i}}:M_i \to M_{\rho_{M_{\bullet}}(i)}$, given by multiplication by $\pi_R$, is a filtered morphism. 
\end{remark}
\begin{definition} \label{definition of linear}
We call a filtered $R$-module \emph{linear} if it satisfies (a) of remark \ref{property of a linear module}. We call a filtered $R$-module \emph{strictly linear} if it satisfies both part (a) and part (b) of remark \ref{property of a linear module}.
\end{definition}
Let $M_{\bullet}$ be a linear filtered $R$-module. Multiplication by $\pi_R$ induces a map $F_i(M_{\bullet}) \to F_{\rho_{M_{\bullet}}(i)}(M_{\bullet})$, which we denote by $[\pi_R]_i$. One has by definition that $[\pi_R]_i=F_1({\pi_R}_{|{M_{i}}})$. Observe that the right hand side is well defined thanks to part (b) of the definition of a linear filtered $R$-module. 
\begin{definition} \label{definition: f, defect, codefect}
Let $M_{\bullet}$ be a linear $R$-filtered module and let $i$ be a positive integer. \\
(a) We denote by $f_i(M_{\bullet})=\text{dim}_{R/(\pi_R)}(F_i(M_{\bullet}))$. \\
(b) We denote by $\text{defect}_{M_{\bullet}}(i):=\text{dim}_{R/(\pi_R)}(\text{ker}([\pi_R]_i))$. \\
(c) We denote by $\text{codefect}_{M_{\bullet}}(i):=\text{dim}_{R/(\pi_R)}(\text{coker}([\pi_R]_i))$. 
\end{definition}

\subsubsection{Free filtered modules} \label{(quasi)-free modules}
Fix $\rho$ a shift map. Here we introduce the class of free filtered $R$-modules with respect to $\rho$. Free filtered modules play a role in the category of filtered modules similar to the one played by free $R$-modules in the category of $R$-modules. We thus recall the role of the latter to clarify the introduction of the former. 

\emph{Free $R$-modules.} \label{free R module}
Recall that if $X$ is a set, then we have a covariant functor $H_{X}:R\text{-mod} \to \text{Set}$, defined on an object $M \in R\text{-mod}$ as $H_{X}(M):=\text{Hom}_{\text{Set}}(X,M)$, and defined on a morphism $\phi:M \to N$ as $H_X(\phi)(f):=\phi \circ f$ for each $f \in \text{Hom}_{\text{Set}}(X,M)$. In other words $H_{X}$ is the restriction of the functor $\text{Hom}_{\text{Set}}(X,-)$ to the image of $R\text{-mod}$ in $\text{Set}$ via the forgetful functor. This functor is representable in $R\text{-mod}$:  up to isomorphism there is a unique $R$ module, $N_{X}$, such that $H_X \simeq_{\text{functor}} \text{Hom}_{R\text{-mod}}(N_{X},-)$. This module is called the free module over $X$, and concretely it is the module of finite formal $R$-linear combinations of elements of $X$. By Yoneda's Lemma the different choices of an isomorphism $\Phi: \text{Hom}_{R\text{-mod}}(N_X,-) \to H_X$, correspond to the different choices of ${\Phi}_{N_X}(\text{id}_{N_X}):X \to N_{X}$, which are the different choices of a basis $\mathcal{B}$ for $N_X$ together with a bijection between $\mathcal{B}$ and $X$. Again, by Yoneda's Lemma, the set $\text{Isom}_{\text{functors}}(H_X,N_X)$ is a torsor under $\text{Aut}_{R\text{-mod}}(N_X)$. 

Free $R$-modules are the easiest $R$-modules, and once we trivialize $\text{Isom}_{\text{functors}}(H_X,N_X)$, by the choice of a basis $\Phi$, then, by construction, for any $R$-module $M$, the set $\text{Hom}_{R\text{-mod}}(N_X,M)$ is in natural bijection with $\text{Hom}_{\text{Set}}(X,M)$, via $\Phi$. Thus we can easily use suitable free $R$-modules to present other modules. The ease in defining presentations $N_X \twoheadrightarrow M$, once a trivialization $\Phi$ is chosen, has the price of obscuring structural information about $M$. Thus one is led to look for properties of the presentation which are invariant under $\text{Aut}_{R\text{-mod}}(N_X)$. This is exactly the path we will follow in attaching jump sets to special filtered modules. So, first, we need to define the analogue of a free filtered module, which we do next.

\emph{Free filtered $R$-modules.} \label{shit}
First we introduce the analogue of the functors $H_X$ of the previous paragraph. Consider pairs $(X,g)$, where $X$ is a set and $g$ is a map $g:X \to \mathbb{Z}_{\geq 1}$. Denote by $\rho\text{-Filt-}R\text{-mod}$ the full sub-category of $\text{Compl-Filt-}R\text{-mod}$, having as objects complete linear $R$-filtered modules $M_{\bullet}$ such that $\rho_{M_{\bullet}} \geq \rho$. Consider the functor $H_{(X,g)}:\rho\text{-Filt-}R\text{-mod} \to \text{Set}$, defined on an object $M_{\bullet} \in \rho\text{-Filt-}R\text{-mod}$ as $H_{(X,g)}(M_{\bullet}):=\{f \in \text{Hom}_{\text{Set}}(X,M_1): \text{for all $x$ in $X$,} \ w(f(x)) \geq g(x) \}$, and defined on morphisms by left composition. The goal of this paragraph is show that this functor is representable. We start with the simplest possible case of a pair $(X,g)$ with $X=\{x\}$ being a point. Put $n:=g(x)$. Clearly the functor depends only on $n$, so, for simplicity, we will denote it by $H_n$. 
\begin{definition} 
The $n$-th standard filtered module, $S_n$, for $\rho$, is given by: $S_n=R$, with weight map defined as $w(x)=\rho^{\text{ord}_R(x)}(n)$, for all $x$ in $R$. 
\end{definition}
Observe that $S_n$ is an object of $\rho\text{-Filt-}R\text{-mod}$ (recall that $R$ is assumed complete). It turns out that it represents $H_n$.
\begin{proposition} \label{small universal property}
The functor $H_n$ is represented by $S_n$. 
\begin{proof}
Observe that by definition $H_n$ is simply the functor sending $M_{\bullet}$ to the set $M_n$, and sending a morphism $ \phi:M_{\bullet} \to N_{\bullet}$ to the restriction $\phi_{|_{M_n}}:M_n \to N_n$. So it suffices to prove that given $M_{\bullet} \in \rho\text{-Filt-}R\text{-mod}$, and given $v \in M_n$, the unique $R$-linear morphism from $R$ to $M_1$ sending $1 \mapsto v$, is a filtered morphism from $S_n$ to $M_{\bullet}$, and that these are all the possible filtered morphism from $S_n$ to $M_{\bullet}$. But this follows directly from the definition of $S_n$ and the fact that $M_{\bullet}$ is an object of $\rho\text{-Filt-}R\text{-mod}$.
\end{proof}
\end{proposition}
Now we can prove that $H_{(X,g)}$ is representable for any set $X$ and any map $g:X \to \mathbb{Z}_{\geq 1}$. For a positive integer $i$ denote by $c_{(X,g)}(i):=|g^{-1}(i)|$. Given $c$ a cardinal number and $N_{\bullet}$ a filtered module, denote by $N_{\bullet}^{(c)}$ the direct sum of $c$ copies of $N_{\bullet}$.
\begin{proposition} \label{Universal property of free filtered modules}
The functor $H_{(X,g)}$ is represented by the filtered $R$-module $\prod_{i \in \mathbb{Z}_{\geq 1}}\widehat{S_i^{(c_{(X,g)}(i))}}$.
\begin{proof}
The functor $H_{(X,g)}$ is isomorphic to the direct product of the functors $H_{g(x)}$ as $x$ varies in $X$. So it follows from Proposition \ref{products}, Claim \ref{small universal property} and the universal property of the completion, that $H_{(X,g)}$ is isomorphic to the functor $\text{Hom}_{\text{filt}}(\prod_{i \in \mathbb{Z}_{\geq 1}}\widehat{S_i^{(c_{(X,g)}(i))}},-)$.
\end{proof}
\end{proposition}
\begin{remark}
Let $i$ be a positive integer. If $c_{(X,g)}(i)$ \emph{finite}, then we can omit the completion of the factor $S_i^{(c_{(X,g)}(i))}$, since it is already a complete filtered module. In our application $c_{(X,g)}(i)$ will always be finite. 
\end{remark}
An object $M_{\bullet}$ in $\rho\text{-Filt-}R\text{-mod}$, representing $H_{(X,g)}$ (so by Yoneda's Lemma and by Proposition \ref{Universal property of free filtered modules}, isomorphic to $\prod_{i \in \mathbb{Z}_{\geq 1}}\widehat{S_i^{(c_{(X,g)}(i))}}$), is said to be free on $(X,g)$. Motivated by the discussion in the above paragraph on free modules, we introduce the following notion. 
\begin{definition}\label{definition of a basis}
Let $M_{\bullet}$ be in $\rho\text{-Filt-}R\text{-mod}$ a free module on $(X,g)$. A \emph{filtered} basis for $M_{\bullet}$ is an element of $\text{Isom}_{\text{functor}}(\text{Hom}_{\text{filt}}(M_{\bullet},-),H_{(X,g)})$. 
\end{definition}
Given $\Phi$ a filtered basis for $M_{\bullet}$, one recovers a more concrete version of the notion of a filtered basis, by means of Yoneda's Lemma, taking $\Phi_{M_{\bullet}}(\text{id}_{M_{\bullet}}):X \to M_1$. The image of this map generates a free $R$-module that is dense in $M_1$ (coinciding with $M_1$ if and only if $X$ is finite, observe that for $X$ infinite the resulting module is never free as an $R$-module). 

Clearly, the functor $H_{(X,g)}$ depends only on the map $c_{(X,g)}$. So from now on we will directly speak of the functors $H_{f^{*}}$, where $f^{*}$ is a map from $\mathbb{Z}_{\geq 1}$ to the cardinal numbers.

We next give an internal criterion for a filtered module to be representing the functor $H_{f^{*}}$, under the assumption that $f^{*}$ is supported in $T_{\rho}$, that is, we assume that $f^{*}(\text{Im}(\rho))=\{0\}$.
\begin{proposition} \label{charact. free-filt-mod}
Let $M_{\bullet}$ be an object of $\rho\emph{-Filt-}R\emph{-mod}$, and $f^{*}$ as above. Then the following are equivalent: \\
\emph{(a)} For every positive integer $i$ one has $\emph{defect}_{M_{\bullet}}(i)=\emph{codefect}_{M_{\bullet}}(i)=0$. Moreover if $i$ is in $T_{\rho}$, one has $f_i(M_{\bullet})=f^{*}(i)$. \\
\emph{(b)} One has an isomorphism of functors $H_{f^{*}} \simeq_{\emph{functor}}\emph{Hom}_{\emph{filt}}(M_{\bullet},-)$. \\
\emph{(c)} One has an isomorphism of filtered modules $M_{\bullet} \simeq_{\emph{filt}} \prod_{i \in \mathbb{Z}_{\geq 1}}\widehat{S_i^{f^{*}(i)}}$.
\begin{proof}
The equivalence between (b) and (c) is an immediate consequence of Proposition \ref{Universal property of free filtered modules} and Yoneda's Lemma. It is a straightforward verification that (c) implies (a). We prove that (a) implies (c). 

For every positive integer $i$ in $T_{\rho}$, lift a basis of $M_i/M_{i+1}$ to $M_i$ and denote it by $\mathcal{B}_i$. The inclusion $\bigcup_{i \in T_{\rho}} \mathcal{B}_i \subseteq M_1$ consists of an element of $H_{\tilde{f}}(M_{\bullet})$, which thus gives, thanks to Proposition \ref{Universal property of free filtered modules}, a filtered morphism $\phi: \prod_{i \in \mathbb{Z}_{\geq 1}} \widehat{S_i^{f^{*}(i)}} \to M_{\bullet}$. We claim that $\phi$ is an isomorphism. 

Indeed by construction $F_i(\phi)$ is an isomorphism for every $i$ in $T_{\rho}$. But together with the fact that for every positive integer $i$ one has $\text{defect}_{M_{\bullet}}(i)=\text{codefect}_{M_{\bullet}}(i)=0$, this easily implies that for every positive integer $i$, the map $F_i(\phi)$ is an isomorphism. So, since $M_{\bullet}$ is complete, we conclude by part (d) of Proposition \ref{phi surj}.
\end{proof}
\end{proposition} 
\begin{definition} \label{free modules def}
Let $M_{\bullet}$ be an object of $\rho\text{-Filt-}R\text{-mod}$, and $f$ a positive integer. Then we call $M_{\bullet}$ a $(f,\rho)$-\emph{free filtered module} if it satisfies any of the equivalent conditions of Proposition \ref{charact. free-filt-mod}, with respect to the constant map $T_{\rho} \to \mathbb{Z}_{\geq 1}$, $i \mapsto f$.
\end{definition}
We denote by $M_{\rho}:=\prod_{i \in T_{\rho}}S_i$, i.e. the $(1,\rho)$-free filtered module. So $M_{\rho}^f$ is the $(f,\rho)$-free filtered module.

We next introduce the class of filtered modules that, together with those described in this paragraph, will suffice to classify the possible filtered structures of $U_1$.
\subsubsection{Quasi-free filtered $R$-modules}
Recall that in case $\rho$ is a shift with $\#T_{\rho}< \infty$, then we denote by $e_{\rho}^{*}=\text{max}(T_{\rho})+1$. Moreover we define $e_{\rho}^{'}$ to be the unique positive integer such that $\rho(e_{\rho}')=e_{\rho}^{*}$.
\begin{definition} \label{def of quasi free}
Let $M_{\bullet}$ be an object of $\rho\text{-Filt-}R\text{-mod}$. Then we call it $(f,\rho)$-quasi-free if it satisfies the following three conditions: \\
(a) For every positive integer $i$, we have that $f_i(M_{\bullet})=f$. \newline
(b) If $T_{\rho}$ is finite (resp.\ if $T_{\rho}$ is not finite), for every positive integer $i$ different from $e_{\rho}^{'}$ (resp.\ for every positive integer $i$), one has $\text{defect}_{M_{\bullet}}(i)=\text{codefect}_{M_{\bullet}}(i)=0$. \newline
(c) If $T_{\rho}$ is finite one has that $\text{defect}_{M_{\bullet}}(e'_{\rho}) \leq 1$.
\end{definition}
So we see that if $T_{\rho}$ is not finite the notion of a $(f,\rho)$-quasi-free module coincides with the notion of a $(f,\rho)$-free module. We characterize this distinction with a module-theoretic property: 
\begin{proposition}{\label{finite module cofinite map}}
Let $M_{\bullet}$ be a $(f,\rho)\text{-quasi-free}$ filtered module. Then the following are equivalent: \\
\emph{(a)} $T_{\rho}$ is finite, \\
\emph{(b)} $M_1$ is finitely generated. 
\begin{proof}
$\text{(a)} \to \text{(b)}$ Since all the $F_i(M_{\bullet})$ are finite dimensional, (b) is equivalent to the statement that for some positive integer $n$, the $R$-module $M_n$ is finitely generated. But for $n>e'_{\rho}$, the filtered $R$-module $M_{\bullet+n}$ is a $(f,\tau_{|T_{\rho}|})$-free-module, where for a positive integer $m$, the symbol $\tau_m$ denotes the shift sending any positive integer $x$ to $x+m$. So one concludes by Proposition \ref{charact. free-filt-mod}.

$\text{(b)} \to \text{(a)}$ Since $M_1$ is finitely generated, so is $M_n$. But for $n>e'_{\rho}$, one has that $M_{\bullet+n}$ is a $(f,\rho \circ \tau_{n-1})$-free module. So by Proposition \ref{charact. free-filt-mod} one has that $T_{\rho \circ \tau_{n-1}}$ is finite, which is equivalent to say that $T_{\rho}$ is finite.
\end{proof}
\end{proposition}
Until the end of the next paragraph, we will restrict to the case that $T_{\rho}$ is finite or equivalently that $M_1$ is finitely generated. We will work again in greater generality only from Section \ref{transitive} onward. 

We now recover the distinction between $(f,\rho)$-quasi-free and $(f,\rho)$-free with a module-theoretic property.
\begin{proposition} \label{quasi free via torsion}
Let $M_{\bullet}$ be a $\rho \text{-}\emph{Filt-}R \text{-}\emph{Mod}$ such that $f_i(M_{\bullet})=f$ for every positive integer $i$, and $\emph{defect}_{M_{\bullet}}(j)=\emph{codefect}_{M_{\bullet}}(j)=0$ for every positive integer $j \neq e'_{\rho}$. Then the following are equivalent: \\
\emph{(a)} $M_{\bullet}$ is $(f,\rho)$-quasi-free. \\
\emph{(b)} $M_{1}[\pi_R] $ is a cyclic $R$-module.
\begin{proof}
Given the hypothesis we have to prove that $\text{defect}_{M_{\bullet}}(e'_{\rho}) \leq 1$ is equivalent to $M[\pi_R]$ cyclic. One has that multiplication by $\pi_R$ is a filtered morphism $\pi_R:M_{\bullet+e'_{\rho}-1} \to M_{\bullet+e_{\rho}^{*}-1}$. Thus the conclusion follows immediately from Corollary \ref{last ker +}.
\end{proof}
\end{proposition}
In particular we have the following.
\begin{corollary} \label{where is the torsion}
Let $M_{\bullet}$ be a $(f,\rho)$-quasi-free module which is not $(f,\rho)$-free. Then we have an isomorphism $M[\pi_R] \simeq_{R\emph{-mod}} R/\pi_R$ and $w_{M_{\bullet}}(M[\pi_R])=\{e'_{\rho},\infty \}$. 
\end{corollary}
\subsubsection{Presentations of a quasi-free modules are conjugate} \label{transitive}
We keep assuming that $T_{\rho}$ is finite. Let $f$ be a positive integer. We will proceed classifying $(f,\rho)$-quasi-free modules with the help of the additional free module $M_{\rho}^{*}:=M_{\rho} \oplus S_{e_{\rho}^{*}}$: we will use the module $M_{\rho}^{f-1} \oplus M_{\rho}^{*}$. This module is the free module over the map $f^{*}:\mathbb{Z}_{\geq 1} \to \mathbb{Z}_{\geq 1}$ defined as $f^{*}(i)=f$ for $i \in T_{\rho}$, $f^{*}(e_{\rho}^{*})=1$ and $f^{*}(i)=0$ for all the other $i$. So we fix an isomorphism between $M_{\rho}^{f-1} \oplus M_{\rho}^{*}$ and $H_{f^{*}}$, that is we fix a filtered basis for $M_{\rho}^{f-1} \oplus M_{\rho}^{*}$. Let now $M_{\bullet}$ be a $(f,\rho)$-quasi-free module that is not $(f,\rho)$-free. We call a subset $\mathcal{B} \subseteq M_1$ a quasi-basis if it consists of the union of the lifting of a basis of $M_{i}/M_{i+1}$ for each $i \in T_{\rho}$ together with the lifting of a generator of $\text{coker}[\pi_R]_{{e'_{\rho}}}$ (this cokernel is 1-dimensional because the kernel is 1-dimensional and we assume that $f_i(M_{\bullet})$ is constantly $f$). By the universal property proved in Proposition \ref{Universal property of free filtered modules}, each inclusion of a quasi-basis $\mathcal{B} \subseteq M_1$ gives uniquely (via the above choice of a filtered basis for $M_{\rho}^{f-1} \oplus M_{\rho}^{*}$) a morphism $\phi_{\mathcal{B}}:M_{\rho}^{f-1} \oplus M_{\rho}^{*} \to M_{\bullet}$.
\begin{proposition} \label{presentations}
For each quasi-basis $\mathcal{B}$, one has that $\phi_{\mathcal{B}}$ is a filtered epimorphism. 
\begin{proof}
By construction for each $i \in T_{\rho}^{*}$, one has that $F_i(\phi_{\mathcal{B}})$ is surjective. But since for both modules one has that $[\pi_R]_i$ is surjective for $i$ different from $e'_{\rho}$, and at $e_{\rho}^{*}$ a generator of the co-kernel has been added, one clearly concludes that $F_i(\phi_{\mathcal{B}})$ is surjective for all $i$, by repeatedly using the above conditions and the multiplication by $\pi_R$. Since $M_{\bullet}$ is complete, we conclude with Proposition \ref{phi surj}. 
\end{proof}
\end{proposition} 
We have found for $M_{\bullet}$ presentations with the easiest possible type of filtered module with the given constraints (namely those on the $\rho$-map) and in a minimal way: $M_{\bullet}$ and $M_{\rho}^{f-1} \oplus M_{\rho}^{*}$ have the same minimal number of generators as $R$-modules. This presentation is obtained via the choice of a quasi-basis.  To read off the intrinsic structure of $M_{\bullet}$ via these presentations we proceed looking at the action of $\text{Aut}_{\text{filt}}(M_{\rho}^{f-1} \oplus M_{\rho}^{*})$ on $\text{Epi}_{\text{filt}}(M_{\rho}^{f-1} \oplus M_{\rho}^{*},M_{\bullet})$, in search of invariants. The next proposition is then quite relevant for us.
\begin{proposition} \label{transitive action}
The action of $\emph{Aut}_{\emph{filt}}(M_{\rho}^{f-1} \oplus M_{\rho}^{*})$ on $\emph{Epi}_{\emph{filt}}(M_{\rho}^{f-1} \oplus M_{\rho}^{*},M_{\bullet})$ is transitive.
\begin{proof}
Recall from definition \ref{functors} that for a positive integer $i$ we denote by $F_{1,i}$ the functor $F_{1,i}:\text{Filt-}R\text{-mod} \to \text{Filt-}R\text{-mod}$, defined as $F_{1,i}(N_{\bullet}):=N_1/N_i$, with the quotient filtration, on the objects, and on a morphism $\phi:M_{\bullet} \to N_{\bullet}$, one has that $F_{1,i}(\phi)$ is defined as the morphism induced by $\phi$, from $F_{1,i}(M_{\bullet})$ to $F_{1,i}(N_{\bullet})$. Fix $\mathcal{B}$ a quasi-basis of $M_{\bullet}$. Take $\phi \in \text{Epi}_{\text{filt}}(M_{\rho}^{f-1} \oplus M_{\rho}^{*},M_{\bullet})$. We claim that we can find a filtered basis $\mathcal{B}'$ of $M_{\rho}^{f-1} \oplus M_{\rho}^{*}$ such that $\phi(\mathcal{B}')=\mathcal{B}$. We prove this in 2 steps.

1) Firstly we observe that $F_{1,{e_{\rho}}^{*}}(\phi)$ is an isomorphism. Indeed by construction $F_i(\phi)=F_i(F_{1,e_{\rho}^{*}})(\phi)$ is an isomorphism for each $i \in T_{\rho}$, and on both sides the $[\pi]_{i}$-maps are isomorphisms for each $i<e'_{\rho}$ since they are $(f,\rho)$-quasi-free. So it follows that they are isomorphisms for all $i<e_{\rho}^{*}$. So the observation is proved by Proposition \ref{phi surj}. This provides us with the piece of the filtered basis corresponding to the elements $x \in \mathcal{B}$ with $w(x) \in T_{\rho}$. 

2) Take the unique $x \in \mathcal{B}$ with $w(x)=e_{\rho}^{*}$. By Proposition \ref{first non iso is epi}, together with Step 1), we can find $y \in \phi^{-1}(x)$ with $w(y)=e_{\rho}^{*}$. Now we claim that $y$ must generate $\text{coker}[\pi]_{e'_{\rho}}$. Since this is a 1-dimensional $R/(\pi_R)$-vector space this is equivalent to claiming that $y$ is not the $0$-class in that cokernel. But if it were the $0$-class, then it there would exist $z$ with $w(z)=e'_{\rho}$, such that $\pi_Rz=y \ \text{mod} \ {(M_{\rho}^{f-1} \oplus M_{\rho}^{*})}_{e_{\rho}^{*}+1}$. But, from Step 1), it follows that $w(\phi(z))=e'_{\rho}$, but then, since $x=\pi_R\phi(z) \ \text{mod} \ {(M_{\rho}^{f-1} \oplus M_{\rho}^{*})}_{e_{\rho}^{*}+1}$, we see that $x$ is in the $0$-class in $\text{coker}[\pi]_{e'_{\rho}}$, which is a contradiction.

So given $\phi_1,\phi_2 \in \text{Epi}_{\text{filt}}(M_{\rho}^{f-1} \oplus M_{\rho}^{*},M_{\bullet})$, there are two filtered basis $\mathcal{B}_1, \mathcal{B}_2$ of $M_{\rho}^{f-1} \oplus M_{\rho}^{*}$ mapping to $\mathcal{B}$ via respectively $\phi_1,\phi_2$ as explained above. It follows that there exists a suitable bijection, $\theta$, between $\mathcal{B}_1$ and $\mathcal{B}_2$ that respects the weights and such that $\phi_2 \circ \theta=\phi_1$ on $\mathcal{B}_1$. But then, by Proposition \ref{Universal property of free filtered modules}, we have that $\theta$ extends to a filtered automorphism of $M_{\rho}^{f-1} \oplus M_{\rho}^{*}$ and $\phi_2 \circ \theta=\phi_1$ holds on all $M_{\rho}^{f-1} \oplus M_{\rho}^{*}$ and we are done. 
\end{proof}
\end{proposition} 
Proposition \ref{transitive action} and Proposition \ref{presentations} tell us that to classify-$(f,\rho)$-quasi-free filtered modules we have to accomplish two tasks: \\
(a) Classify the orbits of vectors in $M_{\rho}^{f-1} \oplus M_{\rho}^{*}$ under $\text{Aut}_{\text{filt}}(M_{\rho}^{f-1} \oplus M_{\rho}^{*})$. \\
(b) Recover which orbits of task (a) arise from $(f,\rho)$-quasi-free filtered modules. \\
This is what we do next.

\subsubsection{Jump sets parametrize orbits} \label{orbits}
We keep denoting by $\rho$ a shift map, and by $f$ a positive integer. Whenever a star is added, and we refer to an extended jump set in the following statements, we will be implicitly assuming that, in that case, $T_{\rho}$ is finite. On the other hand, in the parts of the statements where there is no star and we refer to regular jump sets, we only require $\rho$ to be a shift. We begin by attaching to each jump set a vector.
\begin{definition}
Let $(I,\beta)$ be a $\rho$-jump set (resp.\ an extended $\rho$-jump set). We denote by $v_{(I,\beta)}$ the following vector of $\pi_RM_{\rho}^f$ (resp.\ $\pi_R(M_{\rho}^{f-1} \oplus M_{\rho}^{*})$): for each $i \in T_{\rho}$ (resp.\ in $T_{\rho}^{*}$) with $i \not \in I$, the projection of $v_{(I,\beta)}$ on $S_{i}^{f}$ (resp.\ the same and on $S_{e_{\rho}^{*}}$) is $0$. For each $i \in I$, the projection of $v_{(I,\beta)}$ on $S_{i}^{f}$ (resp.\ the same and on $S_{e_{\rho}^{*}}$) is the vector $(\pi_R^{\beta(i)},0,\ldots ,0)$, having $\pi_R^{\beta(i)}$ on the first coordinate and $0$ on all the others (resp.\ $(\pi_{R}^{\beta(e_{\rho}^{*})})$).
\end{definition}
We now prove that with the map $(I,\beta) \mapsto v_{(I,\beta)}$ we catch each orbit at least once. For a vector $v \in \pi_RM_{\rho}^f$ (resp.\ $\pi_RM_{\rho}^{f-1} \oplus \pi_RM_{\rho}^{*}$), denote by $A_v$ the set of elements of $T_{\rho}$ (resp.\ $T_{\rho}^{*}$), such that $\text{proj}_{i}(v) \neq 0$, where $\text{proj}_{i}$ denotes the projection on the factor $S_{i}^{f}$ (resp.\ the same if $i \in T_{\rho}$ and we look at $\text{proj}_{S_{e_{\rho}^{*}}}$ for $i=e_{\rho}^{*}$). For $a \in A_v$ define $b_v(a)=\text{ord}_{R}(\text{proj}_{i}(v))$, the valuation of the $a$-th projection. Recall the definition of $(I_{(A_v,b_v)}^{-},\beta_{(A_v,b_v)}^{-})$ from Proposition \ref{jump set attached to a function}.
\begin{proposition} \label{reduction process}
For each $v \in \pi_RM_{\rho}^f$ (resp.\ $\pi_RM_{\rho}^{f-1} \oplus \pi_RM_{\rho}^{*}$) there exists an automorphism $\theta \in \emph{Aut}_{\emph{filt}}(M_{\rho}^f)$ (resp.\ $\theta \in \emph{Aut}_{\emph{filt}}(M_{\rho}^{f-1} \oplus M_{\rho}^{*})$) such that $\theta(v)=v_{(I,\beta)}$, where $(I,\beta):=(I_{(A_v,b_v)}^{-},\beta_{(A_v,b_v)}^{-})$.
\begin{proof}
Clearly we can find a filtered automorphism $\theta_0$ such that $\theta_0(v)=v_{(A_v,b_v)}$, where $v_{(A_v,b_v)}$ denotes the following vector of $\pi_RM_{\rho}^f$ (resp.\ $\pi_R(M_{\rho}^{f-1} \oplus M_{\rho}^{*})$): for each $i \in T_{\rho}$ (resp.\ $T_{\rho}^{*}$) with $i \not \in A_v$, the projection of $v_{(A_v,b_v)}$ on $S_i^f$ (resp.\ the same and to $S_{e_{\rho}^{*}}$) is $0$, while for each $i \in A_v$, the projection of $v_{(A_v,b_v)}$ on $S_i^f$ is $(\pi_R^{b_v(i)},0, \ldots ,0)$ (resp.\ the same and for $i=e_{\rho}^{*}$ is $(\pi_R^{b_v(e_{\rho}^{*})})$). So without loss of generality, we can assume that $v$ has this special form.

Next, let $(i,b_v(i)) <_{\rho} (j,b_v(j))$. That means that either $i<j$ and $b_v(i)<b_v(j)$, or that $i>j$ and $\rho^{b_v(i)}(i)<\rho^{b_v(j)}(j)$. Observe that in either case we have $b_v(i)<b_v(j)$ and the $R$-linear automorphism $\theta_{i,j}$ on $M_{\rho}$ (resp.\ $M_{\rho}^{*}$), defined as $\theta_{i,j}((x_h)_{h \in T_{\rho}})=(x_h-\pi_R^{b_v(j)-b_v(i)}\delta_{i,h}x_j)_{i \in T_{\rho}}$ (resp.\ as  $\theta_{i,j}((x_h)_{h \in T_{\rho}^{*}})=(x_h-\pi_R^{b_v(j)-b_v(i)}\delta_{i,h}x_j)_{i \in T_{\rho}^{*}}$) is \emph{filtered}, precisely due to the above inequalities. Clearly we can extend $\theta_{i,j}$ to a filtered automorphism of $M_{\rho}^f$ (resp.\ $M_{\rho}^{f-1} \oplus M_{\rho}^{*}$) by simply letting it act as the identity on the complementary factor $M_{\rho}^{f-1}$. The obtained filtered automorphism $\theta_{i,j}$ satisfies the identity
$$A_{\theta_{i,j}(v)}=A_v-\{j\}, b_{\theta_{i,j}(v)}={b_{v}}_{|A_v-\{j\}}.
$$
If $T_{\rho}$ is finite, by repeatedly applying transformations $\theta_{i,j}$ we end up precisely having constructed a $\theta$ as claimed in this Proposition. If $T_{\rho}$ is infinite, one can repeatedly apply such elementary transformations $\theta_{i,j}$ in a a sequence that \emph{converges} to a filtered automorphism $\theta$ as we wanted to prove this Proposition. 
\end{proof}
\end{proposition}
For a vector $v \in \pi_RM_{\rho}^f$ (resp.\ in $\pi_RM_{\rho}^{f-1} \oplus \pi_RM_{\rho}^{*}$), denote by $g_v$ the map $g_v: \mathbb{Z}_{\geq 0} \to \mathbb{Z}_{\geq 0} \cup \{\infty\}$ defined as $g_v(i):=w_{M_{\rho}^f/\pi_R^{i}M_{\rho}^f}(v)$ (resp.\ $g_v(i):=w_{M_{\rho}^{f-1} \oplus M_{\rho}^{*}/\pi_R^{i}(M_{\rho}^{f-1} \oplus M_{\rho}^{*})}(v)$). Here $w_{M_{\rho}^f/\pi_R^{i}M_{\rho}^f}(v)$ (resp.\ $w_{M_{\rho}^{f-1} \oplus M_{\rho}^{*}/\pi_R^{i}(M_{\rho}^{f-1} \oplus M_{\rho}^{*})}(v)$) denotes the weight of $v$ in the $R$-module $M_{\rho}^f/\pi_R^{i}M_{\rho}^f$ (resp. $M_{\rho}^{f-1} \oplus M_{\rho}^{*}/\pi_R^{i}(M_{\rho}^{f-1} \oplus M_{\rho}^{*})$) viewed as a filtered $R$-module with the quotient filtration (see Section \ref{quotients}).  Say that $g_v$ \emph{breaks} at $i$ if $g_v(i) \neq g_{v}(i+1)$.  
\begin{proposition} \label{reconstruction process}
Let $(I,\beta)$ be a $\rho$-jump set (resp.\ an extended $\rho$-jump set). Let $v_{(I,\beta)} \in \pi_RM_{\rho}^f$ (resp.\ $v_{(I,\beta)} \in \pi_RM_{\rho}^{f-1} \oplus \pi_RM_{\rho}^{*}$). Then $g_v$ breaks at $i$ if and only if $i \in \beta(I)$. Moreover if $i \in I$, then we have that $g_v(\beta(i)+1)=\rho^{\beta(i)}(i)$.
\begin{proof}
Let $n$ be a positive integer such that there exists an $i \in I$ with $\beta(i)<n$. Denote by $i_0$ the smallest such $i$. Fix the standard basis for $M_{\rho}^f$ (resp.\ for $M_{\rho}^{f-1} \oplus M_{\rho}^{*}$) and denote it by $\{b_{ij}: i \in T_{\rho}, \ j \in \{1, \ldots ,f\}\}$ (resp.\ denote it by $\{b_{ij}: i \in T_{\rho}, \ j \in \{1, \ldots ,f\}\} \cup \{b_{e_{\rho}^{*},1}\}$). In this notation we have that $v_{(I,\beta)}=\sum_{i \in I} \pi_R^{\beta(i)}e_{i,1}$. It is clear that 
$$v_{(I,\beta)} \equiv \sum_{i \in I: \beta(i)<n}\pi_R^{\beta(i)}e_{i,1} \ \text{mod} \ \pi_R^{n} M_{\rho}^f.
$$
(resp.\ $v_{(I,\beta)} \equiv \sum_{i \in I: \beta(i)<n}\pi_R^{\beta(i)}e_{i,1} \ \text{mod} \ \pi_R^{n}\cdot(M_{\rho}^{f-1} \oplus M_{\rho}^{*})
$). Observe that, thanks to the definition of a jump set, we have that
$$w_{M_{\rho}^f}(\sum_{i \in I: \beta(i)<n}\pi_R^{\beta(i)}e_{i,1})=\rho^{\beta(i_0)}(i_0).
$$
(resp.\ $w_{M_{\rho}^{f-1} \oplus M_{\rho}^{*}}(\sum_{i \in I: \beta(i)<n}\pi_R^{\beta(i)}e_{i,1})=\rho^{\beta(i_0)}(i_0)
$). Therefore we conclude that
$$w_{M_{\rho}^f/\pi_R^nM_{\rho}^f}(\sum_{i \in I: \beta(i)<n}\pi_R^{\beta(i)}e_{i,1}) \geq \rho^{\beta(i_0)}(i_0).
$$
(resp.\ $w_{(M_{\rho}^{f-1} \oplus M_{\rho}^{*})/\pi_R^n\cdot(M_{\rho}^{f-1} \oplus M_{\rho}^{*})}(\sum_{i \in I: \beta(i)<n}\pi_R^{\beta(i)}e_{i,1}) \geq \rho^{\beta(i_0)}(i_0)
$). We next prove that this inequality is actually an equality which clearly gives the desired result.

Let $x \in M_{\rho}^f$ (resp.\ in $M_{\rho}^{f-1} \oplus M_{\rho}^{*}$). We claim that
$$w_{M_{\rho}^f}(\pi_R^nx+\sum_{i \in I: \beta(i)<n}\pi_R^{\beta(i)}e_{i,1}) \leq \rho^{\beta(i_0)}(i_0).
$$
(resp.\ $w_{M_{\rho}^{f-1}\oplus M_{\rho}^{*}}(\pi_R^nx+\sum_{i \in I: \beta(i)<n}\pi_R^{\beta(i)}e_{i,1}) \leq \rho^{\beta(i_0)}(i_0)
$). Indeed if this claim would not hold we must conclude that
$$w_{M_{\rho}^f}(\pi_R^nx)=\rho^{\beta(i_0)}(i_0).
$$
(resp.\ $w_{M_{\rho}^{f-1} \oplus M_{\rho}^{*}}(\pi_R^nx)=\rho^{\beta(i_0)}(i_0)
$). But, by construction of the free filtered modules, we have that $w_{M_{\rho}^f}(\pi_R^nx)=\rho^n(w_{M_{\rho}^f}(x))$ (resp.\ $w_{M_{\rho}^{f-1} \oplus M_{\rho}^{*}}(\pi_R^nx)=\rho^n(w_{M_{\rho}^{f-1} \oplus M_{\rho}^{*}}(x))$). This implies that $\rho^{n-\beta(i_0)}(w_{M_{\rho}^f}(x))=i_0$, contradicting that $i_0 \in T_{\rho}$, since $n>\beta(i_0)$ by construction (resp.\ it implies that $\rho^{n-\beta(i_0)}(w_{M_{\rho}^{f-1} \oplus M_{\rho}^{*}}(x))=i_0$. In case $i_0<e_{\rho}^{*}$, it again contradicts that $i_0 \in T_{\rho}$. If $i_0=e_{\rho}^{*}$ we would conclude that $b_{e_{\rho}^{*},1} \in \pi_R \cdot (M_{\rho}^{f-1} \oplus M_{\rho}^{*}$), which is not possible). This ends the proof. 
\end{proof}
\end{proposition}
This allows us to conclude the following important corollary.
\begin{corollary} \label{at most one}
In each orbit $\mathcal{O}$ of $\pi_RM_{\rho}^f$ under $\emph{Aut}_{filt}(M_{\rho}^f)$ (resp.\ $M_{\rho}^{f-1} \oplus M_{\rho}^{*}$ under $\emph{Aut}_{filt}(M_{\rho}^{f-1} \oplus M_{\rho}^{*})$), there exist at most one $\rho$-jump set (resp.\ extended $\rho$-jump set) such that $v_{(I,\beta)}$ belongs to $\mathcal{O}$.
\begin{proof}
Clearly the function $g_v$ is preserved by applying a filtered automorphism. But by Proposition \ref{reconstruction process} it follows that from the function $g_{v_{(I,\beta)}}$ one can reconstruct $(I,\beta)$. The conclusion follows.
\end{proof} 
\end{corollary}
So, putting together Proposition \ref{reduction process} and \ref{at most one}, we see that with the map $(I,\beta) \mapsto v_{(I,\beta)}$ we catch each orbit exactly once:
\begin{theorem} \label{bijection orbits jumps}
The map $(I,\beta) \to v_{(I,\beta)}$ induces a bijection between the set of $\rho$-jump sets (resp.\ extended $\rho$-jump sets) and the set of orbits of $\pi_RM_{\rho}^f$ under the action of $\emph{Aut}_{\emph{filt}}(M_{\rho}^f)$ (resp.\ orbits of $\pi_RM_{\rho}^{f-1} \oplus \pi_RM_{\rho}^{*}$ under the action of $\emph{Aut}_{\emph{filt}}(M_{\rho}^{f-1} \oplus M_{\rho}^{*})$).  
\end{theorem}
Given a vector $v \in \pi_RM_{\rho}^f$ (resp.\ in $\pi_RM_{\rho}^{f-1} \oplus \pi_RM_{\rho}^{*}$) we define $\text{filt-ord}(v)$ to be the jump set corresponding to the orbit of $v$ under the above bijection. As the terminology suggests, the map $\text{filt-ord}$ can be considered as the filtered analogue of the map $\text{ord}$, which gives the valuation of the vector $v$. Indeed in the latter case knowing $\text{ord}(v)$ gives exactly the orbit, under $R$-linear automorphisms, of $v$, likewise in the former case knowing $\text{filt-ord}(v)$ gives exactly the orbit, under filtered $R$-linear automorphisms, of $v$. Moreover as $\text{ord}(v)$ is computed by taking the minimum valuation of the coordinates of $v$, with respect to an $R$-linear basis, so $\text{filt-ord}(v)$ is computed by taking the set of minimal points with respect to $\leq_{\rho}$ for the graph of valuations of the coordinates of $v$, with respect to a filtered basis (see definition \ref{definition of a basis}).
\subsubsection{Jump sets parametrize quasi-free filtered module} \label{jump sets and quasi free}
Now we fix $\rho$ a shift with $T_{\rho}$ finite and $f$ a positive integer. Let $M_{\bullet}$ be a $(f,\rho)$-quasi-free filtered module that is not free. By Proposition \ref{presentations} and \ref{transitive action} we see that $M_{\bullet}$ correspond to a unique orbit of vectors in $M_{\rho}^{f-1} \oplus M_{\rho}^{*}$ under $\text{Aut}_{\text{filt}}(M_{\rho}^{f-1} \oplus M_{\rho}^{*})$. So, together with Theorem \ref{bijection orbits jumps}, we obtain a unique extended $\rho$-jump set $(I_{M_{\bullet}},\beta_{M_{\bullet}})$ that determines $M_{\bullet}$ as a filtered module. Thus the map $M_{\bullet} \mapsto (I_{M_{\bullet}},\beta_{M_{\bullet}})$ gives an injection of the set of isomorphism classes of $(f,\rho)$-quasi-free module that are not $(f,\rho)$-free to the set of extended $\rho$-jump sets. We now want to describe the image. By Proposition \ref{reconstruction process}, together with Corollary \ref{where is the torsion}, we find that $\rho^{\beta(\text{min}(I_{M_{\bullet}}))}(\text{min}(I_{M_{\bullet}}))=e_{\rho}^{*}$. Conversely one checks immediately that for an extended $\rho$-jump set $(I,\beta)$ such that $\rho^{\beta(\text{min}(I))}(\text{min}(I))=e_{\rho}^{*}$, the filtered $R$-module $(M_{\rho}^{f-1} \oplus M_{\rho}^{*})/Rv_{(I,\beta)}$ is a $(f,\rho)$-quasi-free module. We call these jump sets \emph{admissible}. We have thus proved the following theorem. 
\begin{theorem} \label{classification of quasi free}
The map sending an admissible extended $\rho$-jump set $(I,\beta)$ to $(M_{\rho}^{f-1} \oplus M_{\rho}^{*})/Rv_{(I,\beta)}$ induces a bijection from the set of admissible extended $\rho$-jump sets to the set of $(f,\rho)$-quasi-free filtered modules that are not $(f,\rho)$-free.
\end{theorem} 

\subsubsection{Reading the jump set inside the module} \label{reading jump inside}
We have classified $(f,\rho)$-quasi-free modules (which are not free) via admissible extended $\rho$-jump sets. We have proceeded by introducing an external module, $M_{\rho}^{f-1} \oplus M_{\rho}^{*}$, presenting each of them, and proving that the invariant of each presentation is an admissible extended $(f,\rho)$-jump set.

We now provide a description of the jump set $(I_{M_{\bullet}},\beta_{M_{\bullet}})$, internally from $M_{\bullet}$, without any further reference to an external module $M_{\rho}^{f-1} \oplus M_{\rho}^{*}$. In other words we face the task of providing the inverse of the bijection in Theorem \ref{classification of quasi free}, without reference to $M_{\rho}^{f-1} \oplus M_{\rho}^{*}$. We will proceed by imitating the way we reconstructed the jump set belonging to each orbit in Proposition \ref{reconstruction process}. For $v \in M_{\bullet}$  denote by $g_{v,M_{\bullet}}$ the map $g_{v,M_{\bullet}}: \mathbb{Z}_{\geq 0} \to \mathbb{Z}_{\geq 0} \cup \{\infty\}$ defined as $g_{v,M_{\bullet}}(i):=w_{M_{\bullet}/{\pi_R}^{i}M_{\bullet}}(v)$. Say that $g_v$ breaks at $i$ if $g_{v,M_{\bullet}}(i) \neq g_{v,M_{\bullet}}(i+1)$. Fix $\tilde{m}$ a generator of $(M_1)_{\text{tors}}$, denote by $N$ the exponent of the torsion, that is $N:=\text{min}\{i \in \mathbb{Z}_{\geq 1}:\pi_{R}^{i} \tilde{m}=0 \}$. The following proposition can be proved by a straightforward imitation of the proof of Proposition \ref{reconstruction process}.
\begin{proposition} \label{break points of g}
The function $g_{\tilde{m},M_{\bullet}}$ breaks exactly at the elements of $\beta_{M_{\bullet}}(I_{M_{\bullet}})-N$, moreover if $i \in I_{M_{\bullet}}$ then $g_{\tilde{m},M_{\bullet}}(i+1)=\rho^{\beta_{M_{\bullet}}(i)-N}(i)$.  
\end{proposition}
So we deduce the following corollary.
\begin{corollary} \label{classifying quasi free modules internally}
Let $M_{\bullet}$ be a $(f,\rho)$-quasi-free filtered module that is not free. Let $\tilde{m} \in M_1$ be a generator of $(M_1)_{\emph{tors}}$, then the map $g_{\tilde{m},M_{\bullet}}$ determines $M_{\bullet}$ as a filtered module. 
\end{corollary}
The following simple corollary of Theorem \ref{classification of quasi free} will be often useful. Recall the notation $(I_{(A,b)}^{-},\beta_{(A,b)}^{-})$ introduced in Proposition \ref{jump set attached to a function}. 
\begin{corollary} \label{The relation between the units}
Let $M_{\bullet}$ be a $(f,\rho)$-quasi-free filtered module that is not free. Let $I$ be a subset of $T_{\rho}^{*}$ and $b$ a map from $I$ to $\mathbb{Z}_{\geq 1}$. Suppose that for each $i \in I$ we have $m_i \in M_i$ satisfying the following three conditions. \\
\emph{(1)} For each $i \in I$ we have that $w_{M_{\bullet}}(m_i)=i$. \\
\emph{(2)} We have that
$$\sum_{i \in I}\pi_R^{b(i)}m_i=0.
$$
\emph{(3)} If $e_{\rho}^{*} \in I$ then $m_{e_{\rho}^{*}} \not \in \pi_R M_1$. 

Then it must be that 
$$(I_{(A,b)}^{-},\beta_{(A,b)}^{-})=(I_{M_{\bullet}},\beta_{M_{\bullet}}).
$$
\end{corollary}
The following proposition shall be often used to recover the structure of the $R$-module $M_1[\pi_R^{\infty}]:=(M_1)_{\text{tors}}$ from $(I_{M_{\bullet}},\beta_{M_{\bullet}})$. This goes as follows.
\begin{proposition} \label{how much torsion}
Let $M_{\bullet}$ be a $(f,\rho)$-quasi-free filtered $R$-module. Then we have that
$$M_1[\pi_R^{\infty}] \simeq R/\pi_R^{\beta(\emph{max}(I_{M_{\bullet}}))}R,
$$
as $R$-modules. 
\begin{proof}
Using Theorem \ref{classification of quasi free} we deduce that 
$$M_1[\pi_R^{\infty}] \simeq R/\pi_R^{\text{min}(\beta(I_{M_{\bullet}}))}R.
$$
Since $(I_{M_{\bullet}},\beta_{M_{\bullet}})$ is a jump set, the map $\beta_{M_{\bullet}}$ is in particular decreasing. Hence $\text{min}(\beta(I_{M_{\bullet}}))=\beta(\text{max}(I_{M_{\bullet}}))$, which gives precisely the desired conclusion.
\end{proof}
\end{proposition}

\section{Jumps of characters of a quasi-free module} \label{characters}
\subsection{Motivation and main results}
In section \ref{U1 as filtered module} we will see that $U_1$ as a filtered module is quasi-free. So, as we will see in detail in \ref{wild extension}, via the local reciprocity map the question of determining the possible upper jumps of a cyclic $p$-power totally ramified extension of a given local field is a special case of the question of determining the jumps of a cyclic character of a given $(f,\rho)$-quasi-free filtered module, which is the goal of the present section.

Let $R$ be a complete DVR. We denote by $Q(R)$ the fraction field of $R$. We equip $Q(R)/R$ with the discrete topology. 
\begin{definition}
(a) Let $M_{\bullet}$ be a filtered $R$-module. A \emph{character} of $M_{\bullet}$ is a continuous $R$-linear homomorphism $\chi:M_1 \to Q(R)/R$, where the implicit topology on $M_1$ is the one coming from the filtration, see \ref{metric structure}. 

(b) Let $\chi$ be a character of $M_{\bullet}$. A positive integer $i$ is said to be a \emph{jump} of $\chi$, if $\chi(M_i) \neq \chi(M_{i+1})$. We denote the collection of jumps of $\chi$ by $J_{\chi}$. Finally we denote by $\mathcal{J}_{M_{\bullet}}$ the collection of all $J_{\chi}$ as $\chi$ varies among characters of $M_{\bullet}$. 
\end{definition}
One can easily show that if $M_{\bullet}$ is linear (see definition \ref{definition of linear}), then for each character $\chi$ of $M_{\bullet}$ the set $J_{\chi}$ is finite. We fix a shift map $\rho$, and a positive integer $f$. Recall that $(f,\rho)$-quasi-free modules are in particular linear. The goal of this section is to understand exactly which are the possible sets of jumps:
\begin{goal}
Let $M_{\bullet}$ be a $(f,\rho)$-quasi-free module. Characterize the sets $A \subseteq \mathbb{Z}_{\geq 1}$ such that $A=J_{\chi}$ for some character $\chi$ of $M_{\bullet}$.
\end{goal}
We will proceed as follows: in \ref{jump set are set jump} we prove that $\mathcal{J}_{M_{\rho}^f}=\text{Jump}_{\rho}$ and $\mathcal{J}_{M_{\rho}^{f-1} \oplus M_{\rho}^{*}}=\text{Jump}_{\rho}^{*}$. Next in \ref{jump set are set jump +} we examine the case of $(f,\rho)$-quasi-free modules that are not free. Given such a module $M_{\bullet}$, we know from Theorem \ref{classification of quasi free} that all we need to know to understand $M_{\bullet}$ as a filtered module is the extended jump set $(I_{M_{\bullet}},\beta_{M_{\bullet}})$. So it must be possible to predict $\mathcal{J}_{M_{\bullet}}$ from $(I_{M_{\bullet}},\beta_{M_{\bullet}})$. We achieve this in Theorem \ref{classification of characters}, where it is shown that $\mathcal{J}_{M_{\bullet}} \subseteq \text{Jump}_{\rho}^{*}$, and the missing jump sets are characterized by a combinatorial criterion involving $(I_{M_{\bullet}},\beta_{M_{\bullet}})$. 
\subsection{Set of jumps are jump set for a free module} \label{jump set are set jump}
We proceed in the same way as we did for orbits of vectors in \ref{orbits}. Clearly the set of jumps of a character does not change if we apply to it a filtered automorphism of $M_{\bullet}$. Therefore we shall take advantage of this symmetry. It turns out that, for free filtered modules, knowing the set of jumps of a character $\chi$ is \emph{equivalent} to knowing to which orbit $\chi$ belongs (under the action of the group of filtered automorphisms). 
\begin{definition}
(a) Let $\chi$ be a character of $M_{\rho}^f$ (resp.\ of $M_{\rho}^{f-1}\oplus M_{\rho}^{*}$). Denote by $A_{\chi}$ the set of $i$ in $T_{\rho}$ (resp.\ $T_{\rho}^{*}$), such that $\chi(\text{proj}_{i}) \neq \{0\}$, where $\text{proj}_{i}$ denotes the projection on $S_i^{f}$ (resp.\ the same if $i \in T_{\rho}^{*}$, where the projection is on $S_{e_{\rho}^{*}}$ for the last coordinate). 

(b) For $a$ in $A_{\chi}$, denote by $b_{\chi}(a)=\text{min} \{r \in \mathbb{Z}_{\geq 1}: \pi_R^{r}\chi(\text{proj}_{a})= \{0\}\}$.
\end{definition}
We next show that, after applying a suitable filtered automorphism, one can make the pair $(A_{\chi},b_{\chi})$ a jump set (resp.\ an extended jump set). Recall the notation $(A_{\chi}^{+},b_{\chi}^{+})$ introduced in Proposition \ref{jump set attached to a function}.
\begin{proposition} \label{reduction for characters}
Let $\chi$ be a character of $M_{\rho}^f$ (resp.\ of $M_{\rho}^{f-1} \oplus M_{\rho}^{*}$). Then there exists $\theta \in \emph{Aut}_{\emph{filt}}(M_{\rho}^f)$ (resp.\ $\emph{Aut}_{\emph{filt}}(M_{\rho}^{f-1}\oplus M_{\rho}^{*})$) such that $(A_{\chi \circ \theta},b_{\chi \circ \theta})=(A_{\chi}^{+},b_{\chi}^{+})$. In particular $(A_{\chi \circ\theta},b_{\chi \circ \theta})$ is a $\rho$-jump set (resp.\ an extended $\rho$-jump set). 
\begin{proof}
The structure of the proof is the same as the one given for Proposition \ref{reduction process}, we just mention some differences. Just as in that proof, as a first step we can assume $\chi$ is a character vanishing on the factor $M_{\rho}^{f-1} \oplus 0$ and, as a character of the factor $M_{\rho}$ (resp.\ $M_{\rho}^{*}$), it is defined as follows. If $i \not \in A_{\chi}$, then we have  $\chi_{|S_i}=0$. If $i \in A_{\chi}$, we have $\chi_{|S_i}(1)=\pi_{R}^{-b_{\chi}(i)}$. Next if for two points $(i,b_{\chi}(i),(j,b_{\chi}(j))$ in $\text{Graph}(b_{\chi})$ we have $(i,b_{\chi}(i)<_{\rho}(j,b_{\chi}(j))$, it follows that the transformation $\theta_{i,j}$, introduced in the proof of Proposition \ref{reduction process}, is filtered. Now, the only difference with that proof, is that the effect of applying $\theta_{i,j}$ is to erase the smaller point, namely $(i,b_{\chi}(i))$. Indeed the character $\chi \circ \theta_{i,j}$  will send to $0$ all the factors $S_a$ with $a \not \in A_{\chi}$, and it will be $0$, additionally also on $S_i$. On the other hand, on all the other factors $S_a$, with $a \in A_{\chi}-\{i\}$ it coincides with $\chi$. Thus by repeatedly applying this type of transformation the sequence of filtered automorphism so produced converges to a filtered automorphism $\theta$ with $(A_{\chi \circ \theta},b_{\chi \circ \theta})=(A_{\chi}^{+},b_{\chi}^{+})$, concluding the proof.
\end{proof}
\end{proposition}
We now show that if $(A_{\chi},b_{\chi})$ is a $\rho$-jump set (resp.\ an extended $\rho$-jump set), then, if viewed as a subset of $\mathbb{Z}_{\geq 1}$, it is the set of jumps of $\chi$. 
\begin{proposition} \label{when characters are reduced}
Let $\chi$ be a character of $M_{\rho}^f$ (resp.\ of $M_{\rho}^{f-1} \oplus M_{\rho}^{*}$), such that $(A_{\chi},b_{\chi})$ is a jump set (resp. an extended jump set). Then $J_{\chi}=J_{(A_{\chi},b_{\chi})}$.
\begin{proof}
For a general $\chi$ we have the following formula
$$ \text{ord}(\chi(M_{\rho}^f)_{i})=\text{max}(\{b_{\chi}(j)-v_{\rho}(j,i)\}_{j \in T_{\rho}}),
$$
where for $i \in \mathbb{Z}_{\geq 1}$ and $j \in T_{\rho}$ (resp.\ $T_{\rho}^{*}$) we have that $v_{\rho}(j,i)=\text{min}(\{s \in \mathbb{Z}_{\geq 0}: \rho^{s}(j) \geq i\})$ (respectively we have the formula
$$\text{ord}(\chi(M_{\rho}^f)_{i})=\text{max}(\{b_{\chi}(j)-v_{\rho}(j,i)\}_{j \in T_{\rho}^{*}})).
$$

Since $(A_{\chi},b_{\chi})$ is a jump set (resp.\ an extended jump set), it is visible from the definition that the right hand side, as a function of $i$, changes value precisely in the set $J_{(A_{\chi},b_{\chi})}$, which is precisely giving the desired identity $J_{(A_{\chi},b_{\chi})}=J_{\chi}$.
\end{proof}
\end{proposition}
So for two characters of $M_{\rho}^f$ or $M_{\rho}^{f-1} \oplus M_{\rho}^{*}$ the equivalence relation ``having the same set of jumps" and ``being in the same filtered orbit" are precisely the same relation and one obtains the following fact. 
\begin{theorem} \label{character for free modules}
Let $\rho$ be a shift map, and let $f$ be a positive integer. We have that $\mathcal{J}_{M_{\rho}^f}=\emph{Jump}_{\rho}$, and if $T_{\rho}$ is finite, then $\mathcal{J}_{M_{\rho}^{f-1} \oplus M_{\rho}^{*}}=\emph{Jump}_{\rho}^{*}$.
\end{theorem}
The similarity with Theorem \ref{j.s.param.orbit} is noteworthy: in both cases jump sets parametrize orbits.

\subsection{Sets of jumps for a quasi-free module} \label{jump set are set jump +}
Let $\rho$ be a shift map with $T_{\rho}$ finite. Let $f$ be a positive integer. Let $M_{\bullet}$ be a $(f,\rho)$-quasi-free module that is not free. Then from Theorem \ref{classification of quasi free} we know that the knowledge of $M_{\bullet}$ as a filtered module is equivalent to the knowledge of the extended $\rho$-jump set $(I_{M_{\bullet}},\beta_{M_{\bullet}})$. So the invariant $\mathcal{J}_{M_{\bullet}}$ is completely determined once we know $(I_{M_{\bullet}},\beta_{M_{\bullet}})$. Here we explain how. 
We know from Proposition \ref{presentations} that $M_{\bullet}$ admits a presentation $\phi:M_{\rho}^{f-1} \oplus M_{\rho}^{*} \to M_{\bullet}$, with $\text{coker}(F_i(\phi))=0$ for every positive integer $i$, so from Proposition \ref{a lot surjective} we know that $\phi_{|(M_{\rho}^{f-1} \oplus M_{\rho}^{*})_i}$ is a map onto $M_i$ for each positive integer $i$. It follows that given a character $\chi$ of $M_{\bullet}$, the induced character on $M_{\rho}^{f-1} \oplus M_{\rho}^{*}$ obtained by post-composition to $\phi$, has the same set of jumps of $\chi$. So together with Theorem \ref{character for free modules} we obtain:
\begin{proposition} \label{set of jump for quasi free are jump set}
Let $M_{\bullet}$ be a $(f,\rho)$-quasi-free module. Then $\mathcal{J}_{M_{\bullet}} \subseteq \emph{Jump}_{\rho}^{*}$.
\end{proposition}
Thus we see that to characterize which elements of $\text{Jump}_{\rho}^{*}$ belongs to $\mathcal{J}_{M_{\bullet}}$ we need to see which jump sets are ruled out when on a character $\chi$ of $M_{\rho}^{f-1} \oplus M_{\rho}^{*}$ we impose the condition $\chi(v_{(I_{M_{\bullet}},\beta_{M_{\bullet}})})=0$. The following simple lemma will be relevant to this end. For $x \in Q(R)/R$ we denote by $\text{ord}(x)$ the smallest non-negative integer $n$ such that $\pi_R^nx=0$. Equivalently we can say that $\text{ord}(x)$ is the unique non-negative integer such that $Rx$ is isomorphic to $R/\pi_R^nR$ as $R$-modules. 
\begin{lemma}{\label{units in a killer vector}}
Let $n$ be a positive integer and $(v_1,\ldots ,v_n) \in (Q(R)/R)^n$. Write $Y:=\{i \in \{1,\ldots ,n\}:0<\emph{ord}(v_i), \emph{ord}(v_i)=\emph{max}\{\emph{ord}(v_j), \ j \in \{1,\ldots ,n\}\}\}$. Then the following hold: 

\emph{(a)} Assume $|R/m_R| \neq 2$. Then there exists a vector $(a_1,\ldots ,a_n) \in (R^{*})^n$ such that $\sum_{i=1}^{n}a_iv_i=0$ if and only if $|Y|\neq 1$. 

\emph{(b)} Assume $|R/m_R|=2$. Then there exists a vector $(a_1,\ldots ,a_n) \in (R^{*})^n$ such that $\sum_{i=1}^{n}a_iv_i=0$ if and only if $|Y| \equiv 0 \ mod \ 2$.  
\begin{proof}
(a) Assume $|Y| \neq 1$. We can assume $|Y| \neq 0$ because otherwise $(v_1,\ldots ,v_n)$ is the zero vector and any $(a_1,\ldots ,a_n) \in (R^{*})^n$ would prove the conclusion. Since $|R/m_R| \neq 2$ we can find $\lambda \in R^{*}$ such that $\lambda \not \equiv 1 \ \text{mod} \ m_R$. Now pick $i,j$ distinct elements of $Y$, and observe that at least one of the following two hold: \\
1) $\text{ord}(v_i)=\text{ord}(v_j+ \sum_{h \not \in \{i,j\}}v_h)$. \\
2) $\text{ord}(v_i)=\text{ord}(\lambda v_j+ \sum_{h \not \in \{i,j\}}v_h)$. \\
In each case, 1) and 2), we can find $\mu \in R^{*}$ such that, respectively $\mu v_i=v_j+ \sum_{h \not \in \{i,j\}}v_h$, or $\mu v_i=\lambda v_j+ \sum_{h \not \in \{i,j\}}v_h$. In each of the two cases we obtain the desired conclusion. Conversely assume that there exists a vector $(a_1,\ldots ,a_n) \in (R^{*})^n$ such that $\sum_{i=1}^{n}a_iv_i=0$. Suppose that $|Y|=1$, call $k$ its unique element: then we have $\text{ord}(v_k)=\text{ord}(\sum_{i=1}^{n}a_iv_i)=0$, contradicting the definition of $Y$. 

(b) Assume $|Y| \equiv 0 \ \text{mod} \ 2$. We can assume $|Y| \neq 0$ because otherwise $(v_1,\ldots ,v_n)$ is the zero vector and any $(a_1,\ldots ,a_n) \in (R^{*})^n$ would prove the conclusion. So pick $i \in Y$. Then observe that, since $|Y - \{i\}| \equiv 1 \ \text{mod} \ 2$ and $|R/m_R|=2$, we have that $\text{ord}(v_i)=\text{ord}(\sum_{h \neq i}v_h)$. Thus, it follows that there exists $\mu \in R^{*}$ such that $\mu v_i=\sum_{h \neq i}v_h$, which is the desired conclusion. Conversely assume there exists a vector $(a_1,\ldots ,a_n) \in (R^{*})^n$ such that $\sum_{i=1}^{n}a_iv_i=0$. Suppose that $|Y| \equiv 1 \ \text{mod} \ 2$. Then pick $k \in Y$ and observe that, since $|R/m_R|=2$, we have $\text{ord}(v_k)=\text{ord}(\sum_{i=1}^{n}a_iv_i)=0$ contradicting the definition of $Y$.
\end{proof}
\end{lemma}
We can now give a criterion to decide if an extended jump set $(I,\beta)$ is realizable as a set of jumps of a character of $M_{\bullet}$. Such a criterion consists in a combinatorial comparison between $(I,\beta)$ and $(I_{M_{\bullet}},\beta_{M_{\bullet}})$. The precise condition for $(I,\beta)$ to be ruled out are conditions 2.1) and 2.2) of the following theorem (in case (a) and (b) respectively). 
\begin{theorem}\label{classification of characters}
Let $f$ be a positive integer and let $\rho$ be a shift. Let $M_{\bullet}$ be a $(f,\rho)$-quasi-free filtered $R$-module that is not free. Let $(I,\beta) \in Jump_{\rho}^{*}$. Define $\emph{Max}((I,\beta),(I_{M_{\bullet}},\beta_{M_{\bullet}})):=\{i \in I \cap I_{M_{\bullet}}:\beta(i)-\beta_{M_{\bullet}}(i)>0 \wedge \forall j \in I \cap I_{M_{\bullet}}, \ \beta(i)-\beta_{M_{\bullet}}(i) \geq \beta(j)-\beta_{M_{\bullet}}(j) \}$. In what follows we denote by $\emph{Max}:=\emph{Max}((I,\beta),(I_{M_{\bullet}},\beta_{M_{\bullet}})$.

\emph{(a)} Suppose $|R/m_R| \neq 2$. Then one has that $(I,\beta) \not \in \mathcal{J}_{M_{\bullet}}$ if and only if the following two conditions are satisfied:\\
\emph{(a.1)} $|\emph{Max}|=1$ and if $f>1$ then $\emph{Max}=\{e_{\rho}^{*}\}$.  \\
\emph{(a.2)} Let $j$  be the unique element of $\emph{Max}_{M_{\bullet}}((I,\beta))$. For every $i \in I_{M_{\bullet}}-I$, the point $(i,\beta(j)-\beta_{M_{\bullet}}(j)+\beta_{M_{\bullet}}(i))$ is maximal in $\emph{Graph}(\beta) \cup \{(i,\beta(j)-\beta_{M_{\bullet}}(j)+\beta_{M_{\bullet}}(i))\}$, with respect to the ordering $\leq_{\rho}$. 

\emph{(b)} Suppose $|R/m_R|=2$. Then $(I,\beta) \not \in \mathcal{J}_{M_{\bullet}}$ if and only if the following two conditions are satisfied: \\
\emph{(b.1)} $|\emph{Max}| \equiv 1 \ \emph{mod} \ 2$ and if $f>1$ then $\emph{Max}=\{e_{\rho}^{*}\}$.  \\
\emph{(b.2)} Let $j$ be any element of $\emph{Max}$. For every $i \in I_{M_{\bullet}}-I$, the point $(i,\beta(j)-\beta_{M_{\bullet}}(j)+\beta(i))$ is maximal in $\emph{Graph}(\beta) \cup \{(i,\beta(j)-\beta_{M_{\bullet}}(j)+\beta_{M_{\bullet}}(i))\}$, with respect to the ordering $\leq_{\rho}$. 
\begin{proof}
(a) Denote by $\{b_{i,j}: i \in T_{\rho}, \ j \in \{1, \ldots , f\} \} \cup \{b_{e_{\rho}^{*},1} \}$ the standard filtered basis for $M_{\rho}^{f-1} \oplus M_{\rho}^{*}$). With this notation we have that
$$v_{(I_{M_{\bullet}},\beta_{M_{\bullet}})}=\sum_{i \in I} \pi_R^{\beta_{M_{\bullet}}(i)}b_{i,1}.
$$  
We divide the proof in 9 elementary steps.

1) We fix a presentation $\phi:M_{\rho}^{f-1} \oplus {M}_{\rho}^{*} \to M_{\bullet}$ as in Proposition \ref{presentations}, with $\text{ker}(\phi)=Rv_{(I_{M_{\bullet}},\beta_{M_{\bullet}})}$ as in Theorem \ref{classification of quasi free}. 

2) The task of realizing $(I,\beta)$ from a character is equivalent to the task of finding a $\chi:M_{\rho}^{f-1} \oplus {M}_{\rho}^{*} \to Q(R)/R$ such that $(I_{\chi},\beta_{\chi})=(I,\beta)$, and $\sum_{i \in I_{M_{\bullet}}}\pi_R^{\beta_{M_{\bullet}}(i)}\chi(b_{i,1})=0$. 

3) Suppose that $I \cap I_{M_{\bullet}}$ is either empty or that $\beta-\beta_{M_{\bullet}}$ does not assume a strictly positive maximum on $I \cap I_{M_{\bullet}}$. We claim that then task 2) is realizable. Indeed thanks to Lemma \ref{units in a killer vector} part (a), we can find for each $i \in I \cap I_{M_{\bullet}}$ a unit $\epsilon_i \in R^{*}$ with the property that 
$$\sum_{i \in I \cap I_{M_{\bullet}}}\frac{\epsilon_i}{\pi_R^{\beta(i)-\beta_{M_{\bullet}}(i)}}=0.
$$
Therefore we can realize task 2) with the following character $\chi$. For $i \in I \cap I_{M_{\bullet}}$ we put $\chi(b_{1,i}):=\frac{\epsilon_i}{\pi_R^{\beta(i)}}$. For $i \in I-I_{M_{\bullet}}$ we put $\chi(b_{1,i}):=\frac{1}{\pi_R^{\beta(i)}}$. For any $(i,h) \in T_{\rho} \times \{1, \ldots ,f \} \cup \{e_{\rho}^{*}\} \times \{1\}$ with $i \not \in I$ or $h>1$ we put $\chi(b_{i,h})=0$. With Proposition \ref{when characters are reduced} we conclude immediately that $J_{\chi}=(I,\beta)$ and we are done. So we can assume that $I \cap I_{M_{\bullet}}$ is non-empty and that $\beta-\beta_{M_{\bullet}}$ assumes a positive maximum at a unique point of $I \cap I_{M_{\bullet}}$, which we shall call $j$. 

4) Assume that $j \neq e_{\rho}^{*}$ and $f>1$. Then we proceed by distinguishing two cases. \\
4.1) There is no other $k \in I \cap I_{M_{\bullet}}$ different from $j$ such that $\beta(k)-\beta_{M_{\bullet}}(k)>0 $. Then we consider the following character $\chi$. For each $i$ in $I-\{e_{\rho}^{*}\}$ we put $\chi(b_{i,2})=\frac{1}{\pi_R^{\beta(i)}}$. If $e_{\rho}^{*} \in I$ we put $\chi(b_{e_{\rho}^{*},1})=\frac{1}{\pi_R^{\beta(e_{\rho}^{*})}}$. For all the other $(i',h)$ in $T_{\rho}^{*} \times \{1, \ldots , f\}$ we put $\chi(b_{i',h})=0$. We see that since $f>1$ and $j \neq e_{\rho}^{*}$ we trivially obtain
$$\sum_{i \in I_{M_{\bullet}}}\pi_R^{\beta_{M_{\bullet}}(i)}\chi(b_{i,1})=0.
$$ 
Hence task 2) is accomplished thanks to Proposition \ref{when characters are reduced}. So we can assume that such a $k$ exists. \\
4.2) Suppose that there exists $k' \in I \cap I_{M_{\bullet}}$ different from $j$ such that $\beta(k')-\beta_{M_{\bullet}}(k')>0$. Choose a $k$ such that $\beta(k)-\beta_{M_{\bullet}}(k) \geq \beta(k')-\beta_{M_{\bullet}}(k')$ for each $k' \in I \cap I_{M_{\bullet}}$ with $k'$ different from $j$. Next observe that thanks to Proposition \ref{units in a killer vector} part (a), we can find for each $i \in I \cap I_{M_{\bullet}}$ a unit $\epsilon_i \in R^{*}$ in such a way that
$$\sum_{i \in (I \cap I_{M_{\bullet}})-\{j\}}\frac{\epsilon_i}{\pi_R^{\beta(i)-\beta_{M_{\bullet}}(i)}}+\frac{\epsilon_j}{\pi_{R}^{\beta(k)-\beta_{M_{\bullet}}(k)}}=0.
$$
Now we proceed constructing a character $\chi$. We put $\chi(b_{j,1})=\frac{\epsilon_j}{\pi_R^{\beta(k)-\beta_{M_{\bullet}}(k)+\beta_{M_{\bullet}}(j)}}$ and $ \chi(b_{j,2})=\frac{1}{\pi_R^{\beta(j)}}$. For all $i \in (I \cap I_{M_{\bullet}})-\{j\}$ we put $\chi(b_{i,1})=\frac{\epsilon_i}{\pi_R^{\beta(i)}}$. For all $i \in I-I_{M_{\bullet}}$ we put $\chi(b_{i,1})=\frac{1}{\pi_R^{\beta(i)}}$. For all remaining vectors $b$ of the basis we put $\chi(b)=0$. Since $\beta(k)-\beta_{M_{\bullet}}(k)+\beta_{M_{\bullet}}(j) < \beta(j) $ we conclude by Proposition \ref{reduction for characters} and Proposition \ref{when characters are reduced} that $J_{\chi}=(I,\beta)$. Moreover, by construction,
$$\sum_{i \in I_{M_{\bullet}}} \pi_R^{\beta_{M_{\bullet}}(i)}\chi(b_{i,1})=0.
$$ 
Hence we have realized task 2) in this case as well.

5) Thanks to Step 1)--4) we can assume that $|\text{Max}|=1$, and that either $f=1$ or $\text{Max}=\{e_{\rho}^{*}\}$. Otherwise we have shown, in the previous steps, that we can accomplish task 2). Keep denoting by $j$ the unique point of $\text{Max}$. 

6) Assume there is $i' \in I_{M_{\bullet}}-I$ such that the point $(i',\beta(j)-\beta_{M_{\bullet}}(j)+\beta_{M_{\bullet}}(i'))$ is not maximal in $\text{Graph}(\beta) \cup \{(i',\beta(j)-\beta_{M_{\bullet}}(j)+\beta_{M_{\bullet}}(i')\}$, with respect to the ordering $\leq_{\rho}$. Then we can accomplish task 2) constructing a character $\chi$ in the following manner. Observe that, thanks to Proposition \ref{units in a killer vector} part (a), we can attach to each $i \in (I \cap I_{M_{\bullet}}) \cup \{i'\}$ a unit $\epsilon_i \in R^{*}$ in such a way that 
$$\sum_{i \in I \cap I_{M_{\bullet}}} \frac{\epsilon_i}{\pi_R^{\beta(i)-\beta_{M_{\bullet}}}} +\frac{\epsilon_{i'}}{\pi_R^{\beta(j)-\beta_{M_{\bullet}}(j)}}=0.
$$
For each $i \in I \cap I_{M_{\bullet}}$ put $\chi(b_{i,1})=\frac{\epsilon_i}{\pi_R^{\beta(i)}}$. Moreover put $\chi(b_{i',1})=\frac{1}{\pi_R^{\beta(j)-\beta_{M_{\bullet}}(j)+\beta_{M_{\bullet}}(i')}}$ and $\chi(b_{i,1})=\frac{1}{\pi_R^{\beta(i)}}$ for each $i \in I-I_{M_{\bullet}}$. By construction we obtain
$$\sum_{i \in I_{M_{\bullet}}}\pi_R^{\beta_{M_{\bullet}}(i)}\chi(b_{i,1})=0.
$$
Finally the hypothesis that the point $(i',\beta(j)-\beta_{M_{\bullet}}(j)+\beta_{M_{\bullet}}(i'))$ is not larger, with respect to $\leq_{\rho}$, than some point in $\text{Graph}(\beta)$, tells us, through Proposition \ref{reduction for characters} and Proposition \ref{when characters are reduced}, that $J_{\chi}=(I,\beta)$.

7) Steps 1)--6) prove that if (a.1) and (a.2) are not both satisfied then $(I,\beta) \in \mathcal{J}_{M_{\bullet}}$. We next proceed proving the converse implication. 

8) Observe that if a set $A\subseteq {\mathbb{Z}}^{2}$ is given, together with a point $(x,y) \in A$ that is maximal in $A$ with respect to $\leq_{\rho}$, then any point of the form $(x,\tilde{y})$ with $\tilde{y} \geq y$ is maximal in $A$, with respect to $\leq_{\rho}$. 

9) Suppose (a.1) and (a.2) both hold. Denote by $j$ the unique element of $\text{Max}$. Let $\chi$ be a character of $M_{\rho}^{f-1} \oplus M_{\rho}^{*}$ with $J_{\chi}=(I,\beta)$. We shall prove that
$$ \sum_{i \in I}\pi_R^{\beta_{M_{\bullet}}(i)}\chi(b_{i,1}) \neq 0.
$$ 
We proceed by contradiction. Suppose that $ \sum_{i \in I}\pi_R^{\beta_{M_{\bullet}}(i)}\chi(b_{i,1})=0
$. By Proposition \ref{reduction for characters} and Proposition \ref{when characters are reduced} we have that $\text{ord}(\chi(b_{j,1}))=\beta(j)$: this is clear for $f=1$ and if $f \geq 2$ we are using that in this case $j$ must be $e_{\rho}^{*}$. 
Next using Lemma \ref{units in a killer vector} part (a), we see that at least one $i \in I_{M_{\bullet}}-\{j\}$ must satisfy $\text{ord}(\chi(b_{i,1})) \geq \beta_{M_{\bullet}}(i)+\beta(j)- \beta_{M_{\bullet}}(j)$. Such an $i$ cannot be in $I$. Indeed in that case we would conclude by Proposition \ref{reduction for characters} and Proposition \ref{when characters are reduced} that $J_{\chi} \neq (I,\beta)$ since we would have $\beta(i) \geq \beta_{M_{\bullet}}(i)+\beta(j)- \beta_{M_{\bullet}}(j)$, which would contradict the defining property of $\text{Max}$. Hence it must be that $i \in I_{M_{\bullet}}-I$.  But then Step 8) together with assumption (a.2) and Proposition \ref{reduction for characters} and Proposition \ref{when characters are reduced} imply again that $\chi$ does not belong to the orbit of characters $\chi'$ having $J_{\chi'}=(I,\beta)$. This ends the proof. 

Statement (b)  can be proved by the same 9 steps of part (a) of this proof, replacing each time, part (a) of Lemma \ref{units in a killer vector} with part (b) of Lemma \ref{units in a killer vector}. 
\end{proof}
\end{theorem}

\section{$U_1$ as a filtered module}\label{U1 as filtered module}
In this section we apply the results of Section \ref{filtered modules} to classify the possible structures of $U_{\bullet}$ as a filtered $\mathbb{Z}_p$-module. Let $p$ be a prime number and let $e$ be in $(p-1)\mathbb{Z}_{\geq 1} \cup \{\infty\}$. Recall the definition of $\rho_{e,p}$ from Example \ref{main example shift}.

Let $K$ be a local field with residue characteristic $p$. Denote by $f_K$ the residue degree, $f_K=[O_K/m_K:\mathbb{F}_p]$. Denote by $\rho_K:=\rho_{e_K,p}$. Recall that $e_{\rho_K}^{*}=\frac{pe_K}{p-1}$ and $e_{\rho_K}^{'}=\frac{e_K}{p-1}$.
\begin{proposition}\label{rho for loc fields}
One has that $U_{\bullet}(K)$ is a $(f_K,\rho_K)$-quasi-free filtered $\mathbb{Z}_p$-module.
\begin{proof}
Firstly one has that $U_i/U_{i+1} \simeq_{\text{ab.gr.}} O/m$, which gives for every positive integer $i$ that $f_i(U_{\bullet}(K))=f_K$ (for a definition of $f_i(U_{\bullet}(K))$ see \ref{definition: f, defect, codefect}). Observe that the formula $(1+x)^p=1+px+\ldots +x^p$ implies that given $u \in U_i(K)$ then $u^p \in U_{\rho_K(i)}$. Moreover if $u \in U_i(K)-U_{i+1}(K)$ and $u^p \in U_{\rho_K(i)+1}$, then $pi=i+e_K$, which implies that $i=e'_{\rho_K}$. So we have firstly that $\rho_{U_{\bullet}(K)} \geq \rho_K$ (for a definition of $\rho_{U_{\bullet}(K)}$ see subsection \ref{rho}), which means that $U_{\bullet}(K)$ is a $\rho_K$-filtered-$\mathbb{Z}_p$-module (see subsection \ref{shit}), and secondly that $\text{defect}_{U_{\bullet}(K)}(i)=\text{codefect}_{U_{\bullet}(K)}(i)=0$ for every positive integer $i \neq e'_{\rho_K}$ (for a definition of $\text{defect}_{U_{\bullet}(K)}(i)$ and $\text{codefect}_{U_{\bullet}(K)}(i)$, see \ref{definition: f, defect, codefect}). On the other hand we know that $\mu_p(U_1(K))$ is a cyclic group. Thus we conclude by Proposition \ref{quasi free via torsion}.
\end{proof}
\end{proposition}
Therefore we deduce the following. 
\begin{theorem} \label{no torsion}
One has that $U_{\bullet}(K)$ is a free $(f_K,\rho_K)$-filtered module if and only if $\mu_{p}(K)=\{1\}$. In other words, $U_{\bullet}(K) \simeq_{\emph{filt}} M_{\rho_K}^{f_K}$ if and only if $\mu_{p}(K)=\{1\}$.
\begin{proof}
This follows immediately from Proposition \ref{charact. free-filt-mod}, Proposition \ref{quasi free via torsion} and Corollary \ref{where is the torsion} combined.
\end{proof}
\end{theorem}
If instead $\mu_p(K) \neq \{1\}$ the following holds.
\begin{theorem}\label{mup neq 1}
Let $K$ be a local field with $\mu_p(K) \neq \{1\}$. Then there is a \emph{unique} $(I_K,\beta_K) \in \emph{Jump}_{\rho_K}^{*}$ such that
 $$U_1 \simeq M_{\rho_K}^{f_K-1} \oplus (M_{\rho_K}^{*}/\mathbb{Z}_pv_{(I_K,\beta_K)})$$ 
 as filtered $\mathbb{Z}_p$-module.
\begin{proof}
This follows immediately from Proposition \ref{rho for loc fields} and Theorem \ref{classification of quasi free} combined. 
\end{proof}
\end{theorem}
We now fix $e_K=e$, and therefore we have $\rho_{K}=\rho_{e,p}$. Fix as well $f_K=f$. Our next goal is to show that every $(\rho_{e,p},f)$-quasi-free filtered module can be realized as $U_{\bullet}(K)$ for some $K$, a totally ramified degree $\frac{e}{p-1}$ extension of $\mathbb{Q}_{p^f}(\zeta_p)$. In view of Theorem \ref{classification of quasi free}, this is tantamount to prove that every jump set realizable from a filtered module can be realized by a local field. Recall from Theorem \ref{classification of quasi free} that the latter are precisely the admissible extended $\rho_{e,p}$-jump sets. For a definition of these jump sets see the discussion immediately before Theorem \ref{classification of quasi free}.
\begin{theorem}\label{admissible j.s. occur}
Let $(I,\beta)$ be an extended admissible $\rho_{e,p}$-jump set. 
Then there is a totally ramified extension $K/\mathbb{Q}_{p^f}(\zeta_p)$ with $e_K=e$ and with
$$ (I_K,\beta_K)=(I,\beta).$$

\end{theorem}
During the proof we will make use of the two propositions that follow below. Recall that if $\zeta_p \in K$, then the extension $L/K:=K(\sqrt[p]{U_{\frac{pe_K}{p-1}}/U_{\frac{pe_K}{p-1}+1}})/K$ is the unique unramified extension of degree $p$ of $K$. Indeed $[L:K]=p$, so if $e_{L/K}>1$ then $e_{L/K}=p$. Observe that the inclusion $U_{pe_K}(K) \subseteq U_{pe_L}(L)$ would, in case that $e_{L/K}=p$, induce an isomorphism $U_{\frac{pe_K}{p-1}}(K)/U_{\frac{pe_K}{p-1}+1}(K) \to U_{\frac{pe_L}{p-1}}(L)/U_{\frac{pe_L}{p-1}+1}(L)$, which, by construction would imply that $\text{codefect}_{U_{\bullet}(L)}(e_{L}^{'})=0$, which is impossible since $\zeta_p \in L$. So it must be that $e_{L/K}=1$ and $f_{L/K}=[L:K]$. 
\begin{proposition}\label{when the star is in I}
Let $K$ be a finite extension of $\mathbb{Q}_p(\zeta_p)$. Then $e_{K}^{*} \in I_K$ if and only if $K(\sqrt[p]{\mu_{p^{\infty}}(K)})/K$ is unramified. 
\begin{proof}
Let $\zeta_{p^j}$ be a generator of $U_1(K)_{\text{tors}}$. Thanks to Proposition \ref{break points of g}, we have that $e_{K}^{*} \in I_K$ if and only if $w_{U_1(K)/U_1(K)^p}(\zeta_{p^j})=\frac{pe_K}{p-1}$. On the other hand this is equivalent to $K(\zeta_{p^{j+1}})=K(\sqrt[p]{U_{\frac{pe_K}{p-1}}/U_{\frac{pe_K}{p-1}+1}})$, which, as explained just above this proposition, is the unique unramified degree $p$ extension of $K$. 
\end{proof}
\end{proposition}
Let $j$ be a positive integer. The following notation will be helpful. Consider the compositum extension $\mathbb{Q}_{p^{pf}}(\zeta_{p^j})\cdot \mathbb{Q}_{p^f}(\zeta_{p^{j+1}})/\mathbb{Q}_{p^f}(\zeta_{p^j})$, which is a Galois extension with Galois group $C_p \times C_p$. So one is provided with $p+1$ degree $p$ sub-extensions. We denote the unique unramified one as $\mathbb{Q}_{p^f}(\zeta_{p^j})(0)$ (which of course is just $\mathbb{Q}_{p^{pf}}(\zeta_{p^j})$). Further we list the $p-1$ totally ramified ones without an element of order $p^{j+1}$ as $\mathbb{Q}_{p^f}(\zeta_{p^j})(i)$ with $i$ running through $\{1,\ldots ,p-1\}$. And we will sometimes make use of an extended notation for $i=p$, by letting $\mathbb{Q}_{p^f}(\zeta_{p^j})(p):=\mathbb{Q}_{p^f}(\zeta_{p^{j+1}})$. 
\begin{proposition} \label{when star is in  for tot ram}
Let $j$ be a positive integer. Let $K$ be a totally ramified extension of $\mathbb{Q}_{p^f}(\zeta_p)$ with $e_K=:e$. Then the following are equivalent: \\
\emph{(1)} $e^{*} \in I_K$ and $\beta_K(e^{*})=j$. \\
\emph{(2)} There is exactly one $i \in \{1,\ldots ,p-1\}$ such that $K$ contains $\mathbb{Q}_{p^f}(\zeta_{p^j})(i)$.
\begin{proof}
$(1) \to (2)$ Thanks to Proposition \ref{when the star is in I}, we have that $(1)$ implies that $K(\zeta_{p^{j+1}})/K$ is unramified, thus we have that $K(\zeta_{p^{j+1}})/\mathbb{Q}_{p^f}(\zeta_{p^j})$ contains $\mathbb{Q}_{p^f}(\zeta_{p^{j+1}}) \cdot \mathbb{Q}_{p^{pf}}(\zeta_{p^j})$. But this last one must then intersect $K$ non-trivially, otherwise one would have $[K(\zeta_{p^{j+1}}):K]=p^2$, which is impossible. At the same time the intersection cannot be $\mathbb{Q}_{p^{pf}}(\zeta_{p^j})$ because $f_K=f$, and it cannot be $\mathbb{Q}_{p^f}(\zeta_{p^{j+1}})$. Indeed we have $\beta_K(e^{*})=j$ and Proposition \ref{how much torsion} implies that $p^j=\#\mu_{p^{\infty}}(K)$. So there must be an $i \in \{1,\ldots ,p-1\}$ such that $K$ contains $\mathbb{Q}_{p^f}(\zeta_{p^j})(i)$. But there must be exactly one since otherwise the whole extension $\mathbb{Q}_{p^f}(\zeta_{p^{j+1}}) \cdot \mathbb{Q}_{p^{pf}}(\zeta_{p^j})$ would be in $K$, which has been already explained to be not possible.

$(2) \to (1)$ We have that $K(\sqrt[p]{U_1(K)_{\text{tors}}})/K$ contains $\mathbb{Q}_{p^f}(\zeta_{p^j})(i)(\zeta_{p^{j+1}}) \supset \mathbb{Q}_{p^{pf}}$, thus one concludes that $K(\sqrt[p]{U_1(K)_{\text{tors}}})/K$ is unramified and by Proposition \ref{when the star is in I} one concludes that $e^{*} \in I_K$. Moreover we must have that $\beta_K(e^{*})=j$. Indeed $K$ contains in particular $\zeta_{p^j}$, which, by Proposition \ref{how much torsion}, implies that $\beta_K(e^{*}) \geq j$. If we would have $\beta_K(e^{*})>j$ then, still by Proposition \ref{how much torsion}, the field $K$ would contain also $\mathbb{Q}_{p^f}(\zeta_{p^{j+1}})$. Hence $K$ would contain the compositum of $\mathbb{Q}_{p^f}(\zeta_{p^{j+1}})$ and $\mathbb{Q}_{p^f}(\zeta_{p^j})(i)$. Therefore $K$ would contain the field  $\mathbb{Q}_{p^{pf}}$ providing a contradiction with $f_K=f$. Hence, since $\beta_K(e^{*}) \geq j$ and $\beta_K(e^{*})<j+1$, it must be that $\beta_K(e^{*})=j$.
\end{proof}
\end{proposition}
In particular we derive the following:
\begin{corollary} \label{I,beta for twisted cyclotomic}
Let $j,f$ be positive integers, and $i \in \{1,\ldots ,p-1\}$. Then $$I_{\mathbb{Q}_{p^f}(\zeta_{p^j})(i)}=\{1,p^{j+1}\}, \ \beta_{\mathbb{Q}_{p^f}(\zeta_{p^j})(i)}(1)=j+1 \ \text{and} \ \beta_{\mathbb{Q}_{p^f}(\zeta_{p^j})(i)}(p^{j+1})=j.$$
\begin{proof}
Since the jump set must be admissible, we know that $1 \in I_{\mathbb{Q}_{p^f}(\zeta_{p^j})(i)}$ with $\beta_{\mathbb{Q}_{p^f}(\zeta_{p^j})(i)}(1)=j+1$. By Proposition \ref{when the star is in I}, we know that $p^{j+1} \in I_{\mathbb{Q}_{p^f}(\zeta_{p^j})(i)}$ with $\beta_{\mathbb{Q}_{p^f}(\zeta_{p^j})(i)}(p^{j+1})=j$. Moreover we certainly have that $1=\text{min}(I_{\mathbb{Q}_{p^f}(\zeta_{p^j})(i)})$ and $p^{j+1}= \text{max}(I_{\mathbb{Q}_{p^f}(\zeta_{p^j})(i)})$ since $1$ and $p^{j+1}$ are respectively the smallest and the largest elements of $T_{\rho}^{*}$, for $\rho:=\rho_{e,p}$ with $e:=p^j(p-1)$. Moreover the very beginning of this proof gives us in particular that $\beta_{\mathbb{Q}_{p^f}(\zeta_{p^j})(i)}(1)-\beta_{\mathbb{Q}_{p^f}(\zeta_{p^j})(i)}(p^{j+1})=1$. Therefore, recalling that the map $\beta$ is strictly decreasing (by definition of a jump set), we must conclude that between $1$ and $p^{j+1}$ no other value of $I_{\mathbb{Q}_{p^f}(\zeta_{p^j})(i)}$ can be found.  This gives us the desired conclusion.
\end{proof}
\end{corollary}
Now we can proceed proving Theorem \ref{admissible j.s. occur}. Take $(I,\beta)$ an extended admissible $\rho_{e,p}$-jump set. We distinguish two cases, depending on whether $e^{*} \in I$. First assume that $e^{*} \not \in I$. Next define the polynomial
$$G(x):=\prod_{i \in I}(1+x^i)^{p^{\beta(i)-1}}-\zeta_p \in \mathbb{Q}_{p^f}(\zeta_p)[x].
$$ 
Using the fact that $(I,\beta)$ is admissible (for a definition see immediately before the statement of Theorem \ref{classification of quasi free}) one finds that the Newton polygon of $G(x)$ consists of the segment connecting $(0,1)$ and $(\frac{e}{p-1},0)$ continued with a horizontal segment starting from $(\frac{e}{p-1},0)$. Therefore there exists a degree $\frac{e}{p-1}$ Eisenstein polynomial $g(x) \in \mathbb{Q}_{p^f}(\zeta_p)$, such that $g(x)$ divides $G(x)$. Define
$$K:=\mathbb{Q}_{p^f}(\zeta_p)[x]/g(x).
$$
Clearly $\pi:=x$ is a uniformizer in $K$. Moreover we have that
$$\prod_{i \in I} (1+\pi^i)^{p^{\beta(i)}}=1
$$
with $\text{v}_{K}((1+\pi^i)-1)=i$, thus giving
$$(I_{K},\beta_K)=(I,\beta),
$$
thanks to Corollary \ref{The relation between the units}.
Now suppose that $e^{*} \in I$ and write $j:=\beta(e^{*})$. We prove that $p^j(p-1)|e$. Indeed we have that $\text{min}(I)<e^{*}$ giving that $\beta(\text{min}(I))\geq j+1$. Thus, since we have that $\rho^{\beta(\text{min}(I))}(\text{min}(I))=p^{\beta(\text{min}(I))}(\text{min}(I))=\frac{pe}{p-1}$, we obtain that $p^j(p-1)|e$. Next, pick $u_1,u_2 \in \mathbb{Q}_{p^f}(\zeta_{p^j})(1)$ such that
$$\text{v}_{\mathbb{Q}_{p^f}(\zeta_{p^j})(1)}(u_1-1)=1, \ \text{v}_{\mathbb{Q}_{p^f}(\zeta_{p^j})(1)}(u_2-1)=p^{j+1}, \ u_2 \not \in (\mathbb{Q}_{p^f}(\zeta_{p^j})(1))^{*p}
$$
and 
$$u_1^{p^{j+1}}u_2^{p^j}=1
$$
as guaranteed by Corollary \ref{I,beta for twisted cyclotomic}. Now define
$$G^{*}(x):=\prod_{i \in I, i<e^{*}}(1+x^i)^{p^{\beta(i)-j-1}}-u_1 \in \mathbb{Q}_{p^f}(\zeta_{p^j})(1)[x].
$$
Using the fact that $(I,\beta)$ is admissible one finds that the Newton polygon of $G^{*}(x)$ consists of the segment connecting $(0,1)$ and $(\frac{e}{p^j(p-1)},0)$ continued with a horizontal segment starting from $(\frac{e}{p^j(p-1)},0)$. Therefore there exists a degree $\frac{e}{p^j(p-1)}$ Eisenstein polynomial $g^{*}(x) \in \mathbb{Q}_{p^f}(\zeta_{p^j})(1)[x]$ such that $g^{*}(x)$ divides $G^{*}(x)$. Define
$$\tilde{K}:=\mathbb{Q}_{p^f}(\zeta_{p^j})(1)[x]/g^{*}(x).
$$
Clearly $\pi:=x$ is a uniformizer in $\tilde{K}$. Moreover we have that 
$$(\prod_{i \in I, i<e^{*}}(1+\pi^i)^{p^{\beta(i)}})u_2^{p^j}=1
$$
with $\text{v}_{\tilde{K}}((1+\pi^i)-1)=i$ for each $i \in I$, with $i<e^{*}$ and with $\text{v}_{\tilde{K}}(u_2-1)=e^{*}$. Thus, in order to apply Corollary \ref{The relation between the units}, we are only left with checking that $u_2 \not \in \tilde{K}^{*p}$. But this follows at once from the fact that $\mathbb{Q}_{p^{pf}} \subseteq \mathbb{Q}_{p^f}(\zeta_{p^j})(1)(\sqrt[p]{u_2})$ and the fact that $g^{*}(x)$ is an Eisenstein polynomial and thus $\tilde{K}/\mathbb{Q}_{p^f}(\zeta_{p^j})(1)$ is totally ramified. This ends the proof of Theorem \ref{admissible j.s. occur} and therefore of Theorem \ref{realizable j.s. are realized} in the Introduction.
\section{Upper jumps of cyclic extensions}\label{wild extension}
In this section we use Theorem \ref{classification of characters}, together with Theorem \ref{rho for loc fields}, to establish Theorem \ref{classification wild characters}, a classification in terms of jump sets for the possible sets of upper jumps of a cyclic wild extension of a local field $K$. We next prove combinatorially that the classification obtained is equivalent to that obtained by Miki \cite{Miki}, Maus \cite{Maus} and Sueyoshi \cite{Sueyoshi}: in this way those results are \emph{deduced} from Theorem \ref{classification wild characters}. Finally we give a sense of how in practice the classification of Theorem \ref{classification wild characters} may look, by examining it for several possible values of the triple $((I,\beta),f,p)$, and in particular we do so for the most typical occurrences of $(I,\beta)$ in the sense of Theorem \ref{counting}. We also show that for $K/\mathbb{Q}_p(\zeta_p)$ totally ramified, the knowledge of the filtered $\mathbb{Z}_p$-module $U_{\bullet}(K)$ is equivalent to the knowledge of all possible sets of upper jumps of cyclic wild totally ramified extensions of $K$ (see Corollary \ref{reconstructing}).

\subsection{Classification of possible sets of jumps} 
In the rest of the section $K$ will denote as usual a local field of residue characteristic a prime number denoted by $p$. We fix $K^{\text{sep}}$ a separable closure of $K$ and we denote by $G_K:=\text{Gal}(K^{\text{sep}}/K)$ the absolute Galois group of $K$. Let $H$ be a normal closed subgroup of $G_K$. Recall that for every $\alpha \in \mathbb{R}_{\geq 0}$ the Galois group $G_K/H$ is provided with a subgroup $(G_K/H)^{\alpha}$ via the so-called upper ramification filtration (see \cite{Local fields}).  Let $L/K$ be a finite cyclic totally ramified extension of $K$, with degree a power of $p$. Denote by $G$ the Galois group $\text{Gal}(L/K)$. A number $\alpha \in \mathbb{R}_{\geq 0}$ is said to be an \emph{upper jump} for $L/K$ if $G^{\alpha} \supsetneq G^{\alpha+\epsilon}$ for each $\epsilon>0$. We denote by $J(L/K)$ the set of upper jumps for $L/K$. From the Hasse--Arf Theorem (see \cite{Local fields}) we have that $J(L/K) \subset \mathbb{Z}_{\geq 1}$. We denote by $\mathcal{J}_K$ the collection of all such subsets of $\mathbb{Z}_{\geq 1}$ as $L$ varies among all cyclic, $p$-power, totally ramified extensions of $K$. The set $\mathcal{J}_{K}$ can be also described as follows. We consider all totally ramified continuous homomorphisms 
$$\chi:G_{K} \to \mathbb{Q}_p/\mathbb{Z}_p,
$$
where $\chi$ is said to be totally ramified if the corresponding field extension is totally ramified. The set of upper jumps for $\chi$ are the $\alpha \in \mathbb{R}_{\geq 0}$ such that $\chi(G_K^{\alpha}) \neq \chi(G_K^{\alpha+\epsilon})$ for all $\epsilon>0$. This set is denoted by $J_{\chi}$. One has that $J_{\chi}=J((K^{\text{sep}})^{\text{ker}(\chi)}/K)$, so that $\mathcal{J}_K$ consists of the collection of all $J_{\chi}$ as $\chi$ varies among continuous totally ramified characters $\chi:G_K \to \mathbb{Q}_p/\mathbb{Z}_p$. Of course the set of such continuous totally characters $\chi:G_K \to \mathbb{Q}_p/\mathbb{Z}_p$ can be equivalently described as the set of all $\chi:G_K^{\text{ab}} \to \mathbb{Q}_p/\mathbb{Z}_p$ continuous totally ramified. Finally it is not difficult to see that $\mathcal{J}_K$ is also the collection of all $J_{\chi}$ for all continuous homomorphisms $\chi:(G_{K}^{\text{ab}})^{1} \to \mathbb{Q}_p/\mathbb{Z}_p$. On the other hand $(G_{K}^{\text{ab}})^{1} \simeq_{\mathbb{Z}_p\text{-filt}}U_{\bullet}(K)$ via the Artin local reciprocity law. Therefore we see that the definition of $\mathcal{J}_K$ given in this section is equivalent to the one given in the introduction: we have $\mathcal{J}_K=\mathcal{J}_{U_{\bullet}(K)}$, where the right hand side is defined at the beginning of Section \ref{characters}. Therefore we are in a position to apply the results of Section \ref{characters}, notably Theorem \ref{classification of characters}. To make the statements simpler we first make a definition. Let $\rho$ denote a general shift with $T_{\rho}$ finite, $f$ a positive integer, and $p$ a prime number. Let moreover $(I,\beta)$, $(I',\beta')$ be in $\text{Jump}_{\rho}^{*}$.
\begin{definition} \label{incompatibility}
We say that $(I,\beta)$ is $((I',\beta'),f,p)$-incompatible if the following conditions hold. 

$(1)$ The set $I \cap I'$ is non-empty. Moreover the subset $\text{Max}((I,\beta),(I',\beta'))$ of $I \cap I'$ consisting of those $i$ in $I \cap I'$ where $\beta(i)-\beta'(i)$ is strictly positive and assumes the maximal possible value, which we denote by $c$, has precisely one element if $p>2$ and an odd number of elements if $p=2$.  

$(2)$ If $f>1$ then $\text{Max}((I,\beta),(I',\beta'))=\{e_{\rho}^{*}\}$. 

$(3)$ Given any $i \in I'-I$, there is no $j \in I$ such that $(j,\beta'(j)) \geq_{\rho} (i,c+\beta'(i))$.

We say that $(I,\beta)$ is $((I',\beta'),f,p)$-compatible if it is not $((I',\beta'),f,p)$-incompatible. 
\end{definition}
Combining Theorem \ref{character for free modules} and Theorem \ref{classification of characters} together with Theorem \ref{rho for loc fields} we obtain the following. 
\begin{theorem} \label{classification wild characters}
Suppose that $\mu_p(K)=\{1\}$. Then $\mathcal{J}_K=\emph{Jump}_{\rho_K}$. Suppose that $\mu_p(K) \neq \{1\}$. Then $\mathcal{J}_K$ consists precisely of the elements of $\emph{Jump}_{\rho_K}^{*}$ that are $((I_K,\beta_K),f_K,p)$-compatible. 
\end{theorem}
We conclude this subsection by proving that the notion of $((I',\beta'),f,p)$-incompatibility is equivalent to a slightly simpler criterion. This is given by the next proposition, which will be repeatedly applied in the next subsection. We make first the following definition. 
\begin{definition} \label{floors in I}
Let $a$ be a positive integer and let $(I,\beta)$ be an extended $\rho$-jump set with $I \neq \emptyset$. Suppose that $a \geq \text{min}(I)$, then we denote by $ \lfloor a \rfloor_{I}$ the largest element $i$ of $I$ such that $i \leq a$. Suppose that $a \leq \text{max}(I)$, then we denote by $\lceil a \rceil_{I}$ the smallest element $i$ of $I$ such that $a \leq i$. 

\end{definition} 
Let now $(I,\beta)$ and $(I',\beta')$ be two extended $\rho$-jump sets.
\begin{proposition}\label{a cheaper criterion for incompatibility}
The jump set $(I,\beta)$ is $((I',\beta'),f,p)$-incompatible if and only if the following holds.

\emph{(a)} Conditions $(1)$ and $(2)$ from definition \ref{incompatibility} hold. If that is the case, let $c(I,I'):=\beta(i_0)-\beta'(i_0)$ for any $i_0 \in \emph{Max}((I,\beta),(I',\beta'))$. 

\emph{(b)} For every point $i \in I'-\emph{Max}((I,\beta),(I',\beta'))$, we have that
$$\beta'(i)+c(I,I')>\beta(\lceil i \rceil_{I}),
$$
whenever $i \geq \emph{min}(I)$ and
$$\rho^{\beta'(i)+c(I,I')}(i)>\rho^{\beta(\lfloor i \rfloor_{I})}(\lfloor i \rfloor_{I}),
$$
whenever $i \leq \emph{max}(I)$.
\begin{proof}
This follows immediately by noticing that condition $(3)$ of definition \ref{incompatibility}, requires only the comparisons with $\lfloor i \rfloor_{I}$ and $\lceil i \rceil_{I}$,  as soon as they are defined, since the two inequalities in part (b) of the present statement must certainly hold, but they trivially imply all the others since $j \mapsto \beta(j)$ is strictly decreasing and $j \mapsto \rho^{\beta(j)}(j)$ is strictly increasing, by definition of a jump set.
\end{proof}

\end{proposition}
\subsection{Comparison with Miki-Maus-Sueyoshi}
In this subsection we give a direct combinatorial verification that Theorem \ref{classification wild characters} and the main Theorem of \cite{Maus} are indeed classifying precisely the same sets. Of course this follows also from applying both theorems, but both criteria being of a purely combinatorial form, it is natural to provide a combinatorial proof of their equivalence, not relying on local fields. As an upshot we can \emph{deduce} Miki's classification from Theorem \ref{classification wild characters} and the bit extra of combinatorial work of this subsection. Moreover the combinatorial nature of the equivalence between the two classification is highlighted from the fact that it follows from a statement about a general shift, see Proposition \ref{inadequacy equivalent to incompatibility}. Recall indeed that the case discussed in the present section is only a very special case of the classification we provide Theorem \ref{classification of characters}, which is about a general $(f,\rho)$-quasi-free $R$-module (see Definition \ref{def of quasi free}), where $R$ is any complete DVR, $f$ is any positive integer and $\rho$ is a general shift. In the case $\mu_p(K)=\{1\}$ the two descriptions are literally equal. So we pass to examine the case $\mu_p(K) \neq \{1\}$, where both Theorems say that $\mathcal{J}_K \subseteq \text{Jump}_{\rho_K}^{*}$ and they both provide a criterion for an element of $\text{Jump}_{\rho_K}^{*}$ to be realizable as the set of jumps of a character. In the case of Theorem \ref{classification wild characters}, this is precisely the notion of being $((I_K,\beta_K),f_K,p)$-compatible. For the convenience of the reader we recap the formulation of Miki's criterion as stated in \cite{Sueyoshi} in terms of a general definition, valid for any shift $\rho$ with $T_{\rho}$ finite. As usual, let $p$ denote a prime number and $f$ a positive integer. Let moreover $(I,\beta)$ and $(I',\beta')$ be in $\text{Jump}_{\rho}^{*}$, with both $I,I'$ being non-empty.
\begin{definition} \label{inadequate}
We say that $(I,\beta)$ is $((I',\beta'),f,p)$-inadequate if the following holds. Write $J_{(I,\beta)}=\{t_1, \ldots ,t_m\}$ and $J_{(I',\beta')}=\{\lambda_1, \ldots ,\lambda_l\}$ (see immediately above Proposition \ref{equivalent data} for the notation $J_{(I',\beta')}$) with $\{t_i\}_{1 \leq i \leq m}$ and  $\{\lambda_i\}_{1 \leq i \leq l}$ written in increasing order. Write $s=\beta'(\text{max}(I'))$. Then there is a positive integer $L$ with $L<m-(s-1)$ such that the sequences $\{x_i\}_{0 \leq i < l-(s-1)}, \{y_i\}_{0 \leq i < l-(s-1)}$ defined as $x_i:=t_{L-i}, y_i:=\lambda_{l-i-(s-1)}$, with $x_i=0$ when $L \leq i$, satisfy the following condition. Whenever $y_i \in I'$, then $x_i \leq y_i$, with equality occurring, among these inequalities, precisely once if $p>2$, and an odd number of times if $p=2$. Moreover, in case $f>2$, equality occurs precisely once, for all $p$, and it occurs for $i=0$ with $x_0=y_0=e_{\rho}^{*} \in I \cap I'$ (i.e. $x_1,y_1<e_{\rho}^{'}$). Recall that $e_{\rho}^{*}=\text{max}(T_{\rho})+1$ and that $e_{\rho}^{'}$ is the unique positive integer such that $\rho(e_{\rho}^{'})=e_{\rho}^{*}$.

We say that $(I,\beta)$ is $((I',\beta'),f,p)$-adequate if it is not $((I',\beta'),f,p)$-inadequate.
\end{definition}
\begin{remark}
In \cite{Sueyoshi}, the final condition requires only that $x_1<e_{\rho}^{'}$ (i.e. that $e_{\rho}^{*} \in I$) because in that case $(I',\beta')$ is admissible (since it is the jump set of a local field, see definition right after Theorem \ref{realizable j.s. are realized}), so the condition $y_0=e_{\rho}^{*}$, which is equivalent to $e_{\rho}^{*} \in J_{(I',\beta')}$, is actually equivalent to $e_{\rho}^{*} \in I'$. 
\end{remark}
We next furnish a direct combinatorial proof that incompatibility and inadequacy are the same notion. 
\begin{proposition} \label{inadequacy equivalent to incompatibility}
Let $\rho$ be a shift with finite $T_{\rho}$, let $p$ be a prime number and $f$ a positive integer. Let $(I,\beta)$ and $(I',\beta')$ be in $\emph{Jump}_{\rho}^{*}$. Then $(I,\beta)$ is $((I',\beta'),f,p)$-inadequate if and only if it is $((I',\beta'),f,p)$-incompatible. 
\begin{proof}
Suppose that $(I,\beta)$ is $((I',\beta'),f,p)$-inadequate. Let $L<m-(s-1)$ and the two sequences $\{x_i\}_{0 \leq i \leq l-(s-1)}, \{y_i \}_{0 \leq i \leq l-(s-1)}$ be as in definition \ref{inadequate}. Let $M$ be the set of non-negative integers $i_0$, with $i_0 \leq l-(s-1)$, $y_{i_0} \in I'$ and $x_{i_0}=y_{i_0}$: the size of $M$ must be, by definition, equal to $1$ if $p>2$, and odd if $p=2$. We claim that $\{y_i \}_{i \in M}=\text{Max}((I,\beta),(I',\beta'))$ (for a definition of $\text{Max}((I,\beta),(I',\beta'))$ see Theorem \ref{classification of characters}). We know that, in either case, $M$ is non-empty. Let $i_0$ be one of its elements. Firstly, from the fact that $L<m-(s-1)$ we deduce precisely that the set $J_{(I,\beta)} \cap [y_{i_0},\infty)$ has strictly more elements than $J_{(I',\beta')} \cap [y_{i_0},\infty)$. In other words $y_{i_0} \in I \cap I'$ with $\beta(y_{i_0})-\beta'(y_{i_0})>0$. Next let $0 \leq i \leq l-(s-1)$ be any other index such that $y_i \in I \cap I'$. Assume $y_i>y_{i_0}$, i.e. that $i<i_0$. From the fact that $x_i \leq y_i$, we conclude that in the interval $[y_{i_0},y_i]$ there are at least as many points of $J_{(I,\beta)}$ as there are points of $J_{(I',\beta')}$, which amounts to $\beta(y_{i_0})-\beta(y_i) \geq \beta'(y_{i_0})-\beta'(y_i)$, which can be rewritten as $\beta(y_{i_0})-\beta'(y_{i_0}) \geq \beta(y_{i})-\beta'(y_i)$, with equality iff $i \in M$. A completely analogous reasoning in the case $i>i_0$ brings us to the same conclusion. In other words we have just shown that $y_{i_0} \in \text{Max}((I,\beta),(I',\beta'))$ and all other $i \in M$ are precisely the $i$ such that $y_i \in \text{Max}((I,\beta),(I',\beta'))$. Therefore we conclude by the very definition of inadequacy that conditions $(1)-(2)$ of definition \ref{incompatibility} hold. We are left with proving condition $(3)$. Let $i_1$ be an index such that $y_{i_1} \in I'-I$. Take $i_0 \in M$, and suppose $i_1<i_0$. Since $|J_{(I,\beta)} \cap [y_{i_0},x_{i_1}]|=|J_{(I',\beta')} \cap [y_{i_0},y_{i_1}]|$, we have that $\beta(y_{i_0})-\beta(\lceil y_{i_1} \rceil_{I})>\beta'(y_{i_0})-\beta'(y_{i_1})$, which can be rewritten as $c+\beta'(y_{i_1})>\beta(\lceil y_{i_1} \rceil_{I})$. This last inequality is precisely the first of the two inequalities in Proposition \ref{a cheaper criterion for incompatibility}. Next, always assuming $i_1<i_0$, consider the two possible cases: $x_{i_1}<\lfloor y_{i_1} \rfloor_{I}$ or $\lfloor y_{i_1} \rfloor_{I} \leq  x_{i_1}<\lceil y_{i_1} \rceil_{I}$. In the first case observe that $\beta(y_{i_0})-\beta(\lfloor y_{i_1} \rfloor_{I})>\beta'(y_{i_0})-\beta'(y_{i_1})$, which can be recast as $c+\beta'(y_{i_1})>\beta(\lfloor y_{i_1} \rfloor_{I})$. This last inequality trivially implies that $\rho^{c+\beta'(y_{i_1})}(y_{i_1})>\rho^{\beta(\lfloor y_{i_1} \rfloor_{I})}(\lfloor y_{i_1} \rfloor_{I})$, since $\rho$ is strictly increasing. So in the first case one, trivially, obtains the second inequality of Proposition \ref{a cheaper criterion for incompatibility}. In the second case observe that $(\beta'(y_{i_0})-\beta'(y_{i_1}))-( \beta(x_{i_0})-\beta(\lfloor y_{i_1} \rfloor_{I}))=v_{\rho}(x_{i_0})$, i.e. $\rho^{\beta'(y_{i_0})-\beta'(y_{i_1}))-( \beta(y_{i_0})-\beta(\lfloor y_{i_1} \rfloor_{I})}(\lfloor y_{i_1} \rfloor_{I})=x_{i_1}<y_{i_1}$, which can be rewritten as $\rho^{\beta(\lfloor y_{i_1} \rfloor_{I})}(\lfloor y_{i_1} \rfloor_{I})<\rho^{c+\beta'(y_{i_1})}(y_{i_1})$. This last inequality is precisely the second inequality in Proposition \ref{a cheaper criterion for incompatibility}. The case $i_1>i_0$ can be treated in the same way. 
Altogether this proves that $(I,\beta)$ is $((I',\beta'),f,p)$-incompatible.

The proof of the converse implication proceeds analogously, and basically it can be obtained by inverting the above arguments. 
\end{proof}
\end{proposition}
From Proposition \ref{inadequacy equivalent to incompatibility} we can infer the main Theorem in \cite{Sueyoshi}.
\begin{theorem} (Miki's Theorem) Suppose that $\mu_p(K) \neq \{1\}$. Then $\mathcal{J}_K$ consists precisely of the elements of $\emph{Jump}_{\rho_K}^{*}$ that are $((I_K,\beta_K),f_K,p)$-adequate. 
\begin{proof}
This follows immediately from Theorem \ref{classification wild characters} and Proposition \ref{inadequacy equivalent to incompatibility} together. 
\end{proof}

\end{theorem}
\subsection{Examples and special cases}
We begin by providing several cases where Theorem \ref{classification wild characters} specializes to something much simpler, the interesting case being clearly that $\mu_p(K) \neq \{1\}$, which we will assume in the rest of this subsection.
\begin{corollary} \label{cor1}
Let $(I,\beta) \in \emph{Jump}_{\rho}^{*}$, with $I \cap I_K= \emptyset$. Then $(I,\beta) \in \mathcal{J}_K$.
\begin{proof}
Indeed, in this case condition $(1)$ of definition \ref{incompatibility} cannot possibly hold if $(I',\beta'):=(I_K,\beta_K), f:=f_K, p:=\text{char}(O_K/m_K)$. Therefore $(I,\beta)$ is $((I_K,\beta_K),f,p)$-compatible and the conclusion follows from Theorem \ref{classification wild characters}.
\end{proof}
\end{corollary} 
As soon as $f_K \geq 2$ we can say the following.
\begin{corollary} \label{corollary2}
Suppose that $f_K \geq 2$. Then the following facts holds. 

\emph{(a)} Suppose that $e_{\rho_K}^{*} \not \in I_K$. Then $\mathcal{J}_K=\emph{Jump}_{\rho_K}^{*}$.

\emph{(b)} Suppose that $e_{\rho_K}^{*} \in I_K$. Then $\emph{Jump}_{\rho_K} \subseteq \mathcal{J}_K \subsetneq \emph{Jump}_{\rho_K}^{*}$. 

\emph{(c)} Suppose that $e_{\rho_K}^{*} \in I_K$. Then each $(I,\beta) \in \emph{Jump}_{\rho_K}^{*}$ with $e_{\rho_K}^{*} \in I$ and $\beta(e_{\rho_K}^{*}) \leq \beta_K(e_{\rho_K}^{*})$ is in $\mathcal{J}_{K}$. 
\begin{proof}
Let $(I,\beta)$ be in $\text{Jump}_{\rho_K}^{*}$. Then condition $(2)$ of definition \ref{incompatibility} cannot possibly hold if $(I',\beta'):=(I_K,\beta_K), f:=f_K, p:=\text{char}(O_K/m_K)$, therefore by Theorem \ref{classification wild characters}, we obtain that $(I,\beta) \in \mathcal{J}_K$, thus giving (a). Similarly if $e_{\rho_K}^{*} \not \in I$, which amounts to saying that $(I,\beta) \in \text{Jump}_{\rho_K}$, then condition $(2)$ from definition \ref{incompatibility} cannot possibly hold, giving $\text{Jump}_{\rho_K} \subseteq \mathcal{J}_K$ from (b). The inclusion $\mathcal{J}_K \subseteq \text{Jump}_{\rho_K}^{*}$ always holds, thanks to Theorem \ref{classification wild characters}, so, to conclude the proof of (b), we only need to prove the strict inclusion, i.e. to provide, under the conditions of (b), an element of $\text{Jump}_{\rho_K}^{*}$ that is not in $\mathcal{J}_K$. Consider $(\{e_{\rho_K}^{*}\},(e_{\rho_K}^{*} \mapsto n))$ with $n>\beta_K(e_{\rho_K}^{*})$: it trivially satisfies condition (a) from Proposition \ref{a cheaper criterion for incompatibility}. Condition (b) amounts to saying that for any $i \in I_K -\{e_{\rho_K}^{*}\}$ we need to have $n-\beta_K(e_{\rho_K}^{*})+\beta_K(i)>n$. This last inequality is equivalent to the inequality $\beta_K(i)-\beta_K(e_{\rho_K}^{*})>0$ and this inequality holds by definition of jump set. Hence we conclude by Theorem \ref{classification wild characters} and Proposition \ref{a cheaper criterion for incompatibility} that $(\{e_{\rho_K}^{*}\},(e_{\rho_K}^{*} \mapsto n)) \not \in \mathcal{J}_K$. This concludes the proof of (b). 

For (c), notice that the assumption $\beta(e_{\rho_K}^{*})<\beta_K(e_{\rho_K}^{*})$ together with $f_K \geq 2$ makes condition $(2)$ of \ref{incompatibility} impossible to hold for $(I,\beta)$, giving by Theorem \ref{classification wild characters} that $(I,\beta) \in \mathcal{J}_K$. 
\end{proof}
\end{corollary}
If instead $f_K=1$, then there are always exceptions.
\begin{corollary} \label{acazz}
If $f_K=1$, then $\mathcal{J}_K \subsetneq \emph{Jump}_{\rho_K}^{*}$. 
\begin{proof}
We proceed as in (b) of the previous corollary. For any $i \in I_{K}$ we consider the jump set $(\{i\},(i \mapsto n))$ with $n>\beta_K(i)$. We proceed to show that this jump set is $((I_K,\beta_K),1,p)$-incompatible. Condition (a) of Proposition \ref{a cheaper criterion for incompatibility} is clearly satisfied, so we proceed to verify condition (b) of that Proposition. Taking $j \in I_K$ with $j<i$, we need to check that $\beta_K(j)+n-\beta_K(i)>n$, or equivalently that $\beta_K(j)>\beta_K(i)$. This last inequality follows from the definition of a jump set. Take now $j \in I_K$ with $j>i$, we need to check that $\rho_{K}^{\beta_K(j)+n-\beta_K(i)}(j)>\rho_K^{n}(i)$, which, $\rho_K$ being strictly increasing, reduces to $\rho_{K}^{\beta_K(i)-\beta_K(j)}(i)<j$, which follows from the definition of a jump set. Therefore we conclude from Theorem \ref{classification wild characters} and Proposition \ref{a cheaper criterion for incompatibility} that $(\{i\},(i \mapsto n)) \not \in \mathcal{J}_K$.
\end{proof} 
\end{corollary}
We remark that if we would have put $n \leq \beta_K(i)$ during the proof of Corollary \ref{acazz} we would have found, thanks to Theorem \ref{classification wild characters}, that $(\{i\},(i \mapsto n)) \in \mathcal{J}_K$, since condition $(1)$ of \ref{incompatibility} is not satisfied. This will be helpful in the next corollary. It turns out that in the case $f_K=1$ there are even enough exceptions to reconstruct the full structure of the filtered $\mathbb{Z}_p$-module $U_{\bullet}(K)$ out of $\mathcal{J}_K$. Namely we have the following. 
\begin{corollary} \label{reconstructing}
Let $K_1,K_2$ be two totally ramified extensions of $\mathbb{Q}_p(\zeta_p)$. Then $\mathcal{J}_{K_1}=\mathcal{J}_{K_2}$ if and only if $U_{\bullet}(K_1) \simeq_{\mathbb{Z}_p\emph{-filt}} U_{\bullet}(K_2)$. 
\begin{proof}
Firstly observe that $T_{\rho_K}^{*}$ consists precisely of the positive integers $i \in \mathbb{Z}_{\geq 1}$ such that $(\{i\}, (i \mapsto n)) \in \mathcal{J}_K$ for some positive integer $n$. Indeed if $i \not \in I_K$, then by Corollary \ref{cor1} any $n \in \mathbb{Z}_{\geq 1}$ is allowed. If instead $i \in I_K$ then any $n \leq \beta(i)$ will be allowed, since in that way $\text{Max}((\{i\},(i \mapsto n)),(I_K,\beta_K))=\emptyset$ (for the definition of $\text{Max}((\{i\},(i \mapsto n)),(I_K,\beta_K))=\emptyset$ see Theorem \ref{classification of characters}). Conversely, by definition of a jump set, it is clear that for any $i \in \mathbb{Z}_{\geq 1}$ such that $(\{i\}, (i \mapsto n)) \in \mathcal{J}_K$ for some positive integer $n$, one has $i \in T_{\rho_K}^{*}$. Hence $T_{\rho_K}^{*}$ can be reconstructed from $\mathcal{J}_K$, and, since $e_K=|T_{\rho_K}^{*}|-1$, we can reconstruct $e_K$ from $\mathcal{J}_K$.

Next, from the proof of the previous corollary, it is clear that under the assumption $f_K=1$, the set $I_K$ can be reconstructed from $\mathcal{J}_K$ as the set of $i \in T_{\rho_K}^{*}$ for which there exists a positive integer $n$ such that the extended $\rho_K$-jump set $(\{i\},(i \mapsto n))$ is not in $\mathcal{J}_K$. Moreover in that proof we saw that, for $i \in I_K$, the set of such integers consists precisely of the left interval $[\beta_K(i)+1,\infty) \cap \mathbb{Z}_{\geq 1}$, hence also $\beta_K$ can be reconstructed from $\mathcal{J}_K$. Hence we can reconstruct $(I_K,\beta_K)$. 

So given $K_1$ and $K_2$ as in the statement we have that $1=f_{K_1}=f_{K_2}$, and we have shown above that we have $e_{K_1}=e_{K_2}$ and so $\rho_{K_1}=\rho_{K_2}$. Moreover by the reasoning just made, from $\mathcal{J}_{K_1}=\mathcal{J}_{K_2}$ we conclude that $(I_{K_1},\beta_{K_1})=(I_{K_2},\beta_{K_2})$. Hence we conclude by Theorem \ref{mup neq 1} that $U_{\bullet}(K_1) \simeq_{\mathbb{Z}_p\text{-filt}} U_{\bullet}(K_2)$. The converse is a triviality. 
\end{proof}
\end{corollary}
In other words, for $K_1,K_2$, totally ramified extension of $\mathbb{Q}_p(\zeta_p)$, one has $\mathcal{J}_{K_1}=\mathcal{J}_{K_2}$ if and only if $e_{K_1}=e_{K_2}$ and $(I_{K_1},\beta_{K_1})=(I_{K_2},\beta_{K_2})$.

We conclude this subsection providing a more explicit description of $\mathcal{J}_K$ in a family of simple cases, namely when $|I_K| \leq 2$. Observe that thanks to Theorem \ref{counting}, the equality $|I_K|=2$ is the most typical phenomenon. If $\text{v}_{\mathbb{Q}_p}(e) \geq 2$, the probability that $|I_K|>2$ is at most $(\frac{1}{q})^{p-1} \cdot \frac{q-1}{q}$ and at least $(\frac{1}{q})^{p-1} \cdot (\frac{q-1}{q})^{2}$, while if $\text{v}_{\mathbb{Q}_p}(e) \leq 1$, then $|I_K| \leq 2$ always. Observe also that $|I_K|=1$ if and only if $K$ is a tame extension of $\mathbb{Q}_{p^{f_K}}(\zeta_{p^n})$ where $n$ is the unique element of $\beta_K(I_K)$. The classification for $|I_K|=1$ takes a very simple form. Denoting by $e_0(K)$ the part of $\frac{e_K}{p-1}$ coprime to $p$, recall that, from the definition of admissible jump sets, one has that $|I_K|=1$ if and only if $I_K=\{e_0(K)\}$. Recall by admissibility that $\beta_K(e_0(K))=\text{v}_{\mathbb{Q}_p}(e_K)+1$.
\begin{corollary}
\emph{(a)} Suppose $|I_K|=1, f_K=1$. An extended $\rho_K$-jump set $(I,\beta)$ belongs to $\mathcal{J}_K$ if and only if either $e_0(K) \not \in I$ or both $e_0(K) \in I$ and $\beta(e_0(K)) \leq \emph{v}_{\mathbb{Q}_p}(e_K)+1$. \\
\emph{(b)} Suppose $|I_K|=1, f_K \geq 2$. Then $\mathcal{J}_K=\emph{Jump}_{\rho_K}^{*}$.
\begin{proof}
(a) If $e_0(K) \not \in I$ we conclude by Corollary \ref{cor1}. If $e_0(K) \in I$ and $\beta(e_0(K)) \leq \beta_K(e_0(K))$, then condition $(1)$ of definition \ref{incompatibility} cannot possibly hold, hence we conclude by Theorem \ref{classification wild characters}. Suppose instead that $\beta(e_0(K)) > \beta_K(e_0(K))$. Then all three conditions of definition \ref{incompatibility} are trivially satisfied and we conclude by Theorem \ref{classification wild characters}, finishing the proof.

(b) This follows immediately from Corollary \ref{corollary2}, given the fact that $e_{\rho_K}^{*} \neq e_0(K) \in I_K$.
\end{proof}
\end{corollary}
We next proceed providing an explicit classification in the case $|I_K|=2, f_K=1$. 
\begin{corollary}
Suppose $|I_K|=2, f_K=1$. Write $I_K=\{e_0(K),i\}$. Let $(I,\beta) \in \emph{Jump}_{\rho_K}^{*}$. Then $(I,\beta) \in \mathcal{J}_K$ if and only if one of the following two conditions holds. \\
$(1)$ One has that $I_K \cap I=\emptyset$. \\
$(2)$ One has that $I_K \subseteq I$, with $\emph{Max}((I,\beta),(I_K,\beta_K))=I_K$ or $\emph{Max}((I,\beta),(I_K,\beta_K))= \emptyset$. 
\begin{proof}
From Corollary \ref{cor1} we see that condition $(1)$ indeed implies that $(I,\beta) \in \mathcal{J}_K$. On the other hand condition $(2)$ implies that $|\text{Max}((I,\beta),(I_K,\beta_K))|$ is even, which makes the condition $(1)$ of definition \ref{incompatibility} impossible to hold. Hence we see that condition $(2)$ also implies that $(I,\beta) \in \mathcal{J}_K$. Conversely, assume that $|\text{Max}((I,\beta),(I_K,\beta_K))|=1$ but $I_K \subseteq I$. Then conditions $(1)-(2)-(3)$ of definition \ref{incompatibility} are clearly satisfied, since $I_K-I=\emptyset$. So are left with the case $\text{Max}((I,\beta),(I_K,\beta_K))=I \cap I_K$. Suppose $I \cap I_K=\{e_0(K)\}$. Then we have to check that $\rho_K^{\beta_K(i)+\beta(e_0(K))-\beta_K(e_0(K))}(i)>\rho_K^{\beta(e_0(K))}(e_0(K))$ which is equivalent to $\rho_K^{\beta_K(e_0(K))-\beta_K(i)}(e_0(K))<i$: this follows from the definition of a jump set. Suppose that $I \cap I_K=\{i\}$. Then we have to check that $\beta_K(e_0(K))+\beta(i)-\beta_K(i)>\beta(i)$ which is saying that $\beta_K(e_0(K))>\beta_K(i)$: this follows from the definition of a jump set. 
\end{proof}
\end{corollary}
\begin{remark}
The reason why for $|I_K|=2$ one gets such a simple criterion can be learned from the proof of the previous corollary. Namely the inequality in condition (3) in definition \ref{incompatibility} will always hold when tested against elements of $\text{Max}((I,\beta),(I',\beta'))$, but if $I_K$ has two elements and $I \cap I_K$ has only one, then that is the only possible test to do. So one is left with either $I_K \subseteq I$ or $I \cap I_K= \emptyset$, where in both cases it is very easy to say what Theorem \ref{classification wild characters} prescribes. Indeed the ease of the latter case was formalized in Corollary \ref{cor1}. For convenience we formalize also the ease of the case $I_K \subseteq I$ in the following last corollary.
\end{remark}
\begin{corollary}
Suppose that $f_K=1$. Suppose that $(I,\beta) \in \emph{Jump}_{\rho_K}^{*}$, with $I_K \subseteq I$. Then $(I,\beta) \in \mathcal{J}_K$ if and only if $|\emph{Max}((I,\beta),(I_K,\beta_K))| \neq 1$ when $\emph{char}(O_K/m_K)>2$ and $|\emph{Max}((I,\beta),(I_K,\beta_K))| \nequiv 1 \ \emph{mod} \ 2$ when $\emph{char}(O_K/m_K)=2$. 
\begin{proof}
The third condition of definition \ref{incompatibility} becomes trivially satisfied, and the first two conditions are precisely translated in the statement. 
\end{proof}
\end{corollary}
\section{The shooting game} \label{shooting game}
The goal of this section is to explain the rules of a certain Markov process, which we called the \emph{shooting game}, and some of its variants. This process is the bridge between the two sides of the equality in Theorem \ref{counting}. This will be explained in detail in the next two sections. We shall begin with an informal description. 

Let $\rho$ be a shift, and let $r$ be a positive integer. Let $p$ be a prime, $f$ a positive integer and let $q:=p^f$. We will use the following notation: given $m \in \mathbb{Z}_{\geq 1}$, denote by $v_{\rho}(m):= \text{max}(i \in \mathbb{Z}_{\geq 0}: m \in \text{im}(\rho^i))$.  Denote by $n:=v_{\rho}(r)$. Imagine there are $n+1$ shooters $S_0,S_1,\ldots,S_n$, and a rabbit $R$ placed in initial position $r$. The activity of the shooters is to shoot at the rabbit in turns. If the rabbit sits in position $x$ the shooter will always shoot from the $y \in T_{\rho}$ such that $\rho^{v_{\rho}(x)}(y)=x$. We shall call such a $y$ the \emph{shooting position} of the shot. The value $v_{\rho}(x)$ is called the \emph{length} of the shot. The rules describing how the shooters take turns and what the outcome of each turn is are the following. 

$(1)$ The shooter $S_i$ cannot perform any shot of length strictly smaller than $i$. 

$(2)$ Whenever it turns out (with the above rules) that a shot of length strictly smaller than $i$ must be performed, then $S_i$ leaves the game forever. 

$(3)$ A shooter $S_i$ can start shooting only when all the other shooters $S_j$ with $j>i$ had to leave the game by rule $(2)$. In this case he will actually shoot. 

$(4)$ The rabbit $R$ moves only when someone shoots. At each shot the rabbit moves somewhere forward on $\mathbb{Z}_{\geq 1}$. If $h$ is a positive integer, then $R$ moves exactly $h$ steps forward with probability $\frac{q-1}{q^h}$. 

$(5)$ The rabbit $R$ starts in position $r$. 

We next explain a natural way to attach to a shooting game $G$ a $\rho$-jump set $(I_G,\beta_G)$. Suppose that during the game $G$ we keep track of the shooting positions where a new shooter came in. Let's call this set $I_G$. To each element of $I_G$ we attach the length of the corresponding shot plus one, and call $\beta_G$ the resulting map from $I_G$ to $\mathbb{Z}_{\geq 1}$. Observe that, thanks to the rules, it is clear that $\beta_G$ gives also exactly one plus the number of shooters still participating in that round. Indeed this is true for $n$ by the assumption that the first round must be played with length $n$ (and so must be played necessarily by $S_n$ otherwise rule $(3)$ would be contradicted). For $b<n$, the shooter $S_b$ cannot enter the game playing a shot of length smaller than $b$ by virtue of rule $(1)$, moreover, by virtue of rule $(3)$, it must be that $S_{b+1}$ has left the game if $S_b$ is playing so the length of the shot cannot be more than $b$, otherwise $S_{b+1}$ is still allowed to play and, still by rule $(3)$, he will do so. Thus the length must be $b$. So the map $\beta_G$ is strictly decreasing.  Moreover, by rule $(4)$, the rabbit moves forward, which means that the map $i \to \rho^{\beta_G(i)}(i)$ is strictly increasing on $I_G$. In other words we have shown the following fact.
\begin{proposition}
For each game $G$, the pair $(I_G,\beta_G)$ is a $\rho$-jump set. 
\end{proposition}
Shooting games can be conveniently formalized in the language of discrete-time Markov processes. We recall the basic definition in the generality that will be relevant for us. 

A discrete-time Markov process consists of a set $S$, called the \emph{state space}, equipped with a \emph{transition function} 
$$P:S \times S \to [0,1],
$$
and with a point $x_0 \in S$, called the \emph{initial state} of the process. Moreover we require that for each  $x$ in $S$ the function $y \mapsto P(x,y)$ is a probability measure on $S$, with respect to the discrete sigma-algebra on $S$. In other words we require that $\sum_{y \in S}P(x,y)=1$. We shall refer to $P(x,y)$ as the probability to \emph{transition} from $x$ to $y$. The data $(S,P,x_0)$ with the above properties, suffice to construct a probability space that models the behavior of a discrete random walk in $S$ starting at $x_0$ and proceeding at each stage from $x$ to $y$ with probability $P(x,y)$. To do so we consider the \emph{space of paths}
$$\Omega:=S^{\mathbb{Z}_{\geq 1}},
$$
as a topological space with the product topology, where $S$ is viewed as a topological space with the discrete topology. On $\mathcal{B}(\Omega)$, the sigma-algebra of Borel sets of $\Omega$, a unique probability measure $\mu_{P,x_0}$ is defined with the following property. Take $m$ a positive integer. Let $y_1, \ldots ,y_m$ be elements of $S$. For convenience put $y_0:=x_0$. Let $Y$ be the \emph{cylinder} set $Y:=\{y_0\} \times \ldots \times \{y_m\} \times S^{\mathbb{Z}_{\geq m+1}}$. We have that
$$\mu_{P,x_0}(Y)=\prod_{i=0}^{m-1}P(y_i,y_{i+1}).
$$
Moreover we ask that $\mu_{P,x_0}(\{x_0\} \times S^{\mathbb{Z}_{\geq 2}})=1$. The existence of such a measure is a simple consequence of the Kolmogorov extension theorem \cite[Theorem 2.4.3]{Tao}. 

For the shooting game the triple $(S,P,x_0)$ is as follows. We take as state space
$$S:=\{(x_1,x_2) \in \mathbb{Z}_{\geq 1} \times \mathbb{Z}_{\geq 0}: v_{\rho}(x_1)=x_2 \}.
$$
In the informal description being in state $(x_1,x_2) \in S$, means that the rabbit $R$ is in position $x_1$ and that the next shot will be of length $x_2=v_{\rho}(x_1)$. 
The initial point is 
$$x_0:=(r,n).
$$
The transition function is defined as follows. Let $x:=(x_1,x_2)$ and $y:=(y_1,y_2)$ be in $S$ with $y_1>x_1$. Then we put 
$$P(x,y):=\frac{q-1}{q^{y_1-x_1}}.
$$
For all other choices of $x,y \in S$ we put $P(x,y)=0$. We shall denote by $(\mathcal{S}(\rho,r,q),\mu_{q,r})$ the pair $(\Omega,\mu_{P,x_0})$ defined in the above paragraph. This is the space of shooting games. Sometimes we shall also use the notation $\mathcal{S}(\rho,r)$ to denote merely the topological space $\Omega=S^{\mathbb{Z}_{\geq 1}}$. Observe that 
$$\mu_{q,r}(\{(\omega_1,\omega_2) \in \mathcal{S}(\rho,r): \text{$\omega_1$ is strictly increasing} \})=1.
$$
The informal description, at the beginning of this section, gives us a map
$$\mathcal{S}(\rho,r) \to \text{Jump}_{\rho},
$$
which can be described as follows. Let $(\omega_1, \omega_2)$ be in  $\mathcal{S}(\rho,r)$ with $\omega_1$ strictly increasing. Define $$
I_{(\omega_1,\omega_2)}:=\{i \in \mathbb{Z}_{\geq 1}:\ \text{for all positive integers $j$ smaller than $i$ we have} \ \omega_2(i)<\omega_2(j) \}.
$$
We put $\beta_{(\omega_1,\omega_2)}$ to be the restriction of $\omega_2+1$ to $I_{(\omega_1,\omega_2)}$. One readily sees that if $G$ is the shooting game corresponding to $(\omega_1,\omega_2)$, then the jump set $(I_{(\omega_1,\omega_2)},\beta_{(\omega_1,\omega_2)})$ coincides with $(I_G,\beta_G)$.  In the subspace, having measure $0$, of $(\omega_1,\omega_2)$ such that $\omega_1$ is not an increasing map, we let $I_{(\omega_1,\omega_2)}=\emptyset$.
Extended jump sets arise from a natural modification of the shooting game, called the extended shooting game. From now on we assume that $T_{\rho}$ is finite. Moreover from now on we shall restrict the variable $r$ to be smaller than $e_{\rho}^{*}=\text{max}(T_{\rho})+1$. The key difference with a shooting game is that in an extended shooting game the shooters can shoot from $T_{\rho}^{*}$ and not only from $T_{\rho}$. We shall directly introduce the extended shooting game in terms of Markov processes.

For the extended shooting game we consider the following triple $(S^{*},P^{*},x_0)$. For any two positive integers $k_1,k_2$ define $v_{\rho}(e_{\rho}^{*},k_1,k_2):=\#\{m \in \mathbb{Z}_{\geq 0}: k_1 < \rho^m(e_{\rho}^{*}) \leq k_2\}$.  We take as state space the set $S^{*}$ of points $(x_1,x_2) \in \mathbb{Z}_{\geq 1} \times \mathbb{Z}_{\geq 0}$ such that one of the following two holds. Either we have  $v_{\rho}(x_1)=x_2$: in this case $(x_1,x_2)$ is said to be of the \emph{first kind}. Or we have that $\rho^{x_2}(e_{\rho}^{*})=x_1$: in this case $(x_1,x_2)$ is said to be of the \emph{second kind}. 
The initial point is
$$x_0:=(r,v_{\rho}(r)).
$$
The definition of the transition function is slightly more involved. However, right after the definition, we will give an intuitive perspective on such functions. Let $x:=(x_1,x_2)$ and $y:=(y_1,y_2)$ be in $S$ with $y_1>x_1$. If $y$ is of the first kind we put
$$P^{*}(x,y):=\frac{q-1}{q^{y_1-x_1}p^{v_{\rho}(e_{\rho}^{*},x_1,y_1)}}.
$$
If $y$ is of the second kind we put
$$P^{*}(x,y):=\frac{p-1}{q^{y_1-x_1-1}p^{v_{\rho}(e_{\rho}^{*},x_1,y_1)}}.
$$
In all the other cases we put $P^{*}(x,y)=0$. A straightforward calculation shows that this function $P^{*}$ satisfies the equations of a transition functions. We shall instead explain this in a different way, which offers a more intuitive perspective on the formula for $P^{*}$. This can be done by considering an auxiliary family of Markov processes: informally these can be imagined as the Markov processes modeling the behavior of a repeated coin toss that is stopped at the first win. Fix $x_1 \in \mathbb{Z}_{\geq 0}$. Let the state space be $S_{x_1}:=\{0,1\} \times \mathbb{Z}_{\geq x_1}$. Let the initial point be $x_0:=(1,x_1)$. Let $y$ be an integer larger than $x_1$. We put $P_{x_1}((1,y),(1,y)):=1$, and $P_{x_1}((1,y),z)=0$ for all the other values of $z$ in $S_{x_1}$. Moreover we put $P_{x_1}((0,y),(0,y+1))=\frac{1}{q}$ and $P_{x_1}((0,y),(1,y+1))=\frac{q-1}{q}$. Finally we put $P_{x_1}((1,x_1),(0,x_1+1))=\frac{1}{q}$ and $P_{x_1}((1,x_1),(1,x_1+1))=\frac{q-1}{q}$. For all the other values of $z_1,z_2$ in $S_{x_1}$ we put $P_{x_1}(z_1,z_2)=0$. In this manner one obtains a Markov process where with probability $1$ a path is eventually constant, with second coordinate strictly greater than $x_1$. In this manner a probability measure on $\mathbb{Z}_{>x_1}$ is induced, with respect to the discrete sigma-algebra. This measure is precisely the one used in the shooting games: it gives to each $x \in \mathbb{Z}_{>x_1}$ weight equal to $\frac{q-1}{q^{x-x_1}}$. We can imagine a path in $S_{x_1}$ as given by the following scenario. A walker is equipped with a coin $C$ giving $1$ with probability $\frac{q-1}{q}$ and $0$ with probability $\frac{1}{q}$. He starts his walk at $x_1$ and moves at $x_1+1$ to see if he will stop there forever. He tosses $C$ and if the result is $1$ he will stop there forever, otherwise he has to move at $x_1+2$ and repeat the operation. On the other hand, to obtain the formula for $P^{*}$ we are in the following scenario. Our walker has also a second special coin $C^{*}$, this coin takes $1$ with probability $\frac{p-1}{p}$ and $0$ with probability $\frac{1}{p}$. The rule is that he can use $C^{*}$ only when he arrives at a position $x$ such that there is a nonnegative integer $m$ with $\rho^m(e_{\rho}^{*})=x$. In this case before using $C$ he uses $C^{*}$. In case $C^{*}$ gives $1$ he will remain in $x$ forever and he is also provided with a cash prize. If $C^{*}$ gives $0$ he will use $C$ that will still follow the rules as before, telling him if he will stay forever at $x$ (though without cash prize) or if he has to move to $x+1$ to try again his luck. In this manner we obtain a natural probability measure on $\{0,1\} \times \mathbb{Z}_{>x_1}$, which we denote by $\mathbb{P}_{x_1}^{*}$, where the first coordinate is $1$ precisely when the walker has obtained also a cash prize. One has that if $(y_1,y_2)$ is of the first type, then
$$\mathbb{P}_{x_1}^{*}((0,y_1))=P^{*}((x_1,x_2),(y_1,y_2)),
$$
and if $(y_1,y_2)$ is of the second type, then
$$\mathbb{P}_{x_1}^{*}((1,y_1))=P^{*}((x_1,x_2),(y_1,y_2)).
$$

The triple $(S^{*},P^{*},r)$ gives rise to the probability space of extended shooting games. This is the pair $(\mathcal{S}^{*}(\rho,r,q),\mu_{q,r}^{*}):=(\Omega^{*},\mu_{P^{*},r})$, where $\Omega^{*}=(S^{*})^{\mathbb{Z}_{\geq 1}}$ is the space of paths and $\mu_{q,r}^{*}$ is the natural probability measure on it, as explained above. Sometimes we shall use the notation $\mathcal{S}^{*}(\rho,r)$  to denote merely the topological space $\Omega^{*}=(S^{*})^{\mathbb{Z}_{\geq 1}}$.

Imitating what we have done in the case of shooting games, we obtain a map 
$$\mathcal{S}^{*}(\rho,r) \to \text{Jump}_{\rho}^{*}.
$$ 
 Equip $\text{Jump}_{\rho}^{*}$ with the discrete sigma algebra. Pushing forward $\mu_{q,r}^{*}$ via the map $\mathcal{S}^{*}(\rho,r,q) \to \text{Jump}_{\rho}^{*}$, we obtain a probability measure on $\text{Jump}_{\rho}^{*}$, which we will denote also by $\mu_{q,r}^{*}$. Let $(I,\beta)$ be in\ $\text{Jump}_{\rho}^{*}$. Observe that for $r=e_{\rho}^{'}$ the measure $\mu_{q,r}^{*}$ gives positive probability to $(I,\beta)$ if and only if $(I,\beta)$ is admissible (for a definition see immediately before Theorem \ref{classification of quasi free}).
 
 We devote the next subsection to describe a number of subspaces and quotients of $\mathcal{S}^{*}(\rho,r)$ that will play an important role in the proof of Theorem \ref{counting}. 
\subsection{Subspaces and quotients of extended shooting games} \label{subspaces and quotient shooting}
For any positive integer $j$ we define $\mathcal{S}_{\geq j}^{*}(\rho,r)$ to be the set of $G \in \mathcal{S}^{*}(\rho,r)$ such that $I_G \neq \emptyset$ and $\text{min}(\beta_G) \geq j$. We denote by $\mathcal{S}_{\geq j}^{*}(\rho,r,q)$ the above set viewed as a measure space with the restriction of $\mu_{q,r}^{*}$. If we normalize the measure in the unique way to get a probability space, we will denote the resulting probability space as $\mathcal{S}_{\geq j}^{*}(\rho,r|q)$: this is the probability space of games $G$ \emph{conditioned} to never invoke a shooter of lower index then $S_j$. 

Next we define $\mathcal{S}_{=j}^{*}(\rho,r)$ to be the set of $G \in \mathcal{S}^{*}(\rho,r)$ such that $I_G \neq \emptyset$ and $\text{min}(\beta_G) = j$ and $e_{\rho}^{*} \not \in I_G$. We denote by $\mathcal{S}_{= j}^{*}(\rho,r,q)$ the above set viewed as a measure space with the restriction of $\mu_{q,r}^{*}$. If we normalize the measure in the unique way to get a probability space, we will denote the resulting probability space as $\mathcal{S}_{=j}^{*}(\rho,r|q)$: this is the probability space of games $G$ \emph{conditioned} to invoke as a last shooter $S_j$, with his first shot being not from $e_{\rho}^{*}$. 

We define $\mathcal{S}_{=j,*}^{*}(\rho,r)$ to be the set of $G \in \mathcal{S}^{*}(\rho,r)$ such that $\text{min}(\beta_G) = j$ and $e_{\rho}^{*} \in I_G$. We denote by $\mathcal{S}_{=j,*}^{*}(\rho,r,q)$ the above set viewed as a measure space with the restriction of $\mu_{q,r}^{*}$. If we normalize the measure in the unique way to get a probability space, we will denote the resulting probability space as $\mathcal{S}_{=j,*}^{*}(\rho,r|q)$: this is the probability space of games $G$ \emph{conditioned} to invoke as a last shooter $S_j$, and by letting him shoot for the first time from $e_{\rho}^{*}$. 

Finally for any positive integers $x$ we define $\mathcal{S}_{x,\text{stop}}(\rho,r,q)$ as the quotient probability space of $\mathcal{S}^{*}(\rho,r,q)$, where two games are identified precisely when the trajectories of the rabbit are identical as long as the rabbit stays below $x$.

For all the following remarks $m$ will denote $v_{\rho}(e_{\rho}^{*})$. We will assume that $\rho$ is such that $m \geq 2$. Equivalently we are assuming that in $\mathcal{S}^{*}(\rho,e_{\rho}^{'},q)$ the subspace $\mathcal{S}^{*}_{=1}(\rho,e_{\rho}^{'},q)$ has probability strictly smaller than $1$: in case the subspace $\mathcal{S}^{*}_{=1}(\rho,e_{\rho}^{'},q)$ has probability equal to $1$, then the shooting game gives constantly the jump set $(\{e_{\rho}^{'}\},e_{\rho}^{'} \mapsto 1)$.

For a positive integer $x$ and a nonnegative integer $m'$, such that $x \in \text{Im}(\rho^{m'})$ we denote by $\rho^{-m'}(x)$ the unique positive integer $y$ such that $\rho^{m'}(y)=x$.
\begin{remark} \label{dec. remark 1}
Let $j \in \{1,\ldots ,m\}$. Given $G$ an element of $\mathcal{S}^{*}(\rho,\rho^{-(j-1)}(e_{\rho}^{'}))$, by applying $\rho^{j-1}$ to it, we obtain an element of $\mathcal{S}^{*}_{\geq j}(\rho,e_{\rho}^{'})$, this map is a bijection. Moreover the map $\rho^{j-1}$ induces an isomorphism of measure spaces $\mathcal{S}^{*}_{\geq j}(\rho,e_{\rho}^{'}|q) \simeq _{\text{meas. space}} \mathcal{S}^{*}(\rho,\rho^{-(j-1)}(e_{\rho}^{'}),q)$. The map induced on jump sets consists simply of shifting $\beta$ by $j-1$. We call $\psi_j$ the inverse of the isomorphism given by $\rho^{j-1}$.
\end{remark}
\begin{remark} \label{dec. remark 2}
The map described in Remark \ref{dec. remark 1} induces an isomorphism of measure spaces $\mathcal{S}^{*}_{=j}(\rho,e_{\rho}^{'},q) \simeq _{\text{meas. space}} \mathcal{S}^{*}_{=1}(\rho,\rho^{-(j-1)}(e_{\rho}^{'}),q)$.
\end{remark}
\begin{remark} \label{dec. remark 3}
For all elements $G$ of a given equivalence class $C$ in $\mathcal{S}_{x,\text{stop}}(\rho,r,q)$ the set $\{i \in I_{G}:\rho^{\beta_G(i)}(i) \leq x \}$ will be the same, and similarly the restriction of $\beta_{G}$ to this set will be the same. The resulting pair $(I_C,\beta_C)$ is also an extended $\rho$-jump set. In particular the set $\mathcal{S}^{*}_{=1}(\rho,e_{\rho}^{'})$  consists of a union of equivalence classes for the projection to $\mathcal{S}_{e_{\rho}^{*},\text{stop}}(\rho,e_{\rho}^{'})$. Moreover in each such equivalence class $C$, the jump set $(I_C,\beta_C)$ coincides with the jump set $(I_G,\beta_G)$ for any $G$ belonging to $C$. 
\end{remark}
\begin{remark} \label{dec. remark 4}
Let $j$ be in $\{1,\ldots ,m-1\}$. Observe that the projection of $\mathcal{S}^{*}_{=j,*}(\rho,e_{\rho}^{'}|q)$ to $\mathcal{S}_{\rho^{j}(e_{\rho}^{*}),\text{stop}}(\rho,e_{\rho}^{'},q)$ lands in the image of $\mathcal{S}^{*}_{\geq j+1}(\rho,e_{\rho}^{'},q)$. Thus we can apply to it $\psi_{j+1}$, landing in $\mathcal{S}^{*}_{e_{\rho}^{'},\text{stop}}(\rho,\rho^{-j}(e_{\rho}^{'}),q)$. We denote by $\psi_{j}^{*}:\mathcal{S}^{*}_{=j,*}(\rho,e_{\rho}^{'}|q) \to \mathcal{S}_{\rho^{-j+1}(e_{\rho}^{'}),\text{stop}}(\rho,\rho^{-j}(e_{\rho}^{'}),q)$ the resulting map.  If $C=\psi_{j}^{*}(G)$ then $I_G= I_C \cup \{e^{*}\}$ with ${\beta_G}_{|I_C}=\beta_C+j-1$ and $\beta_G(e^{*})=j$. This provides a reconstruction of $(I_G,\beta_G)$ from $(I_{\psi(G)},\beta_{\psi(G)})$.  
\end{remark}
\begin{remark} \label{dec. remark 5}
Given $j \in \{1,\ldots ,n\}$, it can be easily shown that one has always that 
$$(p-1)\mu_{q,e_{\rho}^{'}}^{*}(\mathcal{S}^{*}_{\geq j+1}(\rho,e_{\rho}^{'},q))=\mu_{q,e_{\rho}^{'}}^{*}(\mathcal{S}^{*}_{=j,*}(\rho,e_{\rho}^{'},q)).
$$
In the setting of local fields this fact is mirrored by Proposition \ref{when star is in  for tot ram}. 

\end{remark}
\section{Shooting game and filtered orbits}\label{shot-game and orbit}
Fix $p$ a prime number. Using the notation of Section \ref{filtered modules}, we take $R=\mathbb{Z}_p$, and we fix $f$ a positive integer and $\rho$ a shift. Let $q:=p^f$, and we recall that the module $M_{\rho}^{f-1} \oplus M_{\rho}^{*}$ was defined on top of subsection \ref{transitive}. On the other hand recall from Theorem \ref{bijection orbits jumps} that the set of extended jump sets is in bijection with the set of $\text{Aut}_{\text{filt}}(M_{\rho}^{f-1} \oplus M_{\rho}^{*})$-orbits of vectors in $\pi_R(M_{\rho}^{f-1} \oplus M_{\rho}^{*})$. Thus, by using the Haar measure, this induces naturally a probability measure on the set of extended admissible $\rho$-jump sets. We call this measure $\mu_{q,\text{Haar}}$: given $(I,\beta)$ an admissible extended $\rho$-jump set, we have that $\mu_{q,\text{Haar}}(I,\beta):=\mu_{\text{Haar}}(\text{filt-ord}^{-1}(I,\beta))$, where the Haar measure is normalized giving total mass $1$ to the set of orbits corresponding to admissible jump sets (for a definition of admissible see right before Theorem \ref{classification of quasi free}).

On the other hand Section \ref{shooting game} provides us with another measure on admissible extended jump sets, namely the probability that a shooting game in $\mathcal{S}(\rho,e_{\rho}^{'},q)$ gives the jump set $(I,\beta)$. We denoted by $\mu_{q,e_{\rho}^{'}}^{*}$ this probability measure on $\text{Jump}_{\rho}^{*}$.

\begin{proposition} \label{filtered orbits and shootings}
For any extended admissible $\rho$-jump set $(I,\beta)$ one has that
$$\mu_{q,\emph{Haar}}(I,\beta)=\mu_{q,e_{\rho}^{'}}^{*}(I,\beta).
$$
\begin{proof}
We prove slightly more. Namely we construct a map $G(-)$ sending an admissible vector $v$ into a shooting game $G(v) \in \mathcal{S}^{*}(\rho,e_{\rho}^{'},q)$, in such a way that $G^{*}(\mu_{\text{Haar}})=\mu_{q,e_{\rho}^{'}}^{*}$ and that $(I_v,\beta_v)=(I_{G(v)},\beta_{G(v)})$.

To construct such a map fix a filtered basis $\mathcal{B}$ for $M_{\rho}^{f-1} \oplus M_{\rho}^{*}$, in the sense of Definition \ref{definition of a basis}. This provides us for each $i \in T_{\rho}$ with elements $b_{i,1},\ldots ,b_{i,f}$ of $M_{\rho}^{f-1} \oplus M_{\rho}^{*}$ with weights obeying $w(b_{i,j})=i$, and for $e_{\rho}^{*}$ we are provided with an element $b_{e_{\rho}^{*}}$ with $w(b_{e_{\rho}^{*}})=e_{\rho}^{*}$ and $b_{e_{\rho}^{*}} \not \in \pi_R \cdot (M_{\rho}^{f-1} \oplus M_{\rho}^{*})$. Next fix $\mathcal{A}:=\{\alpha_1,\ldots ,\alpha_{|R/m_R|-1}\}$ a set of representatives in $R$ of $(R/m_R)^{*}$. For every $i \in T_{\rho}$, denote by $\mathcal{F}_i:=\mathcal{A}\cdot \mathcal{B}_i$ the set of $ab$ with $a \in \mathcal{A}$ and $b \in \mathcal{B}$. Denote by $i^{*}$ the unique element of $T_{\rho}$ such that there exists a positive integer $m$ with $e_{\rho}^{*}=\rho^m(i^{*})$. Furthermore denote by $\mathcal{F}_{e_{\rho}^{*}}:=\pi_R^{v_{\rho}(e_{\rho}^{*})}\mathcal{F}_{i^{*}}+\mathcal{A}\{b_{e_{\rho}^{*}}\}$. The sets $\mathcal{F}_i$ as $i$ runs in $T_{\rho}^{*}$ are pairwise disjoint. For any vector $z \in M_{\rho}^{f-1} \oplus M_{\rho}^{*}$, there exists a unique $i \in T_{\rho}^{*}$, which we denote by $i_z$, such that there exist $b_z \in \mathcal{F}_{i_z}$ and $v(z) \in \mathbb{Z}_{\geq 0}$  with $w(-\pi_R^{v(z)}b_{z}+z)>w(z)$, where $\rho^{v(z)}(i_z)=w(z)$ and $b_{z} \in \mathcal{F}_{i_z}$. The elements $b_z,v(z)$ are unique.

Now let $v \in \pi_R(M_{\rho}^{f-1} \oplus M_{\rho}^{*})$ be a vector in an admissible orbit. Let $x$ be the vector in $M_{\rho}^{f-1} \oplus M_{\rho}^{*}$ such that $\pi_Rx=v$. We inductively construct a sequence of vectors by letting $x_1=x$ and setting $x_{j+1}=x_j-\pi_R^{v(x_j)}b_{x_j}$ the unique expression explained above. We use this sequence of vectors to attach to $v$ a shooting game $G(v)$ as follows: we consider the map $(f_1,f_2): \mathbb{Z}_{\geq 1} \to \mathbb{Z}_{\geq 1} \times \mathbb{Z}_{\geq 1}$, given by the relation $f_1(i)=w(x_i)$ and $f_2(i)=v(x_i)$. One can easily verify that the pushforward with $G(-)$ of the Haar measure is $\mu_{q,e_{\rho}^{'}}^{*}$ and that the map $G(-)$ preserves jump sets. 
\end{proof}
\end{proposition}
\section{A mass-formula for $U_1$}\label{mass formula} Let $p$ be a prime number, $f$ be a positive integer, denote by $q=p^f$, let $e \in (p-1)\mathbb{Z}_{\geq 1}$ and let $(I,\beta)$ be an extended admissible $\rho_{e,p}$-jump set. The goal of this section is to provide a proof of Theorem \ref{counting}. In virtue of Proposition \ref{filtered orbits and shootings}, this task is equivalent to proving the following Theorem.
\begin{theorem} \label{counting rephrased}
$$
\mu_{\frac{e}{p-1},\mathbb{Q}_{p^f}(\zeta_p)}(\{K \in S(\frac{e}{p-1},\mathbb{Q}_{p^f}(\zeta_p)):(I_K,\beta_K)=(I,\beta)\})=\mu_{q,e_{\rho}^{'}}^{*}(I,\beta).
$$
\end{theorem}
If $F$ is a local field and $h$ is a positive integer, then $\text{Eis}(h,F)$ denotes the set of degree $h$ Eisenstein polynomials in $F[x]$. These are monic polynomials $f(x)$ with coefficients in $O_F$, that reduced modulo $m_F$, the maximal ideal of $O_F$, become $x^h$ and such that $f(0) \not \in m_F^2$. 

\subsection{Proof outline}
Since our proof of Theorem \ref{counting rephrased} is quite long, we shall first explain its basic idea. In subsection \ref{idea} we give an overview of the main ideas of the proof. In subsection \ref{strategy} we explain how the proof reduces to the construction of certain maps from certain spaces of Eisenstein polynomials to shooting games. Finally we spend the rest of the section to construct such maps and to show that they meet all the requirements explained in subsection \ref{strategy}.
\subsubsection{The idea of the proof} \label{idea}
In this subsection the discussion is \emph{informal}. Our priority here is to provide some intuition about how the proof of Theorem \ref{counting rephrased} goes. For a formal proof see from subsection \ref{strategy} on.

The starting idea is to proceed as in the proof of Proposition \ref{filtered orbits and shootings}. One has immediately a difference between the set-up of Proposition \ref{filtered orbits and shootings} and the one of Theorem \ref{counting rephrased}. In Proposition \ref{filtered orbits and shootings} one has a \emph{fixed} free-filtered module where it is possible to successively ``shoot at elements of $M_{\rho}^{f-1} \oplus M_{\rho}^{*}$" as done in that proof, using a fixed filtered basis. In this manner a measure-preserving map is obtained sending each vector of $M_{\rho}^{f-1} \oplus M_{\rho}^{*}$ to an extended shooting game. By measure-preserving here we mean that the push-forward of the measure on the source is equal to the measure on the target. In Theorem \ref{counting rephrased} we have a \emph{varying} quasi-free filtered module, namely $U_{\bullet}(K)$, so we need firstly to find a common manner to successively ``shoot at the units" in order to proceed as in the proof of Proposition \ref{filtered orbits and shootings}. This step can be done by fixing the set of polynomials $\mathcal{B}:=\{(1+\gamma x^i): i \in T_{\rho_{e,p}}, \gamma \in \text{Teich}(\mathbb{Q}_{p^f})-\{0\} \} \cup \{1+\epsilon' x^{\frac{pe}{p-1}} \}$, where $\text{Teich}(\mathbb{Q}_{p^f})$ denotes the set of Teichm\"uller representatives of $\mathbb{F}_{p^f}$ in $\mathbb{Q}_{p^f}$. Here $\epsilon'$ is a Teichm\"uller representative of a fixed element $ \epsilon' \in \mathbb{F}_{p^f}$ with $\text{Tr}_{\mathbb{F}_{p^f}/\mathbb{F}_p}(\epsilon) \neq 0$. To keep a stricter analogy with the proof of Theorem \ref{filtered orbits and shootings}, we should only allow a set of Teich\"muller representatives that, once it is reduced modulo $p\mathbb{Z}_{p^f}$, it becomes a basis of $\mathbb{F}_{p^f}$. Since this restriction would make the description of the next steps heavier and is irrelevant for the present discussion, we shall disregard it. One can attempt to proceed precisely as in the proof of Theorem \ref{filtered orbits and shootings} in order to construct a measure-preserving function from the set of Eisenstein polynomials to the set of shooting games. As we will see, we will use only a part of this idea, one that is still good enough to obtain a proof of Theorem \ref{counting rephrased} and that combined with a different set of observations (explained at the end of this subsection) leads to more informative results. More concretely one starts with an Eisenstein polynomial $g(x)=x^{\frac{e}{p-1}}+\sum_{i=0}^{\frac{e}{p-1}-1}a_ix^i$. Next one finds a unit $u$ in $\mathbb{Z}_{p^f}[\zeta_p]^{*}$ in such a way that $ug(x)=ux^{\frac{e}{p-1}}+\sum_{i=1}^{\frac{e}{p-1}-1}ua_ix^i+1-\zeta_p$. Hence in the field $\mathbb{Q}_{p^f}(\zeta_p)[x]/g(x)$ one can write $\zeta_p=1+\sum_{i=1}^{\frac{e}{p-1}-1}ua_ix^i+ux^{\frac{e}{p-1}}=:g_1(x)$. At this point one multiplies $g_1(x)$ by $(1+\gamma x^{e_0})^{p^{\text{v}_{\mathbb{Q}_p}(e)}}$ for a suitable $\gamma \in \text{Teich}(\mathbb{Q}_{p^f})$, where $e_0$ denotes the largest divisor of $\frac{e}{p-1}$ coprime to $p$. \emph{After} expanding the product, we replace all the powers of $x$ having degree larger than $\frac{e}{p-1}$ with their remainder upon division by $g(x)$. In this way a second expression $g_2(x)=1+\sum_{i=1}^{\frac{pe}{p-1}-1}a_i(2)x^i$ is obtained. Now we would like to iterate this. We do so as long as this unit has weight less than $\frac{pe}{p-1}$. In this case we have precisely one way to choose an element of $\mathcal{B}$ that does the same job $1+\gamma x^{e_0}$ did for $g_1(x)$: in particular we do not use the element $1+\epsilon' x^{\frac{pe}{p-1}}$. If we iterate this procedure as long as the weight stays below $\frac{pe}{p-1}$, we obtain a sequence of polynomials $g_1(x), \ldots ,g_k(x)$ where $g_{s+1}(x)$ is obtained by ``shooting" $g_s(x)$ with an element of $\mathcal{B}$ in the way hinted above. Moreover it is relatively easy to determine that the change of weight from $g_s(x)$ to $g_{s+1}(x)$ obeys the same rule as the change of positions of the rabbit during the shooting game. Indeed, as we shall see in the proof, although the expressions for $g_{s+1}(x)$ can become increasingly complicated, there is a simple way to get the \emph{probability} that the weight of $g_{s+1}(x)$ will be larger than a given $y$, with $y<\frac{pe}{p-1}$. The reason for this is that we can divide in two pieces the expressions that decide whether the weight of $g_{s+1}(x)$ will be larger than $y$. One piece comes from ``lower order terms" and it behaves in the proof, from the probabilistic point of view, as a \emph{constant}. The other piece comes in a very simple manner from the Eisenstein polynomial $g(x)$ and one sees, directly from the definition of Haar measure on Eisenstein polynomials, that it is a uniform random variable in $\text{Teich}(\mathbb{Q}_{p^f})$.  In this way we can prove Theorem \ref{counting rephrased} for all $(I,\beta)$ with $\text{min}(\beta)=1$ and $\frac{pe}{p-1} \not \in I$. To proceed further we need to deal with the case that, in the above ``shooting process", the unit has reached a weight at least $\frac{pe}{p-1}$ and we have not yet used a shot of length $0$. That means that either $\frac{pe}{p-1} \in I_K$ with $\beta(\frac{pe}{p-1})=1$ or $\zeta_{p^2} \in K$. The last remark in Section \ref{subspaces and quotient shooting} tells us that the former possibility should occur precisely $p-1$ times as often as the latter. On the other hand Proposition \ref{when star is in  for tot ram} tells us that the same happens for local fields. Indeed the fields $\{\mathbb{Q}_{p^f}(\zeta_p)(s): s \in \{1, \ldots , p\}$ have all the same mass, therefore by Proposition \ref{when star is in  for tot ram} we conclude that they partition the set of local fields $K$, having either $\frac{pe}{p-1} \in K$ or $\zeta_{p^2} \in K$, into $p$ disjoint sets $X_1, \ldots ,X_p$ having all the same mass, with $K \in X_s$ if and only if $\mathbb{Q}_{p^f}(\zeta_p)(s) \subseteq K$. For $s \in \{1, \ldots , p-1 \}$ we have that the total mass of $X_s$ equals $\frac{1}{p-1}$ of the total mass of the fields $K$ with $\frac{pe}{p-1} \in I_K$ and $\beta_K(\frac{pe}{p-1})=1$. On the other hand $X_p$ consists of those fields $K$ with $\zeta_{p^2} \in K$. So we deal with the set $X_1, \ldots ,X_{p-1}$ working with Eisenstein polynomials over $\mathbb{Q}_{p^f}(\zeta_p)(1), \ldots ,\mathbb{Q}_{p^f}(\zeta_p)(p-1)$ and we deal with $X_p$ using $\mathbb{Q}_{p^f}(\zeta_{p^2})$. Thanks to Proposition \ref{when star is in  for tot ram}, by repeating the above ``shooting argument" for the sets $X_1, \ldots , X_{p-1}$, Theorem \ref{counting rephrased} is proved also in the case that $\frac{pe}{p-1} \in I$ with $\beta(\frac{pe}{p-1})=1$. The idea is to repeat this whole proof structure over $\mathbb{Q}_{p^f}(\zeta_{p^2})$. 

\subsubsection{Proof strategy} \label{strategy}
The plan of the proof is the following. Let $n:=\text{v}_{\mathbb{Q}_p}(e)$. We will use the notation from Section \ref{shooting game} and in particular from Section \ref{subspaces and quotient shooting}. For each $j \in \{0,1,\ldots ,n\}$ we construct maps
$$\sigma_j:\text{Eis}(\frac{e}{p^j(p-1)}, \mathbb{Q}_q(\zeta_{p^{j+1}})) \to \mathcal{S}_{\frac{pe}{p-1}, \text{stop}}(\rho_{e,p},\frac{e}{p^j(p-1)},q)
$$
and for each $j_1 \in \{0,\ldots ,n-1\}$ and $j_2 \in \{1,..,p-1\}$ we construct maps
$$\sigma_{j_1,j_2}:\text{Eis}(\frac{e}{p^{j_1+1}(p-1)}, \mathbb{Q}_q(\zeta_{p^{j_1+1}})(j_2)) \to \mathcal{S}_{\frac{e}{p-1}, \text{stop}}(\rho_{e,p},\frac{e}{p^{j_1+1}(p-1)},q),
$$
having the following two properties. 

$(P.1)$ For any $j \in \{0,1,\ldots ,n\}$ and $f(x) \in \text{Eis}(\frac{e}{p^j(p-1)}, \mathbb{Q}_q(\zeta_{p^{j+1}}))$, denoting by $K_{f(x)}:=\mathbb{Q}_q(\zeta_{p^{j+1}})[x]/f(x)$, we have that
$$\{i \in I_{K_{f(x)}}: \rho_{e,p}^{\beta_{K_{f(x)}}-(j+1)}(i)<\frac{pe}{p-1} \}=I_{\sigma_j(f(x))}
$$
and for each $i \in I_{\sigma_j(f(x))}$ we have that
$$\beta_{K_{f(x)}}(i)=\beta_{\sigma_j(f(x))}(i)+j.
$$

For any $j_1 \in \{0,\ldots ,n-1\}$, $j_2 \in \{1,\ldots ,p-1\}$ and $f(x) \in \text{Eis}(\frac{e}{p^{j+1}(p-1)}, \mathbb{Q}_q(\zeta_{p^{j_1+1}})(j_2))$ , we have that
$$I_{K_{f(x)}}=I_{\sigma_{j_1,j_2}(f(x))} \cup \{\frac{pe}{p-1}\}.
$$
For $i \in I_{\sigma_{j_1,j_2}(f(x))}$ we have
$$\beta_{K_{f(x)}}(i)=\beta_{\sigma_{j_1,j_2}(f(x))}(i)+j_1+1.
$$
Finally we have
$$\beta_{K_{f(x)}}(\frac{pe}{p-1})=j_1+1.
$$ 

$(P.2)$ For any $j \in \{0,1,\ldots ,n\}$  pushing forward $\mu_{\text{Haar}}$, the natural probability measure on $\text{Eis}(\frac{e}{p^j(p-1)},\mathbb{Q}_q(\zeta_{p^{j+1}}))$ coming from the Haar measure on the coefficients, with $\sigma_j$ one obtains $\mu_{q,\frac{e}{p^j(p-1)},\rho_{e,p}}^{*}$, the probability measure on $\mathcal{S}_{\frac{pe}{p-1}, \text{stop}}(\rho_{e,p},\frac{e}{p^j(p-1)},q)$ introduced in Section \ref{shooting game}.

For any $j_1 \in \{0,1,\ldots ,n-1\}$ and $j_2 \in \{1,\ldots ,p-1\}$, pushing forward $\mu_{\text{Haar}}$ with $\sigma_{j_1,j_2}$ from $\text{Eis}(\mathbb{Q}_{q}(\zeta_{p^{j_1+1}})(j_2))$ to $\mathcal{S}_{\frac{e}{p^{j_1}(p-1)}}(\rho_{e,p},\frac{e}{p^{j_1+1}(p-1)},q)$, one obtains $\mu_{q,\frac{e}{p^{j_1+1}(p-1)},\rho_{e,p}}^{*}$. 

The construction of such maps $\sigma_{j}$ and $\sigma_{j_1,j_2}$ satisfying $(P.1)$ and $(P.2)$ as above, is sufficient to prove Theorem \ref{counting rephrased} and thus Theorem \ref{counting}. Indeed, thanks to Remark \ref{dec. remark 3}, we can conclude with $\sigma_0$ that Theorem \ref{counting rephrased} holds for all $(I,\beta)$ with $\text{min}(\beta)=1$ and $\frac{pe}{p-1} \not \in I$. At that point we know that the probability of the event $\{\text{min}(\beta)>1 \ \text{or} \ \frac{pe}{p-1} \in I \}$ has equal probability on both sides (Eisenstein polynomials and shooting games). We remark that this conclusion can be reached alternatively also by a direct computation. By Remark \ref{dec. remark 5} we know that, at the level of shooting games, the probability of the event $\{\frac{pe}{p-1} \in I, \beta(\frac{pe}{p-1})=1 \}$ is $p-1$ times as large as the event $\{\text{min}(\beta)>1 \}$. On the other hand this is clearly true also at the level of Eisenstein polynomials: the fields $\mathbb{Q}_{q}(\zeta_p)(j)$ have the same mass as $j$ runs through $\{1,\ldots ,p\}$, and by Proposition \ref{when star is in  for tot ram}, precisely the first $p-1$ of them give the event $\{\frac{pe}{p-1} \in I, \beta(\frac{pe}{p-1})=1 \}$, while the last (which is $\mathbb{Q}_{q}(\zeta_{p^2})$) gives the event $\{\text{min}(\beta)>1 \}$. Thus we can go on in the proof of Theorem \ref{counting rephrased} by conditioning on both sides with either the event $\{\frac{pe}{p-1} \in I, \beta(\frac{pe}{p-1})=1\}$ or the event $\{ \text{min}(\beta)>1 \}$. Thus by Remark \ref{dec. remark 4}, and with the $\sigma_{0,j_2}$ we conclude the validity of Theorem \ref{counting rephrased} for $(I,\beta)$ with $\frac{pe}{p-1} \in I$ and $\beta(\frac{pe}{p-1})=1$. Here we are using that if $F/K$ is a totally ramified Galois extension of local fields, then, for an extension $F/E$ and a positive integer $d \in [F:E]\mathbb{Z}_{\geq 1}$, the conditional probability measure $\mu_{d,E}(-|\{F \text{\ is a subfield}\})$ equals the probability measure $\mu_{\frac{d}{[F:E]},F}$\footnote{Here we are using the following standard notation. If $(X,\mu)$ is a probability space and $A \subseteq X$ is a measurable subset with $\mu(A)>0$, then $\mu(-|A)$ denotes the probability measure on $A$, defined by the formula $\mu(-|A)(B):=\frac{\mu(B)}{\mu(A)}$ for each $B \subseteq A$ measurable.}. That justifies the passage to Eisenstein polynomials over the extensions $\mathbb{Q}_{q}(\zeta_{p}(j))$.

Now we continue working over $\mathbb{Q}_{q}(\zeta_{p^2})$ and we proceed precisely as above. Namely we first use the map $\sigma_{1}$ to show that Theorem \ref{counting rephrased} holds for $(I,\beta)$ with $\text{min}(\beta)=2$ and $\frac{pe}{p-1} \not \in I$. If $n=1$ we are done. Otherwise we again obtain that the measure of the event $\{\text{min}(\beta)>2 \ \text{or} \ \frac{pe}{p-1} \in I\}$ coincides on both sides of Theorem \ref{counting rephrased}. Finally Remark \ref{dec. remark 5} gives that, at the level of shooting games, the event $\{\frac{pe}{p-1} \in I, \beta(\frac{pe}{p-1})=2 \}$ is $p-1$ times as frequent as the event $\{\text{min}(\beta)>2\}$. This holds also for Eisenstein polynomials thanks to the fact that the extensions $\mathbb{Q}_{q}(\zeta_{p^2})(j)$ of $\mathbb{Q}_{q}(\zeta_{p^2})$ for $j \in \{1,\ldots ,p\}$ have all the same mass, and by Proposition \ref{when star is in  for tot ram} we have that the first $p-1$ give the event $\{\frac{pe}{p-1} \in I, \beta(\frac{pe}{p-1})=2 \}$ while the last (which is $\mathbb{Q}_{q}(\zeta_{p^3})$) gives the event $\{\text{min}(\beta)>2\}$. Thus we use the maps $\sigma_{1,j}$ to prove Theorem \ref{counting rephrased}, with the same considerations made above, and we go on working over $\mathbb{Q}_{q}(\zeta_{p^3})$. Iterating this argument we prove Theorem \ref{counting rephrased} for every $(I,\beta)$, an extended admissible jump set. Therefore to finish the proof, we are left with constructing the maps $\sigma_{j}, \sigma_{j_1,j_2}$ and showing that they have properties $(P.1)$ and $(P.2)$. This done in the next two subsections.
\subsection{Construction of the maps $\sigma_j, \sigma_{j_1,j_2}$} \label{Construction of the maps}
Let $j \in \{0,\ldots ,n\}$, we begin with the construction of $\sigma_j$. To lighten the notation, denote $e_j:=\frac{e}{p^j(p-1)}$. An element 
$$f(x):=x^{e_j}+\sum_{i=0}^{e_j-1}a_ix^i
$$
in $\text{Eis}(e_j,\mathbb{Q}_q(\zeta_{p^{j+1}}))$ can be equivalently represented 
as
$$
\tilde{f}(x):=1+\sum_{i=1}^{e_j}{\tilde{a}}_ix^i
$$
where $\tilde{f}(x):=\frac{1-\zeta_{p^{j+1}}}{a_0}f(x)+\zeta_{p^{j+1}}$. This gives us an embedding of $\text{Eis}(e_j, \mathbb{Q}_q(\zeta_{p^{j+1}}))$ into $H_{e_j}(\mathbb{Q}_q(\zeta_{p^{j+1}})):=\{g \in \mathbb{Z}_q[\zeta_{p^{j+1}}]:\text{deg}(g) \leq e_j, g(0)=1, g(x) \equiv 1 \ \text{or} \ g(x) \equiv 1+ax^{e_j} \ \text{mod} \ (1-\zeta_{p^{j+1}}) \ \text{for some $a \in \mathbb{Z}_q[\zeta_{p^{j+1}}]^{*}$} \}$. Starting with $f_0(x):=\tilde{f}(x)$, we define inductively a sequence $\{f_n(x)\}_{n \in \mathbb{Z}_{\geq 0}}$ with $f_n(x) \in H_{e_j}(\mathbb{Q}_q(\zeta_{p^{j+1}}))$ for every $n \in \mathbb{Z}_{\geq 0}$. To do so we first define a weight map on $H_{e_j}(\mathbb{Q}_q(\zeta_{p^{j+1}}))$ by
$$w(1+\sum_{i=1}^{e_j}b_ix^i)=\text{min}_{1 \leq i \leq e_j: b_i \neq 0}(e_j\text{v}_{\mathbb{Q}_q(\zeta_{p^{j+1}})}(b_i)+i).
$$
Now, suppose that $w(f_n(x)) \geq \frac{pe}{p-1}$, then declare 
$$f_{n+1}(x)=f_n(x).
$$
So, suppose that 
$w(f_n(x))<\frac{pe}{p-1}$. Observe that $$f_n(x) \in U_{w(f_n(x))}(\mathbb{Q}_q(\zeta_{p^{j+1}})[x]/f(x))-U_{w(f_n(x))+1}(\mathbb{Q}_q(\zeta_{p^{j+1}})[x]/f(x)),$$ 
thus there exist unique $i_n \in T_{\rho_{e,p}}, \beta_n \in \mathbb{Z}_{\geq 0}$ and unique $\epsilon_n$, a Teichm\"uller representative in $\mathbb{Q}_q$, such that 
$$(1+\epsilon_nx^{i_n})^{p^{\beta_n}}f_n(x) \in U_{w(f_n(x))+1}(\mathbb{Q}_q(\zeta_{p^{j+1}})[x]/f(x)).
$$
It is not difficult to show that there exists a unique element of $H_{e_j}(\mathbb{Q}_q(\zeta_{p^{j+1}}))$ congruent to $(1+\epsilon_nx^{i_n})^{p^{\beta_n}}f_n(x)$ modulo $f(x)$ and of degree at most $e_j$. We put $f_{n+1}(x)$ to be this element. It follows by construction that $w(f_{n+1}(x)) \geq w(f_n(x))$ with equality occurring iff $w(f_{n}(x)) \geq \frac{pe}{p-1}$. Moreover in the case of equality we have $f_{n+1}(x)=f_n(x)$. Thus we define
$$\sigma_j(f(x)):=\{n \mapsto (w(f_{n-1}(x)),\beta_{n-1})\}_{n \in \mathbb{Z}_{\geq 1}} \in \mathcal{S}_{\frac{pe}{p-1},\text{stop}}(\rho_{e,p},\frac{e}{p^j(p-1)},q).
$$

Let now $j_1 \in \{0,\ldots ,n-1\}$ and $j_2 \in \{1,\ldots ,p-1\}$. The map $\sigma_{j_1,j_2}$ is defined similarly to how the maps $\sigma_j$ were defined. We briefly explain the modifications. Fix units $u_1,u_2 \in \mathbb{Q}_q(\zeta_{p^{j_1+1}})(j_2)$ with $$
u_1^{p}u_2=\zeta_{p^{j_1+1}}, \ \text{v}_{\mathbb{Q}_q(\zeta_{p^{j_1+1}})(j_2)}(u_1-1)=1, \ \text{v}_{\mathbb{Q}_q(\zeta_{p^{j_1+1}})(j_2)}(u_2-1)=\frac{pe}{p-1}
$$
and $u_2 \not \in (\mathbb{Q}_q(\zeta_{p^{j_1+1}})(j_2))^{*p}
$, as guaranteed by Proposition \ref{I,beta for twisted cyclotomic} and Corollary \ref{The relation between the units} together. Next, given $f(x):=x^{e_{j_1+1}}+\sum_{i=0}^{e_{j_1}}a_ix^i \in \text{Eis}(e_{j_1+1}, \mathbb{Q}_q(\zeta_{p^{j_1+1}}))$, define this time
$$\tilde{f}(x):=\frac{1-u_1}{a_0}f(x)+u_1.
$$
Also change $H_{e_{j_1+1}}(\mathbb{Q}_q(\zeta_{p^{j_1+1}})(j_2))$ to be the set
$$\{g \in \mathbb{Z}_q[\zeta_{p^{j_1+1}},u_1]:\text{deg}(g) \leq e_{j_1+1}, g(0)=1, g(x) \equiv 1 \ \text{or} \ g(x) \equiv 1+x^{e_j} \ \text{mod} \ (1-u_1) \},$$
and set the cut-off for concluding $f_{n+1}(x)=f_n(x)$ to be $w(f_n(x)) \geq e_j$. Following the above procedure, with these modifications, we get the construction of 
$$\sigma_{j_1,j_2}(f(x)) \in \mathcal{S}_{\frac{e}{p-1},\text{stop}}(\rho_{e,p},\frac{e}{p^{j_1+1}(p-1)},q).
$$
\subsection{The maps $\sigma_j,\sigma_{j_1,j_2}$
 satisfy properties $(P.1), (P.2)$}
Let us begin showing that, for $j \in \{0,\ldots ,n\}$, the map $\sigma_j$  obeys the property $(P.1)$. By construction, we know that for $f(x) \in \text{Eis}(e_j, \mathbb{Q}_q(\zeta_{p^{j+1}}))$, we have that
$$\zeta_{p^{j+1}} \cdot \prod_{n:w(f_n(x))<\frac{pe}{p-1}}(1+\alpha_nx^{i_n})^{p^{\beta_n}} \in U_{\frac{pe}{p-1}}(\mathbb{Q}_q(\zeta_{p^{j+1}})[x]/f(x)),
$$
with $p^{\beta_n}i_n=w(f_n(x))$, so the sequence $n \mapsto p^{\beta_n}i_n$ is strictly increasing as $n$ runs with the constraint $w(f_n(x))<\frac{pe}{p-1}$. Of course, the weight of $1+\alpha_nx^{i_n}$ in $U_{\bullet}(\mathbb{Q}_q(\zeta_{p^{j+1}})[x]/f(x))$ is precisely $i_n$. Therefore one sees that the values of $n$ such that $i_n \in I_{\mathbb{Q}_q(\zeta_{p^{j+1}})[x]/f(x)}$ are precisely those where $\beta_n$ reaches a new minimum. This is precisely the same rule that implies $i_n \in I_{\sigma_j(f(x))}$. For such an $i_n$ it easily follows from Corollary \ref{The relation between the units} that
$$\beta_n+j+1=\beta_{\mathbb{Q}_q(\zeta_{p^{j+1}})[x]/f(x)}.$$
This shows that $\sigma_j$ enjoys the property $(P.1)$ for each $j \in \{0,\ldots ,n\}$. 

We next show that for $j_1 \in \{0,\ldots ,n-1\}$ and $j_2 \in \{1,\ldots ,p-1\}$, the map $\sigma_{j_1,j_2}$ satisfies $(P.1)$. Recall the definition of the units $u_1,u_2$ introduced during the construction of the map $\sigma_{j_1,j_2}$. By construction, we know that for $f(x) \in \text{Eis}(e_{j_1+1}, \mathbb{Q}_q(\zeta_{p^{j_1+1}})(j_2))$, we have that
$$u_1 \cdot \prod_{n:w(f_n(x))<e_{j_1}}(1+\alpha_nx^{i_n})^{p^{\beta_n}} \in U_{\frac{e}{p-1}}(\mathbb{Q}_q(\zeta_{p^{j_1+1}})(j_2)[x]/f(x)).
$$
This implies that
$$u_2^{-p^{j_1}} \cdot \prod_{n:w(f_n(x))<e_{j_1}}(1+\alpha_nx^{i_n})^{p^{\beta_n}+j_1+1} \in U_{\frac{p^{j_1+1}e}{p-1}}(\mathbb{Q}_q(\zeta_{p^{j_1+1}})(j_2)[x]/f(x)).
$$
Therefore all the other units that will be employed in order to write the full relation, cannot give a contribution to $(I_{\mathbb{Q}_q(\zeta_{p^{j_1+1}})(j_2)[x]/f(x)},\beta_{\mathbb{Q}_q(\zeta_{p^{j_1+1}})(j_2)[x]/f(x)})$, due to the presence of $\frac{pe}{p-1}$ in $I_{\mathbb{Q}_q(\zeta_{p^{j_1+1}})(j_2)[x]/f(x)}$, with $\beta_{\mathbb{Q}_q(\zeta_{p^{j_1+1}})(j_2)[x]/f(x)}(\frac{pe}{p-1})=j_1$ as guaranteed by Proposition \ref{when star is in  for tot ram}. Thus one concludes using the same argument employed for $\sigma_j$. 

We next prove that the maps $\sigma_j,\sigma_{j_1,j_2}$ satisfy $(P.2)$. We will do so for $\sigma_j$, the argument for $\sigma_{j_1,j_2}$ being basically the same with different notation. 

Given $f \in \text{Eis}(e_j, \mathbb{Q}_q(\zeta_{p^{j+1}}))$, let us begin expanding each of the coefficients of
$$\tilde{f}(x):=1+\sum_{i=1}^{e_j}\tilde{a}_jx^j 
$$
as
$$\tilde{a}_i=\sum_{k=1}^{\infty}\epsilon_{k,i}(1-\zeta_{p^{j+1}})^k,
$$
for $1 \leq i<e_j$ and for $i=e_j$ we write
$$\tilde{a}_{e_j}=\sum_{k=0}^{\infty}\epsilon_{k,e_j}(1-\zeta_{p^{j+1}})^k,
$$
where all $\epsilon_{k,i}$ are Teichm\"uller representatives of $
\mathbb{Z}_q$. We can consider any finite subset of the $\epsilon_{k,i}$ as independent random variables taking values in all possible Teichm\"uller representatives with the uniform distribution if $(k,i) \neq (0,e_j)$ and uniformly in the non-zero Teichm\"uller representatives of $\mathbb{Z}_q$ for $\epsilon_{0,e_j}$.

Next, for each $n \in \mathbb{Z}_{\geq 1}$, we set
$$f_n(x)=1+\sum_{i=1}^{e_j}a_i(n)x^i,
$$
and we let
$$a_i(n):=\sum_{k=1}^{\infty}\epsilon_{k,i}(n)(1-\zeta_{p^{j+1}})^k.
$$
the corresponding Teichm\"uller expansion with respect to $(1-\zeta_{p^{j+1}})$. The fact that the second sum started with $k=1$ is a consequence of the fact that the weights of $f_n(x)$ are strictly increasing as $n$ increases together with the definition of $H_{e_j}(\mathbb{Q}_q(\zeta_{p^{j+1}}))$.

For any fixed $n \in \mathbb{Z}_{\geq 0}$ the monomials $\epsilon_{k,i}(n)(1-\zeta_{p^{j+1}})^kx^i$ can be given the weight $ke_j+i$ if $\epsilon_{k,i}(n) \neq 0$ and $\infty$ otherwise. This induces a total order on the various non-zero monomials, and the weight $w(f_n(x))$ of $f_n(x)$ as defined before, equals the minimum weight of the various monomials as long as there is a monomial with weight less than $\frac{pe}{p-1}$, otherwise we have already arrived at the point where the sequence $f_n(x)$ is eventually constant. 

From the rule to obtain $f_{m+1}(x)$ out of $f_m(x)$ we see that for any $n \in \mathbb{Z}_{\geq 1}$, and any positive integer $\frac{pe}{p-1}>w_0>w(f_{n-1}(x))$, there exists a function $F_{w_0,n}$ taking as input the sequence of $(\epsilon_{k,i})_{e_jk+i<w_0}$ and giving as output a Teichm\"uller representative of $\mathbb{Z}_q$, in such a way that if we write $w_0=e_jq'+h$, the division with remainder of $w_0$ by $e_j$, we have that
$$\epsilon_{q',h}(n)=[F_{w_0,n}((\epsilon_{k,i})_{e_jk+i<w_0})+\epsilon_{q',h}]_{\text{Teich}},
$$
where for $a \in \mathbb{Z}_q$, the symbol $[a]_{\text{Teich}}$ denotes the unique Teichm\"uller representative $\epsilon$ in $\mathbb{Z}_q$ such that $\epsilon \equiv a \ \text{mod} \ p$. It thus follows at once that for the collection of $(q',h)$ with $\frac{pe}{p-1}>e_jq'+h>w(f_{n-1}(x))$, the variables $\epsilon_{q',h}(n)$ are independent random variables taking values in the Teichm\"uller representatives of $\mathbb{Z}_q$ with the uniform distribution. Therefore the change of weights from $w(f_{n-1}(x))$ to $w(f_n(x))$ is governed precisely by the rules of the shooting game. This ends the proof. 
\section{Finding jump sets inside an Eisenstein polynomial} \label{finding jump sets inside}
The primary goal of this Section is to establish Theorem \ref{valuation coefficients more general}, which is a generalization of Theorem \ref{valuation coefficients} from the Introduction. We will next specialize Theorem \ref{valuation coefficients more general} to obtain several consequences that aim to give a sense to the reader on how efficiently one can establish the value of $(I,\beta)$ in the range of the Theorem. Most notably we will see that for $q$ odd or for $j \geq 1$, the set of \emph{strongly} Eisenstein polynomials (see Definition \ref{strongly Eisenstein}) over $\mathbb{Q}_q(\zeta_{p^{j+1}})$  is precisely the set of polynomials giving the jump set that has the highest probability. Also we will see the relation between Theorem \ref{valuation coefficients more general} and Theorem \ref{counting rephrased}. Indeed we shall prove Theorem \ref{valuation coefficients more general}, by establishing the equality between the jump set of a shooting game coming from the valuation of the coefficients of an Eisenstein polynomial (denoted as $\tilde{\sigma}_j$ below) and (part of the) jump set of the shooting game constructed using the maps introduced during the proof of Theorem \ref{counting rephrased} (denoted as $\sigma_j$). Also we observe that Theorem \ref{counting rephrased} partially \emph{follows} as a direct counting from Theorem \ref{valuation coefficients more general}, namely it does so for the jump sets coming from the region of Eisenstein polynomials where Theorem \ref{valuation coefficients} applies (which for instance for $p=2$ and $j=0$ (i.e. over $\mathbb{Q}_2$) is empty, and for general $p$ it misses an open set of Eisenstein polynomials). Finally we shall give examples, showing that without the main assumption on the different, the conclusion of Theorem \ref{valuation coefficients more general} is not anymore valid in general. 

Let $j \in \mathbb{Z}_{\geq 0}$, $p$ a prime number, $f \in \mathbb{Z}_{\geq 1}$ and $q:=p^f$. Let $e \in p^{j}(p-1)\mathbb{Z}_{\geq 1}$. Recall the notation $e_j:=\frac{e}{p^j(p-1)}$, used during the proof of Theorem \ref{counting rephrased}. Let $g(x) \in \text{Eis}(e_{j}, \mathbb{Q}_q(\zeta_{p^{j+1}}))$ (see notation from the proof of Theorem \ref{counting rephrased}). We proceed to define a stopping shooting game attached to $g$, which will be denoted as
$$ {\tilde{\sigma}}_{j}(g(x)) \in \mathcal{S}_{e,\text{stop}}(\rho_{e,p},e_j,q).
$$
It is defined with the following simple rule. Write $g(x)=x^{e_j}+\sum_{i=0}^{e_j-1}a_ix^i$ and give to each monomial $a_ix^i$ weight $w(a_ix^i):=e_j\text{v}_{\mathbb{Q}_q(\zeta_{p^{j+1}})}(a_i)+i$.  Arrange $w(a_ix^i)$ as an increasing sequence $n \mapsto w(a_{i_n}x^{i_n})$. Then the sequence 
$$\tilde{\sigma_j}(g(x)):=\{(w(a_{i_n}x^{i_n}),v_p(i_n))\}_{n:w(a_{i_n}x^{i_n}) \leq e}$$
is an element of
$$\mathcal{S}_{e,\text{stop}}(\rho_{e,p},e_j,q).
$$
One can now see that the pair $(I_{g(x)},\beta_{g(x)})$ defined in the Introduction right before Theorem \ref{valuation coefficients} is simply the jump set of the shooting game $\tilde{\sigma}_j(g(x))$. We now explain more closely how one calculates this pair. It is clear that the smallest weight is precisely $e_j=w(x^{e_j})$, consistently with the fact that in $\tilde{\sigma}_j(g(x))$ the rabbit is supposed to start from $e_j$. So we start with $\alpha_0=e_j$. Next, given $\alpha_h$ (thus the rabbit being at $\text{v}_{\mathbb{Q}_q(\zeta_{p^{j+1}})}(a_{\alpha_h})e_j+\alpha_h$), to obtain a larger weight, either we find other weights that are contained in the interval $[\text{v}_{\mathbb{Q}_q(\zeta_{p^{j+1}})}(a_{\alpha_h})e_j,(\text{v}_{\mathbb{Q}_q(\zeta_{p^{j+1}})}(a_{\alpha_h})+1)e_j ]$ or there are no such other weights. In the first case the contribution comes only from the weights $\alpha$ that are larger than $\alpha_h$ (otherwise the weight is smaller). Among these, in order to have a change of shooters and thus a contribution to the jump set, we are only interested in those requiring a smaller shot-length, i.e. with smaller $\text{v}_{\mathbb{Q}_p}(\alpha_h)$, in good harmony with rule $(3)$ of the shooting game. Thus the first such weight with smaller $\text{v}_{\mathbb{Q}_p}$ is precisely where the shooter is changed. In the second case, the weight will be larger anyway, thus (as long as larger weight matters) we are now interested in examining all $\alpha$ with $\text{v}_{\mathbb{Q}_q(\zeta_{p^{j+1}})}(a_{\alpha})>\text{v}_{\mathbb{Q}_q(\zeta_{p^{j+1}})}(a_{\alpha_h})$. Again, among these we are only interested in those where $\text{v}_{\mathbb{Q}_p}(\alpha)$ becomes smaller. The smallest such weight, again, is the first place where the shot length became smaller and a new shooter came in giving the next contribution to the jump set of the shooting game. We formalize this explanation in the following procedure. 

\emph{Procedure.} \label{Procedure Eis}Let $g(x):=x^{e_j}+\sum_{i=0}^{e_j-1}a_ix^i \in \text{Eis}(e_j, \mathbb{Q}_q(\zeta_{p^{j+1}}))$. Set $\alpha_0=e_j$. Given $\alpha_h$, construct $\alpha_{h+1}$ as follows. Search if there is $\alpha$ such that $\text{v}_{\mathbb{Q}_q(\zeta_{p^{j+1}})}(a_{\alpha})=\text{v}_{\mathbb{Q}_q(\zeta_{p^{j+1}})}(a_{\alpha_h})$ and $\alpha \geq \alpha_h$. If such an $\alpha$ exists, search if there is among them one with $\text{v}_{\mathbb{Q}_p}(\alpha)<\text{v}_{\mathbb{Q}_p}(\alpha_h)$. If there is such $\alpha$, pick the smallest such $\alpha$ and declare $\alpha_{h+1}=\alpha$. If no such $\alpha$ exists, then look if there is an $\alpha$ such that $\text{v}_{\mathbb{Q}_q(\zeta_{p^{j+1}})}(a_{\alpha})>\text{v}_{\mathbb{Q}_q(\zeta_{p^{j+1}})}(a_{\alpha_{h}})$ and $e_j\text{v}_{\mathbb{Q}_q(\zeta_{p^{j+1}})}(a_{\alpha})<e$. If no such $\alpha$ exists then set $\alpha_{h+1}=\alpha_h$. Otherwise let $d$ be the next valuation that attains the above constraints. Look if there is $\alpha$ with $\text{v}_{\mathbb{Q}_p}(\alpha)<\text{v}_{\mathbb{Q}_p}(\alpha_h)$ and $\text{v}_{\mathbb{Q}_q(\zeta_{p^{j+1}})}(a_{\alpha})=d$, in that case take the smallest such $\alpha$ as $\alpha_{h+1}$. If there is none, go to the next valuation with the above constraints and do the same search, iterating until you either have to set $\alpha_{h+1}=\alpha_h$, or you have found an $\alpha_{h+1} \neq \alpha_h$. In this way the sequence $\{\alpha_i\}$ is produced. With this notation, writing $I_{\tilde{\sigma}_j(g(x))}=\{i_1<\ldots <i_s\}$, we have that $\beta_{\tilde{\sigma}_j}(i_k)=v_p(\alpha_k)$ and $p^{\beta_{\tilde{\sigma}_j}(i_k)}i_k=e_j\text{v}_{\mathbb{Q}_q(\zeta_{p^{j+1}})}(a_{\alpha_k})+\alpha_k$. 

Let $j$ be a positive integer. Recall that for an integer $e$ in $p^j(p-1)\mathbb{Z}_{\geq 1}$ we define $e_j:=\frac{e}{p^j(p-1)}$.
\begin{theorem}\label{valuation coefficients more general}
Let $j$ be a positive integer and let $e \in p^j(p-1)\mathbb{Z}_{\geq 1}$. For any $g(x) \in \emph{Eis}(e_j, \mathbb{Q}_q(\zeta_{p^{j+1}}))$, we have that the set 
$$i_1<\ldots <i_s,
$$
described in the above procedure, is equal to the set
$$ \{i \in I_{\mathbb{Q}_q(\zeta_{p^{j+1}})[x]/g(x)}:p^{\beta_{\mathbb{Q}_q(\zeta_{p^{j+1}})[x]/g(x)}(i)-j-1}i<e
\}
$$
and for each $k \in \{1,..,s\}$ we have that
$$\beta_{\mathbb{Q}_q(\zeta_{p^{j+1}})[x]/g(x)}(i_k)=\emph{v}_{\mathbb{Q}_p}(\alpha_k)+j+1.
$$
\begin{proof}
We proceed by looking more closely at the construction of the maps $\sigma_j$ in the proof of Theorem \ref{counting rephrased}. One first crucial ingredient is that, thanks to the shape of the conclusion, we can disregard, in the setting of the proof of Theorem \ref{counting rephrased}, monomials with weights larger than $e$, so that we can perform $p$-th powering as if we were in a characteristic $p$ field. Keeping this in mind one sees from the construction of the sequence $g_n(x)$ in the proof of Theorem \ref{counting rephrased}, that given an $\alpha_k$ as above, then as long as $n$ satisfies $w(g_n(x))<e_j\text{v}_{\mathbb{Q}_q(\zeta_{p^{j+1}})}(a_{\alpha_k})+\alpha_k$, then for each positive integer $w_0$ with
$$w(g_n(x))<w_0 \leq e_j\text{v}_{\mathbb{Q}_q(\zeta_{p^{j+1}})}(a_{\alpha_k})+\alpha_k,
$$
and
$$\text{v}_{\mathbb{Q}_p}(w_0) \leq \text{v}_{\mathbb{Q}_p}(\alpha_k),
$$
one sees that
 $$F_{w_0,n}((\epsilon_{k,i})_{e_jk+i<e_ja_{\alpha_k}})=0.$$
This is seen by induction on $n$ and direct inspection. The key observation is that, once we can disregard the multiples of $p$ in $p$-powering, when we perform a shot, as in the proof of Theorem \ref{counting rephrased}, it sends all the monomials having weight \emph{smaller} than $e_j\text{v}_{\mathbb{Q}_q(\zeta_{p^{j+1}})}(a_{\alpha_k})+\alpha_k$ only to monomials having an index with \emph{larger} $p$-adic valuation.
With this, the formula appearing at the end of the proof of Theorem \ref{counting rephrased} gives
$$\epsilon_{w_0,h}(n)=[F_{w_0,n}((\epsilon_{k,i})_{e_jk+i<w_0})+\epsilon_{q',h}]_{\text{Teich}}=[\epsilon_{q',h}]_{\text{Teich}},
$$
where $w_0=e_jq'+h$. In terms of the shooting game $\sigma_j(g(x))$, this means precisely that the rabbit will visit the position $e_j\text{v}_{\mathbb{Q}_q(\zeta_{p^{j+1}})}(a_{\alpha_k})+\alpha_k$, and that all the shots used before that event are of length strictly larger than $\text{v}_{\mathbb{Q}_p}(\alpha_k)$. Indeed the rabbit doesn't visit any of the positions $w_0<e_j\text{v}_{\mathbb{Q}_q(\zeta_{p^{j+1}})}(a_{\alpha_k})+\alpha_k$ with $\text{v}_{\mathbb{Q}_p}(w_0) \leq \text{v}_{\mathbb{Q}_p}(\alpha_k)$. But these are precisely the positions where a stop of the rabbit would have given a shot of length at most $\text{v}_{\mathbb{Q}_p}(\alpha_k)$ before the position $e_j\text{v}_{\mathbb{Q}_q(\zeta_{p^{j+1}})}(a_{\alpha_k})+\alpha_k$ was reached.

\end{proof}
\end{theorem}
Observe that from the \emph{Procedure} it is clear that the set of Eisenstein polynomials $g(x)$ such that the full jump sets of the field $\mathbb{Q}_q(\zeta_{p^{j+1}})[x]/g(x)$ can be reconstructed from Theorem \ref{valuation coefficients more general}, consists precisely of those polynomials $g(x)$ having a coefficient $a_i$, with $(i,p)=1$, such that $\text{v}_{\mathbb{Q}_q(\zeta_{p^{j+1}})}(a_i)<\text{v}_{\mathbb{Q}_q(\zeta_{p^{j+1}})}(p)=p^{j}(p-1)$. This condition is precisely equivalent to the condition on the different
$$\text{v}_{\mathbb{Q}_q(\zeta_{p^{j+1}})}(\delta(\mathbb{Q}_q(\zeta_{p^{j+1}})[x]/g)/\mathbb{Q}_q(\zeta_{p^{j+1}}))<p^j(p-1).
$$
For $j=0$, this shows Theorem \ref{valuation coefficients} from the Introduction. The only case where this is an empty set of Eisenstein polynomials is if $p=2$, $j=0$ and $2|e$: one would get a non extremal coefficient of an Eisenstein polynomial being a unit, which is impossible by definition. For all other values of $p$ and $j$ one obtains, with Theorem \ref{valuation coefficients more general}, a positive proportion of the Eisenstein polynomials where the jump set can be read off completely from the valuation of the coefficients of the polynomial, also in a fairly easy way. For $p$ or $j$ getting large the volume of this region gets quickly pretty large. In particular, if $(p,j) \neq (2,0)$, we next see that one can identify the set of Eisenstein polynomials giving the \emph{most likely} jump set.
\begin{definition}\label{strongly Eisenstein}
If $K$ is a local field, $d \geq 2$ an integer, and $g(x):=x^d+\sum_{i=0}^{d-1}a_ix^i \in \text{Eis}(d,K)$, we say that $g(x)$ is strongly Eisenstein if $\text{v}_K(a_1)=1$. 
\end{definition}
The following is a very special case of Theorem \ref{valuation coefficients more general}. Recall that if $e \in p^j(p-1)\mathbb{Z}_{\geq 1}$ we have the notation $e_j:=\frac{e}{p^j(p-1)}$.
\begin{theorem}\label{strongly Eisenstein characterized}
Let $p,j$, such that $(p,j) \neq (2,0)$. Let $e \in p^{j+1}(p-1)\mathbb{Z}_{\geq 1}$, $f$ a positive integer and set $q:=p^f$. Then $g(x) \in \emph{Eis}(e_j, \mathbb{Q}_q(\zeta_{p^{j+1}}))$ is strongly Eisenstein if and only if
$$I_{\mathbb{Q}_q(\zeta_{p^{j+1}})[x]/g(x)}=\{\frac{e}{p^{\emph{v}_{\mathbb{Q}_p}(e)}(p-1)},e_j+1\}
$$
with $$\beta_{\mathbb{Q}_q(\zeta_{p^{j+1}})[x]/g(x)}(\frac{e}{p^{\emph{v}_{\mathbb{Q}_p}(e)}(p-1)})=\emph{v}_{\mathbb{Q}_p}(e)+1,  \ \beta_{\mathbb{Q}_q(\zeta_{p^{j+1}})[x]/g(x)}(e_j+1)=j+1.
$$
\end{theorem}
Observe that this gives, explicitly, the counting that the above jump set, $\{\frac{e}{p^{\text{v}_{\mathbb{Q}_p}(e)}(p-1)},e_j+1\}$ with $\frac{e}{p^{\text{v}_{\mathbb{Q}_p}(e)}(p-1)} \mapsto \text{v}_{\mathbb{Q}_p}(e)+1$ and $e_j+1 \mapsto j+1$, occurs with probability $\frac{q-1}{q}$ among all totally ramified degree $e_j$-extensions of $\mathbb{Q}_q(\zeta_{p^{j+1}})$: this is the jump set occurring with highest probability. We know that this jump set occurs with probability $\frac{q-1}{q}$ also from Theorem \ref{counting rephrased}. So in particular this fact is true also for $(2,0)$. To see that explicitly for $e=2$, observe that among the $6$ totally ramified quadratic extension of $\mathbb{Q}_2$, the only ones not giving the above jump set are $\mathbb{Q}_2(\zeta_4)$ and $\mathbb{Q}_2(\zeta_4)(1)=\mathbb{Q}_2(\sqrt{3})$; they have same mass (as we saw in general) and it equals $\frac{1}{4}$, hence the remaining mass equals $\frac{1}{2}$. But we can immediately see that in this case the same conclusion of Theorem \ref{strongly Eisenstein characterized} does not hold. Consider for instance $x^2+2x+2 \in \text{Eis}(2, \mathbb{Q}_2)$: it is a strongly Eisenstein polynomial. But $\mathbb{Q}_2[x]/g(x)$ is isomorphic to the extension $\mathbb{Q}_2(\zeta_{4})$, whose jump set is merely $\{1\}$, with $1 \mapsto 2$, in contrast to the conclusion of Theorem \ref{strongly Eisenstein characterized}. Thus in Theorem \ref{strongly Eisenstein characterized} the requirement $(p,j) \neq (2,0)$ cannot be dropped, and so in particular the assumption in Theorem \ref{valuation coefficients} of being strongly separable cannot be avoided.

\section{Filtered inclusions of principal units}
\label{generalization of I beta}
In this section we explain how to attach to any strongly separable extension of local fields, $L/K$, a $\rho_{\infty,p}$-jump set $(I_{L/K},\beta_{L/K})$, which is an invariant of the filtered inclusion 
$$U_{\bullet}(K) \subseteq U_{\bullet}(L).
$$
Moreover for $K=\mathbb{Q}_{q}(\zeta_p)$, we will have that
$$(I_{L/K},\beta_{L/K})=(I_L,\beta_L).
$$
As we shall see, the fact that the extension is strongly separable will force $(I_{L/K},\beta_{L/K})$ to be a $\rho_{e_L,p}$-jump set as well for $e_L=v_L(p)$. 

We will begin to attach to any $u \in U_1(K)-U_2(K)$ a $\rho_{e_L,p}$-jump set $(I_{L/K}(u),\beta_{L/K}(u))$. We will immediately see that it is also a $\rho_{\infty,p}$-jump set, thanks to strong separability. Finally we will see the big effect of assuming strong separability: the jump set $(I_{L/K}(u),\beta_{L/K}(u))$ is independent on the choice of $u \in U_1(K)-U_2(K)$ and can be computed, by means of an immediate generalization of Theorem \ref{valuation coefficients}, from an Eisenstein polynomial giving the extension $L/\tilde{K}$, where $\tilde{K}$ is the largest unramified extension of $K$ in $L$. 

Let $u \in U_1(K)-U_2(K)$. Recall from Section \ref{reading jump inside} that we can attach to $u$ the function $g_{u,U_{\bullet}(L)}$. We have the following. The proof is along the same lines seen in Proposition \ref{reconstruction process} and is therefore omitted.
\begin{proposition} \label{boh}
There exists a unique jump set $(I_{L/K}(u),\beta_{L/K}(u))$ such that $g_{u,U_{\bullet}(L)}$ breaks at the elements of $\emph{Im}(\beta_{L/K}(u))-1$. Moreover if $i \in I_{L/K}(u)$, then 
$$g_{u,U_{\bullet}(L)}(i+1)=\rho_{e_L,p}^{\beta_{L/K}(u)(i)-1}(i).
$$
\end{proposition}
In the torsion-free case, the jump set $(I_{L/K}(u),\beta_{L/K}(u))$ has a more familiar interpretation. In what follows the function $\text{filt-ord}$ (as introduced in Proposition \ref{bijection orbits jumps}) will always be with respect to the filtered module (denoted as) $U_{\bullet}(L)$.
\begin{proposition}
Let $u,L,K$ as above and suppose moreover that $\mu_p(L)=\{1\}$. Then
$$\emph{filt-ord}(u^p)=(I_{L/K}(u),\beta_{L/K}(u)).
$$
Moreover for $u_1,u_2 \in U_1(K)-U_2(K)$ we have that
$$(I_{L/K}(u_1),\beta_{L/K}(u_1))=(I_{L/K}(u_2),\beta_{L/K}(u_2)),
$$
if and only if $u_1,u_2$ are in the same orbit under $\emph{Aut}_{\emph{filt}}(U_{\bullet}(L))$. 
\begin{proof}
This is a simple consequence of Theorem \ref{no torsion} and Proposition \ref{bijection orbits jumps} combined.
\end{proof}
\end{proposition}
We now show that the jump set of Proposition \ref{boh} is independent of the choice of $u \in U_1(K)-U_2(K)$ for all strongly separable extensions $L/K$. Recall the way we attached to any strongly separable Eisenstein polynomial $g(x)$ a jump set $(I_{g(x)},\beta_{g(x)})$ right after Theorem \ref{counting} in the Introduction. 
\begin{theorem} \label{independence}
Let $L/K$ be any strongly separable extension of local fields. Let $u_1,u_2 \in U_1(K)-U_2(K)$. Then
$$(I_{L/K}(u_1),\beta_{L/K}(u_1))=(I_{L/K}(u_2),\beta_{L/K}(u_2)).
$$
Denote by $(I_{L/K},\beta_{L/K}):=(I_{L/K}(u),\beta_{L/K}(u))$ for any $u \in U_1(K)-U_2(K)$. 
Denote by $\tilde{K}$ the maximal unramified extension of $K$ in $L$, and let $g(x)$ be any Eisenstein polynomial in $\tilde{K}[x]$ giving the extension $L/\tilde{K}$. We have that
$$(I_{L/K},\beta_{L/K})=(I_{g(x)},\beta_{g(x)}).
$$
\begin{proof}
This can be shown by precisely the same argument used in the proof of Theorem \ref{valuation coefficients more general}. 
\end{proof}
\end{theorem}
In particular we find the following corollary.
\begin{corollary} \label{all in same orbit}
Let $L/K$ be a strongly separable extension of local fields, with $\mu_p(L)=\{1\}$. Then $U_1(K)-U_2(K)$ is contained in one orbit under $\emph{Aut}_{\emph{filt}}(U_{\bullet}(L))$. Call this orbit $\mathcal{O}_{L/K}$. The set $\mathcal{O}_{L/K}$ can be also characterized as follows
$$\mathcal{O}_{L/K}=\{u \in U_{\bullet}(L): u^p \in \emph{filt-ord}^{-1}((I_{L/K},\beta_{L/K}))\}.
$$ 
\end{corollary}
In positive characteristic the statement further simplifies. 
\begin{corollary}
Let $L/K$ be a separable extension of local fields with $\text{char}(K)=p$. Then $U_1(K)-U_2(K)$ is contained in one orbit under $\emph{Aut}_{\emph{filt}}(U_{\bullet}(L))$. Call this orbit $\mathcal{O}_{L/K}$. The set $\mathcal{O}_{L/K}$ can be also characterized as follows
$$\mathcal{O}_{L/K}=\{u \in U_{\bullet}(L): u^p \in \emph{filt-ord}^{-1}((I_{L/K},\beta_{L/K})) \}.
$$ 
\end{corollary}
\section{Jump sets under field extensions} \label{Jump sets under field extensions}
Let $K_1/\mathbb{Q}_p(\zeta_p)$ be a finite extension. Fix a positive integer $d$. Consider the following natural question. \\
\\
\textbf{Question}: Which extended admissible $\rho_{de_{K_1},p}$-jump sets are realizable as $(I_{K_2},\beta_{K_2})$ for some totally ramified extension $K_2/K_1$ of degree $d$?\\

In case $(d,p)=1$ the answer is very easy. 
\begin{proposition} \label{The tame case}
Let $K_2/K_1$ be totally ramified degree $d$ extension, with $(d,p)=1$. Then
$$I_{K_2}=dI_{K_1}
$$
with 
$$\beta_{K_2}(di)=\beta_{K_1}(i),
$$
for each $i \in I_{K_1}$.
\begin{proof}
First notice that, since $(d,p)=1$, we have $dT_{\rho_{e,p}}^{*} \subseteq T_{\rho_{de,p}}^{*}$. Moreover we notice that the assignment $(I_{K_2},\beta_{K_2})$ given in the statement is clearly an extended $\rho_{de_{K_1},p}$-jump set. Next we write
$$\prod_{i \in I_{K_1}}u_i^{p^{\beta_{K_1}(i)-1}}=\zeta_p,
$$
with $u_i \in U_i(K_1)-U_{i+1}(K_1)$ for each $i \in I_{K_1}$, and $\frac{pe_{K_1}}{p-1} \in I_{K_1}$ implies $u_{\frac{pe_{K_1}}{p-1}} \not \in K_1^{*p}$. We thus conclude with Corollary \ref{The relation between the units} by noticing that $u_i \in U_{di}(K_2)-U_{di+1}(K_2)$ for each $i \in I_{K_1}$, and that if $\frac{pe_{K_1}}{p-1} \in I_{K_1}$ then we must have that $u_{\frac{pe_{K_1}}{p-1}} \not \in K_2^{*p}$. Indeed taking a $p$-th root of $u_{\frac{pe_{K_1}}{p-1}}$ gives an unramified degree $p$ extension of $K_1$ which would contradict both that $(d,p)=1$ and that $K_2/K_1$ is totally ramified.
\end{proof} 
\end{proposition}
The previous proof teaches us also what is the difficulty when $(d,p) \neq 1$ in answering Question. In this case the relation
$$\prod_{i \in I_{K_1}}u_i^{p^{\beta_{K_1}(i)-1}}=\zeta_p,
$$
cannot be directly used to calculate $(I_{K_2},\beta_{K_2})$, because $\text{v}_{K_2}(u_i-1) \not \in T_{\rho_{de_{K_1},p}}$ for each $i< \frac{pe_{K_1}}{p-1}$. Nevertheless, a more careful inspection shows that this relation can sometimes be used to extrapolate properties of $(I_{K_2},\beta_{K_2})$. This is the content of the next theorem, which, together with Theorem \ref{last guy}, contains as a very special case Proposition \ref{The tame case}. 
\begin{theorem} \label{constraining j.s. under ext.}
Let $d$ be a positive integer and $K_2/K_1$ a degree $d$ totally ramified extension. Let $i \in I_{K_1}$ with $i \neq \frac{pe_K}{p-1}$. Suppose that if the set $J:=\{j \in I_{K_1}: j<i\}$ is not empty, then 
$$ \beta_{K_1}(\emph{max}(J))-\beta_{K_1}(i)>\emph{v}_{\mathbb{Q}_p}(d).
$$
Then $$\frac{d}{p^{\emph{v}_{\mathbb{Q}_p}(d)}}i \in I_{K_2}$$ 
with
$$\beta_{K_2}(\frac{d}{p^{\emph{v}_{\mathbb{Q}_p}(d)}}i)=\beta_{K_1}(i)+\emph{v}_{\mathbb{Q}_p}(d).
$$
\begin{proof}
Take $i \neq \frac{pe_{K_1}}{p-1}$ as in the assumptions of this theorem. Write $$\prod_{i' \in I_{K_1}}u_{i'}^{p^{\beta_{K_1}(i')-1}}=\zeta_p,
$$
with $u_{i'} \in U_{i'}(K_1)-U_{i'+1}(K_1)$ for each $i' \in I_{K_1}$, and $\frac{pe_{K_1}}{p-1} \in I_{K_1}$ implies $u_{\frac{pe_{K_1}}{p-1}} \not \in K_1^{*p}$. Next, for each $i' \in I_{K_1}$, write
$$\prod_{j \in A(i')}u_{i',j}^{p^{\beta(i',j)}},
$$
with $A(i') \subseteq T_{\rho_{K_2}}^{*}$, $\text{v}_{K_2}(u_{i',j}-1)=j$ for each $j \in A(i')$ and $\frac{d}{p^{\text{v}_{\mathbb{Q}_p}(d)}}i' \in A(i')$ with $\beta(i',j)=\text{v}_{\mathbb{Q}_p}(d)$ and $\text{v}_{K_2}(u_{i',i'}^{p^{\beta(i',i')}}-1)<\text{v}_{K_2}(u_{i',j}^{p^{\beta(i',j)}}-1)$ for each $j \in A(i')-\{i'\}$. We now proceed to expand the above expression for $\zeta_p$. Attach to each term $u_{i',j}$ the pair $(\text{v}_{K_2}(u_{i',j}-1),\beta_{K_1}(i')+\beta(i,j))$. We see that the point attached to $u_{i,i}$, which is $(\frac{d}{p^{\text{v}_{\mathbb{Q}_p}(d)}}i,\beta_{K_1}(i)+\text{v}_{\mathbb{Q}_p}(d))$, is strictly smaller, with respect to $\leq_{\rho_{K_2}}$, than all the other points (and hence occurs precisely once). Indeed, using that $(I_{K_1},\beta_{K_1})$ is a jump set, we see that it must be smaller than any term coming from some $u_{i'}$ with $i'>i$. On the other hand for each $i'<i$, we use the fact that $\beta_{K_2}(i')>\beta_{K_2}(i)+\text{v}_{\mathbb{Q}_p}(d)$ to conclude that the point attached to $u_{i,i}$ must be smaller than any term attached to $u_{i',j}$ with $i'<i$. This is enough to conclude with Corollary \ref{The relation between the units}.

\end{proof}
\end{theorem}
The case of $e_{K_1}^{*}$ requires no special assumptions and can be treated more easily in a different way. 
\begin{theorem} \label{last guy}
Let $d$ be a positive integer and $K_2/K_1$ a degree $d$ totally ramified extension. Suppose $\frac{pe_K}{p-1} \in I_{K_1}$. Then $di \in I_{K_2}$ and $\beta_{K_2}(di)=\beta_{K_1}(i)$.
\begin{proof}
This follows immediately from Proposition \ref{when star is in  for tot ram} and Proposition \ref{when the star is in I}.
\end{proof}
\end{theorem}
\begin{remark}
In the very special case $K_1=\mathbb{Q}_q(\zeta_p)$ one recovers the restriction that $(I_{K_2},\beta_{K_2})$ must be an \emph{admissible} extended $\rho_{d,p}$-jump set as a very special case of Theorem \ref{constraining j.s. under ext.}, see Theorem \ref{realizable j.s. are realized}.
\end{remark}
In particular Theorem \ref{constraining j.s. under ext.} implies the following fact.\footnote{We take the opportunity here to signal a typo in the way this result was mentioned in \cite{de Boer--Pagano}, where the assumption of Theorem \ref{constraining j.s. under ext.} and the conclusion of Corollary \ref{all of them} were accidentally merged in transcribing the statement. It was stated only with the assumption of Theorem \ref{constraining j.s. under ext.}, but the conclusion mentioned there is about \emph{both} consecutive indexes, which we can guarantee, instead, only under the assumption of Corollary \ref{all of them}.} 
\begin{corollary} \label{all of them}
Let $d$ be a positive integer and $K_2/K_1$ a degree $d$ totally ramified extension. Suppose that for any two consecutive elements $i,j$ in $I_{K_1}$ (that is $(i,j) \cap I_{K_1}=\emptyset$) we have that
$$\beta_{K_1}(i)-\beta_{K_1}(j)>\emph{v}_{\mathbb{Q}_p}(d).
$$
Then
$$\frac{d}{p^{\emph{v}_{\mathbb{Q}_p}(d)}}(I_{K_1}-\{e_{K_1}^{*}\}) \subseteq I_{K_2},
$$
with 
$$\beta_{K_2}(\frac{d}{p^{\emph{v}_{\mathbb{Q}_p}(d)}}i)=\beta_{K_1}(i)+\emph{v}_{\mathbb{Q}_p}(d)
$$
for each $i \in I_{K_1}-\{e_{K_1}^{*} \}$. 
\end{corollary}

\end{document}